\documentclass[reqno]{amsart}

\makeatletter
\let\origsection=\section \def\section{\@ifstar{\origsection*}{\mysection}} 
\def\mysection{\@startsection{section}{1}\z@{.7\linespacing\@plus\linespacing}{.5\linespacing}{\normalfont\scshape\centering\S}}
\makeatother  

\usepackage[utf8]{inputenc} 
\usepackage{amssymb}
\usepackage{dsfont}
\usepackage{mathrsfs}
\usepackage{graphicx}
\usepackage{latexsym}
\usepackage{stmaryrd}
\usepackage{enumitem}
\usepackage[bookmarks,hyperfootnotes=false]{hyperref}
\usepackage{multirow}
\usepackage{xcolor}
\usepackage{caption}
\colorlet{darkishRed}{red!60!black}
\colorlet{darkishBlue}{blue!60!black}
\colorlet{darkishGreen}{green!50!black}
\colorlet{lightishGreen}{green!70!black}
\hypersetup{
    draft = false,
    bookmarksopen=true,
    colorlinks,
    linkcolor={darkishGreen},
    citecolor={lightishGreen},
    urlcolor={darkishBlue}
}
\usepackage[nameinlink, capitalise, noabbrev]{cleveref}
\crefformat{enumi}{#2#1#3}
\crefformat{equation}{#2(#1)#3}
\usepackage[abbrev]{amsrefs} 
\usepackage{doi}

\renewcommand{\PrintDOI}[1]{\doi{#1}}

\let\setminus=\smallsetminus
\usepackage{tikz}
\usepackage{tikz-cd}
\usetikzlibrary{calc,through,intersections,arrows, trees, positioning, decorations.pathmorphing, cd}
\usepackage{relsize}
\usepackage{comment}
\usepackage{svg}
\usepackage{mathtools}
\usepackage{nccmath}
\usepackage{pifont}
\usepackage{comment}

\usepackage{subcaption} 

\usepackage{geometry}
\usepackage[utf8]{inputenc}
\usepackage[T1]{fontenc}
\usepackage{lmodern}
\usepackage[babel]{microtype}
\usepackage[english]{babel}
\usepackage{relsize}

\let\setminus=\smallsetminus
\renewcommand{\leq}{\leqslant}
\renewcommand{\geq}{\geqslant}
\renewcommand{\ge}{\geq}
\renewcommand{\le}{\leq}

\let\rho=\varrho
\let\phi=\varphi

\renewcommand{\supset}{\supseteq}

\newcommand{ \N } { \mathbb{N} }

\newcommand{\defn}[1]{{\color{darkishRed}{\emph{#1}}}}

\newcommand{\COMMENT}[1]{{}}

\newcommand{\abs}[1]{\lvert#1\rvert}

\makeatletter

\def\calCommandfactory#1{%
   \expandafter\def\csname c#1\endcsname{\mathcal{#1}}}
\def\frakCommandfactory#1{%
   \expandafter\def\csname frak#1\endcsname{\mathfrak{#1}}}

\newcounter{ctr}
\loop
  \stepcounter{ctr}
  \edef\X{\@Alph\c@ctr}
  \expandafter\calCommandfactory\X
  \expandafter\frakCommandfactory\X
  \edef\Y{\@alph\c@ctr}
  \expandafter\frakCommandfactory\Y
\ifnum\thectr<26
\repeat

\setenumerate{label={\normalfont (\roman*)}}

\usepackage[presets={vec-cev,abc,ABC,cAcBcC}]{letterswitharrows}

\newtheorem{theorem}{Theorem}[section] 
\newtheorem{proposition}[theorem]{Proposition}
\newtheorem{corollary}[theorem]{Corollary}
\newtheorem{lemma}[theorem]{Lemma}

\newtheorem{problem}[theorem]{Problem}
\newtheorem{sublemma}[theorem]{Sublemma}
\newtheorem{mainresult}{Theorem}
\crefname{mainresult}{Theorem}{Theorems}
\newtheorem{maincorollary}[mainresult]{Corollary}
\crefname{maincorollary}{Corollary}{Corollaries}

\newtheorem{innercustomthm}{}

\newenvironment{customthm}[1]
  {\renewcommand\theinnercustomthm{#1}\innercustomthm}
  {\endinnercustomthm}

\theoremstyle{definition}
\newtheorem{example}[theorem]{Example}

\newtheorem{setting}[theorem]{Setting}

\theoremstyle{remark}

\crefname{claim}{Claim}{Claims}
\usepackage{etoolbox}
\AtEndEnvironment{proof}{\setcounter{claim}{0}}

\newenvironment{claimproof}{\noindent\textit{Proof.}}{\hfill\ensuremath{\blacksquare}\medskip}

\newtheorem*{claim*}{Claim}

\newcommand{\td}{tree-decom\-pos\-ition}

\usepackage{etoolbox}

\newbool{pdfBool}
\booltrue{pdfBool} 

\newcommand{\pdfOrNot}[2]{\ifbool{pdfBool}{{#1}}{{#2}}}

\usepackage{svg}

\newbool{arXiv}
\booltrue{arXiv} 

\newcommand{\arXivOrNot}[2]{\ifbool{arXiv}{{#1}}{{#2}}}

\title{On vertex sets inducing tangles}
\author[S.\ Albrechtsen \and H.\ von Bergen \and R.\ W.\ Jacobs \and P.\ Knappe]{Sandra Albrechtsen, Hanno von Bergen, Raphael W. Jacobs, Paul Knappe}
\address{University of Hamburg, Department of Mathematics, Bundesstraße 55 (Geomatikum), 20146 Hamburg, Germany}
\email{\{sandra.albrechtsen,hanno.von.bergen,raphael.jacobs, paul.knappe\}@uni-hamburg.de}
\author[P. Wollan]{Paul Wollan}
\address{University of Rome, “La Sapienza”, Department of Computer Science, Via Salaria 113, 00198 Rome, Italy}
\email{wollan@di.uniroma1.it}


\begin{document}

\begin{abstract}
    Diestel, Hundertmark and Lemanczyk asked whether every $k$-tangle in a graph is induced by a set of vertices by majority vote.
    We reduce their question to graphs whose size is bounded by a function in~$k$.
    Additionally, we show that if for any fixed $k$ this problem has a positive answer, then every~$k$-tangle is induced by a vertex set whose size is bounded in $k$.
    More generally, we prove for all~$k$ that every $k$-tangle in a graph $G$ is induced by a weight function $V(G) \to \N$ whose total weight is bounded in $k$.
    As the key step of our proofs, we show that any given $k$-tangle in a graph $G$ is the lift of a $k$-tangle in some topological minor of $G$ whose size is bounded in $k$.
\end{abstract}

\maketitle

\section{Introduction}

\subsection{Vertex sets inducing tangles}
Tangles are an abstract notion of `clusters' in graphs that originates in the theory of graph minors developed by Robertson and Seymour~\cite{GM}.
They allow for a unified treatment of various concrete highly cohesive substructures in graphs, such as large clique or grid minors.
Tangles describe these clusters in a graph indirectly.
Instead of describing what the cluster is composed of, they describe its position, in that they orient the low-order separations of the graph towards it.
Intuitively, a concrete cluster orients all low-order separations by majority vote, that is the cluster orients such a separation $\{A,B\}$ towards its side, $A$ or $B$, which contains most of the cluster. 
Such a side exists; otherwise, the cluster would be separated by few vertices, which contradicts its high cohesion.

The orientations of the low-order separations induced by concrete clusters are `consistent', in that they all point to the cluster.
Robertson and Seymour's key innovation was to distil from this an abstract notion of `consistency' which leads to the notion of tangle \cite{GMX}: 
Formally, a~\defn{$k$-tangle}~$\tau$ in a graph~$G$ is an orientation of the separations of~$G$ of order~$<k$ such that there do not exist three separations $(A_i, B_i) \in \tau$ such that the union of the small sides $G[A_i]$ covers the whole graph.
We refer to~$k$ as the \defn{order} of~$\tau$, and we will denote a tangle of unspecified order as simply a~\defn{tangle in~$G$}.

While every concrete cluster induces a tangle by majority vote, is the converse also true in that all tangles stem from concrete clusters in this way?
Without a precise definition of concrete cluster, it seems difficult to answer this question.
However, the question remains interesting if we consider arbitrary vertex sets instead of concrete clusters:
Is every tangle at least induced by the majority vote of some set of vertices?
Diestel, Hundertmark and Lemanczyk \cite{ProfilesNew}*{Section~7} formalised this problem as follows.
A set~$X$ of vertices of a graph $G$ \defn{induces} a tangle~$\tau$ in~$G$ if for every separation~$(A, B) \in \tau$ we have~$|X \cap A| < |X \cap B|$.
For example, the vertex set of a complete subgraph induces a tangle in this way.
\begin{problem} \cite{ProfilesNew} \label{conj:Decider}
    Is every tangle in a graph $G$ induced by some set~$X \subseteq V(G)$?
\end{problem}

\noindent In what follows we will often consider \cref{conj:Decider} for all tangles of some fixed order $k \in \N$ and then say for short: \cref{conj:Decider} for $k$. 

We remark that \cref{conj:Decider} is already answered in the affirmative for $k \leq 3$: 
Such sets $X$ inducing $k$-tangles exist for $k \leq 2$ due to the well-known correspondence of these tangles to components and blocks, respectively (cf.\ \cite{GMX}*{(2.6)}).
For $k = 3$, Elbracht \cite{ElbrachtMSc}*{Theorem~1.2} proved \cref{conj:Decider} directly.
Independently, Grohe \cite{char3tangles}*{Theorem~4.8} proved a direct correspondence between the $3$-tangles and the `proper triconnected components' of a graph, which yields another proof of \cref{conj:Decider} for $k=3$.
Similarly, it should be possible to derive \cref{conj:Decider} for $k = 4$ from the recent characterisation of $4$-tangles in terms of `internal $4$-connectedness' by Carmesin and Kurkofka \cite{characterising4tangles}*{Theorem~1}.
Besides these results for small~$k$, Diestel, Elbracht and Jacobs \cite{Focus}*{Theorem~12} showed that \cref{conj:Decider} is true for every $k$-tangle $\tau$ in a graph $G$ which \defn{extends} to a $2k$-tangle $\tau'$ in $G$, that is, $\tau \subseteq \tau'$.

For general $k$, Elbracht, Kneip and Teegen \cite{weighted_deciders_AIC}*{Theorem~2} made substantial progress towards~\cref{conj:Decider} by proving a relaxed weighted version.
A~\defn{weight function}
$w\colon V(G) \to \N$
on the vertex set~$V(G)$ of a graph~$G$ \defn{induces} a tangle~$\tau$ in~$G$ if~$w(A) < w(B)$ for every~$(A, B) \in \tau$.
Note that a set $X \subseteq V(G)$ induces a tangle~$\tau$ if and only if its indicator function $\mathds{1}_X$ on $V(G)$ induces $\tau$.

\begin{theorem} \label{thm:TanglesDecidedWeightedVertexSets}{\cite{weighted_deciders_AIC}} \label{thm:WeightFctsExist}
    Every tangle in a graph~$G$ is induced by some weight function on $V(G)$.
\end{theorem}

\noindent 
Contrary to their positive result, Elbracht, Kneip and Teegen~\cite{weighted_deciders_AIC}*{Theorem~10} explicitly construct an example which shows that not only \cref{conj:Decider}, but also \cref{thm:TanglesDecidedWeightedVertexSets} fails for tangles in general discrete contexts, such as matroids or data sets (see e.g\,\cites{FiniteSplinters,TangleTreeGraphsMatroids,AbstractSepSys,ProfilesNew}). 
However, no such example is known for tangles in graphs. Thus, \cref{conj:Decider} is open for all $k \geq 4$.

\subsection{Our contributions to \texorpdfstring{\cref{conj:Decider}}{Problem 1.1}}

In this paper, we reduce \cref{conj:Decider} for every $k$ to graphs whose size is bounded by a function in~$k$:

\begin{mainresult}\label{mainresult:reduction}
    For every integer $k \geq 1$, there exists $M = M(k) \in O(3^{k^{k^5}})$ such that for every $k$-tangle $\tau$ in a graph $G$, there exists a $k$-tangle $\tau'$ in a connected topological minor $G'$ of $G$ with fewer than $M$ edges such that
    if a weight function $w'$ on $V(G')$ induces the tangle $\tau'$, then the weight function $w$ on $V(G)$ which extends $w'$ by zero\footnote{Given two sets $X' \subseteq X$, a function $w \colon X \to \N$ \defn{extends} a function $w' \colon X' \to \N$ \defn{by zero} if $w$ restricted to $X'$ is $w'$ and $w$ restricted to ${X \setminus X'}$ is $0$.} induces the tangle $\tau$.
    In particular, a set of vertices which induces $\tau'$ also induces~$\tau$.
\end{mainresult}

As an immediate corollary of \cref{mainresult:reduction}, we obtain that one may verify the validity of~\cref{conj:Decider} for any fixed $k$ computationally in explicitly bounded, though impractically long, time by checking \cref{conj:Decider} for every $k$-tangle in every connected graph with fewer than $M$ edges.
\begin{maincorollary}
\label{main:MinimalCterex}
    For $k \geq 1$, there exists $M = M(k) \in O(3^{k^{k^5}})$ such that 
    \cref{conj:Decider} holds for $k$ if it holds for all $k$-tangles in connected graphs $G$ with fewer than $M$ edges.
\end{maincorollary} 

Our second corollary of \cref{mainresult:reduction} asserts that whenever \cref{conj:Decider} is true for some fixed $k$, then every $k$-tangle is induced by a set of vertices of size bounded in~$k$.
This extends the known fact that one may choose the set $X$ in \cref{conj:Decider} that induces a given $k$-tangle with $k \leq 2$ to be of size at most $k$; this follows from the characterisation of $1$- and $2$-tangles by components and blocks, respectively.
More generally, we prove that for every $k$-tangle $\tau$ in a graph $G$ one may choose a weight function that induces $\tau$, which exists by \cref{thm:TanglesDecidedWeightedVertexSets}, in such way that its \defn{total weight} $w(V(G)) = \sum_{v \in V(G)} f(v)$, which it distributes on~$V(G)$, is bounded in~$k$.

\begin{maincorollary} \label{main:BoundTotalWeight}
    For every integer $k \geq 1$, there exists $K = K(k)$ such that for every $k$-tangle $\tau$ in a graph~$G$ there exists a weight function $V(G) \to \N$ which induces~$\tau$ and whose total weight $w(V(G))$ is bounded by $K$.
    In particular, the support of $w$ has size $\leq K$.

    Moreover, if~\cref{conj:Decider} holds for~$k$, then every~$k$-tangle in a graph is induced by a set of at most~$M(k)$ vertices, where $M(k)$ is given by \cref{mainresult:reduction}.
\end{maincorollary}

\noindent We derive \cref{main:BoundTotalWeight} from \cref{mainresult:reduction} by first fixing a weight function for every $k$-tangle in a connected graph with fewer than $M$ edges (such weight functions exist by~\cref{thm:TanglesDecidedWeightedVertexSets}) and then taking $K$ as the maximum of the total weights of these finitely many fixed weight functions.
The moreover-part follows immediately from \cref{mainresult:reduction} by fixing each such weight function to be an indicator function of some inducing set given by the assumed positive answer to~\cref{conj:Decider} (see \cref{sec:Deciders} for details).

\subsection{An inductive proof method for tangles}\label{subsecintro:inductiveproofmethod}

Our proof of \cref{mainresult:reduction} is based on the following theorem which allows inductive proofs for statements about tangles in graphs.
We expect that this inductive proof method will be of independent interest.

\begin{mainresult}\label{thm:InductiveProofMethod}
    For every integer $k \geq 1$ there is some $M(k) \in O(3^{k^{k^5}})$ such that the following holds:
    Let $\tau$ be a $k$-tangle in a graph $G$. 
    Then there exists a sequence $G_0, \dots, G_m$ of graphs and $k$-tangles $\tau_i$ in $G_i$ for every $i \in \{0,\dots,m\}$ such that
    \begin{itemize}
        \item $G_0 = G$, $\tau_0 = \tau$;
        \item $G_i$ is obtained from $G_{i-1}$ by deleting an edge, suppressing a vertex, or taking a proper component;
        \item the $k$-tangle $\tau_{i-1}$ in $G_{i-1}$ survives as the $k$-tangle $\tau_i$ in $G_i$ for every $i \in [m]$;
        \item $G_m$ is connected and has fewer than $M(k)$ edges.
    \end{itemize}
\end{mainresult}

\noindent Before we proceed to explain the crucial term `survive' in this theorem, we remark that it is necessary to allow the suppression of a vertex in \cref{thm:InductiveProofMethod}. Indeed, there are connected graphs $G$ with arbitrarily many edges which have a $k$-tangle such that every graph obtained from $G$ by deleting an edge has no $k$-tangles at all (\cref{ex:noktangleafterdeletion}). 

What does it mean that the $k$-tangle $\tau$ `survives' as a $k$-tangle $\tau'$ in $G'$ obtained from $G$ by deleting an edge, suppressing a vertex, or taking a proper component?
Let us here consider the first case that $G' = G-e$ is obtained from $G$ by deleting an edge $e$ of $G$.
Then every separation of $G$ is also a separation of its subgraph $G'$.
But in general $G'$ admits more separations than $G$; namely those separations~$\{A,B\}$ of~$G'$ which have an endpoint of $e$ in each of the two sets $A \setminus B$ and $B \setminus A$.
A $k$-tangle $\tau$ in~$G$ \defn{extends} to a $k$-tangle $\tau'$ in~$G'$ if $\tau \subseteq \tau'$. 
Then we also say that the $k$-tangle $\tau$ \defn{survives} as the $k$-tangle~$\tau'$ in $G'$. 
We remark that such an extension~$\tau'$ of~$\tau$ may or may not exist in $G'$; if it exists, it need not be unique.
If such a~$\tau'$ exists and is unique, then we also say that~$\tau$ \defn{induces}~$\tau'$.

For the cases that $G'$ is a component of $G$ or obtained from $G$ by suppressing a vertex of $G$, we refer the reader to \cref{subsec:suppressingavertex} for the details.
For readers familiar with the fact that a tangle of order $k \geq 3$ in a minor of a graph $G$ `lifts' to a $k$-tangle in $G$ (cf. \cite{GMX}*{(6.1)} and \cite{characterising4tangles}*{Lemma~2.1}), we remark that `surviving' is the reverse notion.

\subsection{Overview of the proof of \texorpdfstring{\cref{thm:InductiveProofMethod}}{Theorem 4}}

It suffices to describe one step of the construction of the above sequence, that is to find $G_i$ and the $k$-tangle $\tau_i$ in it given $\tau_{i-1}$ and $G_{i-1}$.
We reduce the argument to several cases.
Let $\tau$ be a $k$-tangle in a graph $G$.
The following is an immediate consequence of the well-known correspondence of $1$- and $2$-tangles in~$G$ to the components and blocks of~$G$, respectively:
\begin{enumerate}
    \item \label{item:k1} if $k =1$, then $\tau$ extends to some $k$-tangle in $G-e$ for every edge $e$ of $G$ (\cref{lem:Smallk1}), and
    \item \label{item:k2} if $k=2$, then $\tau$ extends to a $k$-tangle in $G-e$ for some edge $e$ of $G$ (\cref{lem:Smallk2}).
\end{enumerate}

\noindent A rather simple analysis will yield that
\begin{enumerate}
\setcounter{enumi}{2}
    \item \label{item:component} if $G$ is disconnected, then $\tau$ induces to a $k$-tangle in a unique component of $G$ (\cref{prop:CarryOverToComponent}),
    \item \label{item:deg1} if $k \geq 3$ and $e$ is the unique edge incident to a vertex of degree $1$, then $\tau$ induces a $k$-tangle in $G-e$ (\cref{lemma:vertexofdegree1}), and
    \item \label{item:deg2} if $k \geq 3$, then $\tau$ induces a $k$-tangle in every graph obtained from $G$ by suppressing any vertex of degree $2$  (\cref{lem:VertexOfDegreeTwo}).
\end{enumerate}
\noindent We remark that in each of the above \cref{item:k1} to \cref{item:deg2}, the tangle $\tau$ survives as a $k$-tangle in a strictly smaller graph (\cref{sec:SmallSpecialCases}). 
It remains to consider the case that $\tau$ is a tangle of order $k \geq 3$ in a connected graph with minimum degree~$\geq 3$.
Recall that a suppressed vertex always has degree $2$.
So as soon as we restrict to connected graphs of minimum degree at least~$3$, the statement requires that we find an edge~$e$ to delete. 

The proof hinges on an argument that this will always be possible: by carefully picking the edge $e$ of $G$, we may always extend $\tau$ in $G'= G-e$ as long as $G$ is sufficiently large by a function in the tangle's order $k$.

\begin{mainresult} \label{main:SplitterTheorem}
    For every integer $k \geq 3$, there is some $M = M(k) \in O(3^{k^{k^5}})$ such that the following holds:
    
    For every $k$-tangle in a connected graph $G$ with minimum degree $\geq 3$ and at least $M$ edges, there is an edge $e$ of $G$ such that $\tau$ extends to a $k$-tangle in $G - e$.
\end{mainresult}

Together \cref{item:k1} to \cref{item:deg2} along with \cref{main:SplitterTheorem} complete the proof, since in any given case one of them will ensure that the $k$-tangle $\tau$ in a sufficiently large graph survives in some smaller graph.

\subsection{Proof sketch of \texorpdfstring{\cref{main:SplitterTheorem}}{Theorem 5}}

Our task is to find a suitable edge $e$ of $G$ such that $\tau$ extends to some $k$-tangle $\tau'$ in $G' = G -e$.
For this, we consider two cases.
First, we assume that the graph $G$ contains a tangle $\tilde \tau$ of order $> k$ (\cref{sec:ExistenceOfHighOrderTangle}).
If $\tau \subseteq \tilde \tau$, then we will observe that for every edge~$e$ of~$G$ the $k$-tangle~$\tau$ extends to the $k$-tangle~$\tau'$ in~$G'$ which is essentially the restriction of $\tilde \tau$ to the separations of order $< k$ of~$G'$ (\cref{lem:TangleExtends}).
Else we may find an edge $e$ far away from $\tau$ and close to $\tilde \tau$ such that the high order of $\tilde \tau$ enables us to define a suitable extension~$\tau'$ of~$\tau$ in $G' = G - e$ (\cref{lem:AnotherExtendingTangle}).

Second, we assume that the graph $G$ contains no tangle of order $>k$ (\cref{sec:DeletingTheEdge}).
We remark that this case requires significantly more effort than the first.
In the analysis of this case we aim to decompose $G$ in such a way that we can control the separations which arise from the deletion of an edge $e$ in a suitable location (\cref{sec:RCDecompAndSeps}).
To obtain the desired decomposition, we start with the \td\ obtained from the tangle-tree duality theorem \cite{DiestelBook16noEE}*{Theorem~12.5.1} due to the absence of high-order tangles.
As $G$ is sufficiently large, the decomposition tree contains a very long path whose structure we may regularise to obtain a \defn{rainbow-cloud decomposition} of $G$ (\cref{thm:ExistenceOfRCDecomp}).

\begin{figure}[ht]
    \centering
    \pdfOrNot{\includegraphics[page=2]{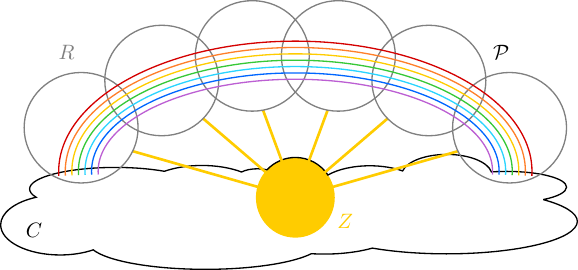}}{\includesvg[width=0.6\columnwidth]{svg/RCdecompositionschematic.svg}}
    \caption{Schematic drawing of a rainbow-cloud decomposition}
    \label{fig:schematicrcdecomp}
\end{figure}

Roughly speaking (see also \cref{fig:schematicrcdecomp}), a rainbow-cloud decomposition of $G$ consists of a long \defn{rainbow}, an induced subgraph $R$ of $G$ with a very regular and linear connectivity structure, which has many connections to the \defn{sun} $Z \subseteq V(G)$ and the remaining graph is gathered in a \defn{cloud}, another induced subgraph $C$ with $G = G[Z \cup V(R)] \cup C$ (see \cref{sec:ExistenceRCDecomp} for the precise definition).
We show that if we choose the edge $e$ deep inside the rainbow $R$ and far away from $\tau$, then the connectivity structure of the long rainbow $R$ ensures that the new separations which arise from the deletion of $e$ may be oriented in such a way that $\tau$ extends to a $k$-tangle $\tau'$ in the graph $G - e$ (\cref{thm:AtLeastOneRainbowPath}).

\subsection{How this paper is organised}
We recall the relevant terminology concerning tangles in~\cref{sec:Prelims}.
We then introduce in \cref{subsec:CarryOver} the notion of `survives' and prove \cref{thm:InductiveProofMethod} in several  special cases:~the case~$k \leq 2$, the case $\delta(G) \leq 2$,  and when the graph $G$ is disconnected.
In \cref{sec:ExistenceOfHighOrderTangle} we then prove \cref{main:SplitterTheorem}, the remaining case of~\cref{thm:InductiveProofMethod} when~$G$ contains a tangle of order~$k+1$.
In~\cref{sec:ExistenceRCDecomp}, we introduce the definition of `rainbow-cloud decompositions' and prove that such a decomposition exists in the absence of~$(k+1)$-tangles in~$G$, assuming that $G$ is sufficiently large.
Building on several lemmas which we prove in~\cref{sec:RCDecompAndSeps} about the interaction of rainbow-cloud decompositions and separations, we complete the proof of~\cref{main:SplitterTheorem} in~\cref{sec:DeletingTheEdge}.
In \cref{sec:Deciders} we collect all the individual cases to prove \cref{thm:InductiveProofMethod}, derive \cref{mainresult:reduction}, and deduce~\cref{main:MinimalCterex} and~\cref{main:BoundTotalWeight}.

\section{Preliminaries}\label{sec:Prelims}

In this section, we recall the relevant terms and definitions concerning tangles in graphs.
For general graph-theoretic concepts and notation, we follow~\cite{DiestelBook16noEE}.
In particular, we denote the set $\{1, \dots, n\}$ with $n \in \N$ by $[n]$.
A path is \defn{trivial} if it has length~$0$.
We remark that all graphs in this paper are finite.

\subsection{Separations of sets}\label{subsec:Separations}

While we will only work with separations of graphs in this paper, we formally introduce the notion of a separation on a set to transfer separations from one graph to another.

Given an arbitrary set $V$, an \defn{(unoriented) separation} of $V$ is an unordered pair $\{A,B\}$ of subsets $A, B$ of~$V$ such that $A \cup B = V$.
The sets~$A$ and~$B$ are the~\defn{sides} of the separation~$\{A, B\}$, and~$A \setminus B$ and~$B \setminus A$ are its~\defn{strict sides}.
The separation~$\{A,B\}$ \defn{separates} every two sets $X \subseteq A$ and $Y \subseteq B$.
The \defn{order} of the separation, denoted~$|A, B|$, is the cardinality of its \defn{separator} $A \cap B$. 
Every separation~$\{A, B\}$ of~$V$ has two \defn{orientations}, $(A,B)$ and $(B,A)$.
These orientations of~$\{A, B\}$ are \defn{oriented separations} of~$V$, that is, ordered pairs of subsets of~$V$ whose union equals~$V$.
We refer to $A$ as the \defn{small} side of an oriented separation~$(A,B)$ and to $B$ as its \defn{big} side.
Given a set~$V$, we write~$U(V)$ for the set of all (unoriented) separations of $V$, and let~$S_k(V) := \{\{A,B\} \in U(V) : \abs{A \cap B} < k\}$.
The set of all oriented separations of~$V$ is denoted by~$\vU(V)$ and the set of all oriented separations of $V$ of order less than~$k$ by~$\vS_k(V)$.
We will often also refer to `oriented separations' simply as `separations' if the meaning is clear from the context.
Similarly, we will transfer definitions for separations verbatim to oriented separations and vice-versa if the transfer is unambiguous.

The oriented separations of~$V$ have a natural partial order:
\begin{equation*}
    (A,B) \leq (C,D) :\Leftrightarrow A \subseteq C \text{ and } B \supseteq D.
\end{equation*}
With this, the set $\vU(V)$ is a lattice with \defn{infimum} $(A,B) \wedge (C,D) := (A \cap C, B \cup D)$ and \defn{supremum} $(A,B) \vee (C,D) := (A \cup C, B \cap D)$.
Moreover, $\vU(V)$ is \defn{distributive}, that is, $x \wedge (y \vee z) = (x \wedge y) \vee (x \wedge z)$ for all $x,y,z \in \vU(V)$.
It is well-known \cite{birkhoff}*{Chapter~IX~Corollary~1} that a lattice is distributive if and only if every element of the lattice is uniquely determined by its infimum and supremum with any given element.

If two unoriented separations of a set~$V$ have orientations which are comparable, 
then they are called~\defn{nested}; if they are not nested, then the two separations \defn{cross}.

By double counting, we obtain that the order of the supremum and infimum of~$\{A, B\}$ and~$\{C, D\}$ sum up to
\begin{equation*}
    |(A,B) \wedge (C,D)| + |(A,B) \vee (C,D)| = |A, B| + |C, D|.
\end{equation*}
The inequality `$\leq$' implied by the equality above implies that $|\cdot|$ is a \defn{submodular} function on the lattice~$(\vU(V), \le)$.
In particular, this implies that for two separations~$\{A,B\}, \{C,D\}$ of order~$<k$, at least one of $\{A \cap C, B \cup D\}$ and $\{A \cup C, B \cap D\}$ has also order~$<k$.

\subsection{Tangles and separations of graphs}\label{subsec:Tangles}

Let~$G = (V, E)$ be a graph.
A \defn{separation~$\{A, B\}$ of~$G$} is a separation of its vertex set~$V$ such that no edge of~$G$ joins~$A \setminus B$ and~$B \setminus A$ (equivalently: $G = G[A] \cup G[B]$).
We write~$U(G)$ for the set of all unoriented separations of~$G$ and~$S_k(G) := \{\{A,B\} \in U(G) : \abs{A \cap B} < k\}$; as above, we then define $\vU(G)$ and $\vS_k(G)$ as the respective sets of all oriented separations.
Note that $U(G) \subseteq U(V)$ by definition.
Moreover, $(\vU(G), \le)$ is a sublattice of~$(\vU(V), \le)$,
as it is straight forward to check that given two separations $(A,B)$, $(C,D)$ of $G$, the ordered pairs $(A \cap C, B \cup D)$ and $(A \cup C, B \cap D)$ are again separations of $G$. 

Now let $k \ge 1$ be an integer.
An \defn{orientation} of~$S_k(G)$ is a set~$O \subseteq \vS_k(G)$ that contains precisely one orientation of every separation~$\{A, B\} \in S_k(G)$.%
\COMMENT{$(V, V)$ hat formal nur eine Orientierung, weil "beide Orientierungen" identisch sind.}
If an orientation $O$ of $S_k(G)$ contains $(A,B)$, we say that~$O$ \defn{orients} $\{A,B\}$ \defn{as} $(A,B)$.
A \defn{$k$-tangle} or \defn{tangle of order $k$} in~$G$ is an orientation~$\tau$ of~$S_k(G)$ which contains no element of
\begin{equation*}
    \cT := \cT(G) := \{\{(A_1, B_1), (A_2, B_2), (A_3, B_3)\} \subseteq \vU(G) : G[A_1] \cup G[A_2] \cup G[A_3] = G\}.
\end{equation*}
as a subset.
We say the $k$-tangle \defn{avoids}~$\cT$.
The elements of~$\cT$ are called \defn{forbidden triples}.  A forbidden triple of~$\cT$ such that two of the~$(A_i, B_i)$ are equal is a \defn{forbidden tuple}.
Tangles of unspecified order are referred to as \defn{tangles in~$G$}.
A~$k$-tangle $\tau$ in~$G$ \defn{extends} to a $k^*$-tangle $\tau^*$ in $G$ for $k^* \geq k$ if $\tau \subseteq \tau^*$.
We say that a separation $\{A,B\}$ of $G$ \defn{distinguishes} two tangles $\tau, \tau'$ in $G$ if $\tau$ and~$\tau'$ contain distinct orientations of~$\{A,B\}$.

It is well-known \cite{GMX}*{(2.6)} that tangles of order $1$ and $2$ are in a~one-to-one correspondence to the components and~blocks\footnote{Recall that a \defn{block} of a graph is a $\subseteq$-maximal connected subgraph $H$ such that $H-v$ is connected for every vertex $v \in H$. In particular, as the empty graph is not considered to be connected, every block contains some edge.}, respectively:
\begin{proposition} \label{thm:TanglesOfSmallOrder}
    For a tangle~$\tau$ in a graph~$G$, let~$X_\tau 
    := \bigcap_{(A, B) \in \tau} B$.
    Then the map~$\tau \mapsto G[X_\tau]$ is a bijection between the tangles of order~$1$ and the components of~$G$ as well as between the tangles of order~$2$ and the set of all blocks of~$G$. \qed
\end{proposition}

From its definition, tangles have several properties that we will frequently make use of.
First, every tangle~$\tau$ is \defn{consistent}: there are no two separations~$(A,B), (C,D) \in \tau$ such that $(D,C) \leq(A,B)$.
Secondly, if~$\tau$ has order~$k$, then~$(X, V(G)) \in \tau$ for all~$X \subseteq V(G)$ of size~$<k$; this property is sometimes called the~\defn{regularity} of~$\tau$.
Thirdly, $\tau$ has the \defn{profile property}: if~$(A,B), (C,D) \in \tau$ and their supremum~$(A \cup C, B \cap D)$ again has order~$<k$, then it is also in the tangle $\tau$.

Finally, let us make two observations about the interplay of tangles with the partial order~$\le$ on oriented separations.
First, if an orientation~$\tau$ of~$S_k(G)$ is not a tangle and thus contains a forbidden triple, then $\tau$ also contains a forbidden triple whose elements are $\leq$-maximal in $\tau$.
Secondly, if $\tau$ is a tangle in $G$, then every $\leq$-maximal separation $(A,B) \in \tau$ satisfies that $G[B \setminus A]$ is connected.

\subsection{Topological minors}

Let $G$ be a graph.
By \defn{suppressing} a degree-$2$ vertex $v$ of $G$, we obtain the graph $G' := G -v + uw$ where $u,w$ are the two neighbours of $v$ in $G$.
As this operation is only defined for vertices of degree $2$, we often just say that the graph $G'$ is obtained from~$G$ by suppressing a vertex.
A \defn{topological minor} $G'$ of $G$ is a graph obtained from $G$ by a sequence of deleting edges, suppressing vertices and deleting vertices.
Equivalently, $G'$ is obtained from~$G$ by a sequence of suppressing vertices in a subgraph of $G$.
We remark that if a topological minor $G'$ of $G$ is connected, we may obtain $G'$ also by a sequence of deleting edges, suppressing vertices and passing to proper components of the current graph.

\subsection{Lifts of tangles}

It is well-known that tangles in minors lift to tangles in the host graph (see \cite{GMX}*{(6.1)} or \cite{characterising4tangles}*{Lemma~2.1}). 
Here, we introduce the notion only for topological minors.

Let $\tau'$ be a $k$-tangle in a subgraph $G'$ of $G$.
Then the \defn{lift} $\tau$ of $\tau'$ to $G$ is the set consisting of precisely those $(A,B) \in \vS_k(G)$ whose restriction $(A \cap V(G'), B \cap V(G')) $ is in $\tau'$.
It is immediate that $\tau$ is a $k$-tangle in $G$, since $G'$ is a subgraph of $G$ and thus every forbidden triple in $\tau$ restricts to a forbidden triple in $\tau'$ by definition.

Now let $\tau'$ be a tangle of order $k \geq 3$ in a graph $G'$ obtained from $G$ by suppressing a vertex $v$.
Denote the two neighbours of $v$ in $G$ by $u_1,u_2$.
We define the \defn{lift} $\tau$ of $\tau'$ to $G$, as follows.

Let $(A,B) \in \vS_k(G)$ be arbitrary. We denote by $(A',B')$ the pair $(A \setminus \{v\}, B \setminus \{v\})$ and by $C'_i$ the set $C' \cup \{u_i\}$ for $C' =  A',B'$ and $i = 1,2$.
\begin{enumerate}
	\item If $u_1, u_2 \in A$ or $u_1, u_2 \in B$, then $(A',B')$ is a separation of $G'$, and we include $(A,B)$ in $\tau$ if $(A',B') \in \tau'$.
	\item \label{item:case2} If, for $\{i,j\} = \{1,2\}$, $u_i \in A \setminus B$ and $u_j \in B \setminus A$, then $(A',B'_i)$ and $(A'_j,B')$ are separations of $G'$, and we include $(A,B)$ in $\tau$ if at least one of $(A',B'_i)$ and $(A'_j,B')$ is in $\tau'$.
\end{enumerate}

\noindent Note that in \cref{item:case2} either both $(A', B'_i), (A'_j, B')$ are in $\tau'$ or both their inverses are in $\tau'$, which implies that~$\tau$ does not contain both orientations of a separation, and is hence an orientation of $S_k(G)$.
Indeed, we otherwise have $(A', B'_i), (B', A'_j) \in \tau$, as we cannot have~$(B'_i, A), (A'_j, B') \in \tau$ because of~$(A', B'_i) \leq (A'_j, B')$ and the consistency of the tangle $\tau'$.
But we have $(G'[A'] \cup G'[B']) + u_1u_2 = G'$, and thus $(A',B'_i)$, $(B', A'_j)$ and $(\{u_1,u_2\}, V(G'))$ would form a forbidden triple in $\tau'$, since $(\{u_1,u_2\}, V(G')) \in \tau$ due to the regularity of the $k$-tangle $\tau$ with $k \geq 3$.
From this it also follows that a forbidden triple in $\tau$ would correspond to a forbidden triple in $\tau'$, by replacing each $(A,B)$ in the forbidden triple by $(A'_1, B')$ or $(A'_2, B')$ if necessary.
Thus, the lift $\tau$ of the tangle $\tau'$ of~$G'$ of order $k \geq 3$ to $G$ as defined above is indeed a $k$-tangle in $G$.

Given a $k$-tangle $\tau'$ in a topological minor $G'$ of a graph $G$, we now obtain the \defn{lift} $\tau$ of $\tau'$ to $G$ by iteratively considering the lifts along the sequence of edge deletions, vertex suppressions and vertex deletions from which $G'$ originated from $G$.

\section{Definition of `survive' and special cases of \texorpdfstring{\cref{thm:InductiveProofMethod}}{Theorem 4}} \label{subsec:CarryOver} \label{sec:SmallSpecialCases}

In this section we define what it means for a tangle $\tau$ in a graph $G$ to `survive' in a topological minor $G'$ of $G$. 
This notion can be seen as a converse to lifting a tangle.
But let us emphasise that while a tangle of order at least $3$ in $G'$ always lifts to $G$, a tangle in $G$ need not survive in $G'$.
We first define `survive' for subgraphs $G'$ of $G$ in \cref{subsec:surviveinsubgraphs} and then for graphs $G'$ obtained from $G$ by suppressing a single vertex of degree $2$ in \cref{subsec:suppressingavertex}. 
Alongside these definitions, we prove several lemmas which all deal with special cases of \cref{thm:InductiveProofMethod}.
Finally, we provide \cref{ex:noktangleafterdeletion} which demonstrates that suppressing vertices of degree~$2$ needs to be allowed in~\cref{thm:InductiveProofMethod}.

\subsection{Extending and inducing tangles in subgraphs}\label{subsec:surviveinsubgraphs}

Recall that for every subgraph~$G'$ of a graph~$G$ on the same vertex set, a separation of~$G$ is also a separation of~$G'$, i.e.\ $U(G) \subseteq U(G')$ and~$S_k(G) \subseteq S_k(G')$ for all integers~$k \ge 1$.
We say that an orientation~$\tau$ of~$S_k(G)$ for some integer~$k \ge 1$ \defn{extends} to an orientation~$\tau'$ of~$S_k(G')$ if $\tau'$ orients every separation in~$S_k(G)$ in the same way as~$\tau$ (equivalently: $\tau \subseteq \tau'$).
In this situation, we also sometimes say that the $k$-tangle $\tau$ in $G$ \defn{survives} as the $k$-tangle $\tau'$ in $G'$.
If $\tau$ is a $k$-tangle in $G$ and there exists precisely one $k$-tangle $\tau'$ in $G'$ to which $\tau$ extends, then we also say that~$\tau$ \defn{induces} the $k$-tangle $\tau'$.
In this paper, $G'$ will often either be a component of $G$ or arise from~$G$ by the deletion of a single edge~$e \in G$.
We remark that if a $k$-tangle~$\tau$ in~$G$ extends to a $k$-tangle~$\tau'$ in a subgraph~$G'$ of~$G$, then~$\tau$ is obviously the lift of~$\tau'$ to~$G$.

\begin{lemma} \label{lem:Smallk1}
    Let $\tau$ be a $1$-tangle in a graph $G$.
    Then $\tau$ extends to a $1$-tangle $\tau'$ in $G - e$ for every edge~$e \in G$.
\end{lemma}
\begin{proof}
    Let $e$ be an arbitrary edge of~$G$, and let~$G' := G - e$.
    By~\cref{thm:TanglesOfSmallOrder}, the set~$X_\tau := \bigcap_{(A,B) \in \tau} B$ is a component of~$G$ corresponding to the~$1$-tangle~$\tau$.
    Let~$X'$ be a component of~$G'$ whose vertex set is contained in~$X_\tau$.
    By~\cref{thm:TanglesOfSmallOrder} the component~$X'$ of $G'$ corresponds to a~$1$-tangle $\tau'$ in $G'$ with $X_{\tau'} = \bigcap_{(A,B) \in \tau'} B = V(X')$.
    Since every vertex of~$G$ lies on precisely one side of every separation of~$G$ of order~$0$, a separation $(A,B)$ of $G$ of order~$0$ is contained in $\tau'$ if and only if $X_{\tau'} \subseteq B$.
    So for every~$(A,B) \in \tau$, we have~$X_{\tau'} \subseteq X_\tau \subseteq B$ and thus~$(A,B) \in \tau'$. 
    Hence, $\tau$ extends to~$\tau'$.
\end{proof}

\noindent We remark that in the proof of \cref{lem:Smallk1} the deletion of $e$ may give rise to (at most) two components $X'$ in $G[X_\tau]$. Thus, $\tau$ does not necessarily extend to a unique $1$-tangle in $G'$.

\begin{lemma} \label{lem:Smallk2}
    Let $\tau$ be a $2$-tangle in a graph $G$ with at least $2$ edges.
    Then $\tau$ extends to a $2$-tangle $\tau'$ in $G - e$ for some edge $e \in G$.
\end{lemma}
\begin{proof}
    By~\cref{thm:TanglesOfSmallOrder}, $G[X_\tau]$ is a block of $G$; in particular, it contains an edge~$f$.
    Let $e \neq f$ be any other edge in $G$, and consider $G' := G - e$.
    Let~$X'$ be the vertex set of the~block of~$G'$ containing~$f$, and note that~$X' \subseteq X_\tau$ by the definition of~$G'$.
    By~\cref{thm:TanglesOfSmallOrder}, the block $X'$ corresponds to a~$2$-tangle~$\tau'$ of~$G'$; in particular, $f \in G'[X'] = G'[X_{\tau'}]$.
    Since every edge in~$G$ lies on precisely one side of every separation of $G$ of order~$\leq 1$, a separation $(A, B)$ of $G$ of order~$\leq 1$ is contained in $\tau'$ if and only if $f \in G[B]$.
    For every~$(A, B) \in \tau$, we have~$f \in G[X_\tau] \subseteq G[B]$ and hence~$(A, B) \in \tau'$.
    Thus, $\tau$ extends to~$\tau'$.
\end{proof}

\noindent Similarly, as in the proof of \cref{lem:Smallk1} also in the proof of \cref{lem:Smallk2} the block $X_\tau$ may split into two (or even more) blocks in $G'$. Thus, $\tau$ does not necessarily extend to a unique $2$-tangle in $G'$.

\begin{lemma} \label{prop:CarryOverToComponent}
    Let~$k \ge 1$ be an integer, and let~$G$ be a graph with a~$k$-tangle~$\tau$.
    Then there exists a (unique) component~$G'$ of~$G$ such that~$\tau$ extends to a~$k$-tangle~$\tau'$ in~$G'$.
    Moreover, $\tau$ induces the $k$-tangle $\tau'$ in $G'$.
\end{lemma}

\begin{proof}
    Let~$\tau_1$ be the~$1$-tangle in~$G$ with~$\tau_1 \subseteq \tau$.
    By~\cref{thm:TanglesOfSmallOrder}, $G[X]$ with~$X := X_{\tau_1} = \bigcap_{(A, B) \in \tau_1} B$ is a component of~$G$.
    We claim that~$G' := G[X]$ is as desired.

    We define an orientation~$\tau'$ of~$S_k(G')$ based on~$\tau$.
    For any separation~$\{A', B'\}$ of order~$<k$ of~$G'$, we consider the separation~$\{A, B\}$ of~$G$ defined by~$A := A' \cup (V(G) \setminus X)$ and~$B := B'$.
    This again has order~$<k$ and is hence oriented by~$\tau$.
    So if~$(A, B) \in \tau$, then we put~$(A', B') \in \tau'$, and if~$(B, A) \in \tau$, we put~$(B', A') \in \tau'$.
    
    We now show that $\tau$ extends to $\tau'$.
    It suffices to check that if we included $(A',B')$ or $(B',A')$ in $\tau$ because of $(A,B)$ or $(B,A)$ in $\tau$, respectively, then there exists no $(C,D) \in \tau$ such that $(C \cap V(X), D \cap V(X)) = (B',A')$ or $= (A',B')$, respectively.
    Suppose for a contradiction that there is such a $(C,D) \in \tau$.
    The definition of $X := X_{\tau_1}$ ensures that the separation $((V(G) \setminus X), X)$ of order $0$ is in $\tau_1 \subseteq \tau$.
    Thus, the forbidden triple 
    \begin{equation*}
        \{(A, B), (C,D)), ((V(G) \setminus X), X)\} \text{ or } \{(B, A), (C,D)), ((V(G) \setminus X), X)\})
    \end{equation*}
    is contained in the tangle~$\tau$, respectively, which is a contradiction.
    Thus, $\tau$ extends to $\tau'$.

    It remains to show that the orientation~$\tau'$ of~$S_k(G')$ is even a tangle in~$G'$.
    Suppose for a contradiction that there exists a forbidden triple~$\{(A'_i, B'_i) : i \in [3]\}$ in~$\tau'$.
    As above, $(A_i, B_i) := (A'_i \cup (V(G) \setminus X), B'_i)$\ is a separation of~$G$ of order~$<k$ and by construction contained in~$\tau$ for all~$i \in [3]$.
    Since~$G'$ is a component of~$G$, it is immediate that~$\{(A_i, B_i) : i \in [3]\}$ forms a forbidden triple in~$\tau$, which is a contradiction.

    We remark that the proof immediately ensures that $G'$ is the unique component such that $\tau$ extends to a $k$-tangle in it, and also $\tau'$ is the unique such $k$-tangle.
\end{proof}

\begin{lemma}\label{lemma:vertexofdegree1}
    Let $\tau$ be a tangle in $G$ of order $k \geq 3$.
    Suppose that $G$ has a vertex $v$ of degree $1$ and let $e$ be the unique edge incident to $v$.
    Then $\tau$ induces a $k$-tangle $\tau'$ in $G-e$.
\end{lemma}

\begin{proof}
    Let $\tau'$ be the subset of $\vS_k(G')$ which contains $(A,B) \in \tau'$ if $(A,B) \in \tau$, $(A \setminus \{v\}, B \cup \{v\}) \in \tau$ or $(A \cup \{v\}, B \setminus \{v\}) \in \tau$.
    The regularity of the tangle $\tau$ of order $\geq 3$ ensures that $(e,V(G)) \in \tau$.
    Thus, $\tau'$ is an orientation of $S_k(G')$, as any violation would form together with $(e,V(G))$ a forbidden triple in $\tau$.

    Suppose for a contradiction that $\{(A_i',B_i') : i \in [3]\}$ is a forbidden triple in $\tau'$.
    Since $G' = \bigcup_{i \in [3]} G'[A_i']$, some $A_j'$ contains the other endvertex $u \neq v$ of $e$.
    Thus, $(A_j, B_j) := (A_j' \cup \{v\}, B_j' \setminus \{v\})$ is a separation of~$G$ and thus in~$\tau$, as~$\tau'$ is an orientation.
    Let $(A_i,B_i) \in \tau$ which witnesses that $(A_i',B_i') \in \tau'$ for $i \in [3] \setminus \{j\}$.
    Now the $(A_i,B_i)$ form a forbidden triple in $\tau$, which is a contradiction.
\end{proof}

\subsection{Inducing tangles in graphs obtained by vertex suppression} \label{subsec:suppressingavertex}

Recall that for a given vertex $v \in V(G)$ of degree $2$, the graph obtained from $G$ by \defn{suppressing the vertex $v$} is $G' := G - v + uw$ where $u,w$ are the two neighbours of $v$ in $G$.
Let $\{A,B\}$ be a separation of $G'$.
If $u,w \in A$, then $\{A \cup \{v\}, B\}$ is a separation of $G$, and analogously if $u, w \in B$, then $\{A, B \cup \{v\}\}$ is a separation of $G$. In particular, they have the same order as $\{A,B\}$.
We remark that at least one of $u,w \in A$ and $u,w \in B$ holds, as $uw$
is an edge of $G'$.

We say that a $k$-tangle~$\tau$ of $S_k(G)$ for some integer $k \geq 1$ \defn{induces} the subset $\tau' \subseteq \vS_k(G')$ consisting of those $(A,B) \in \vS_k(G')$ such that at least one of $(A \cup \{v\}, B)$ and $(A, B \cup \{v\})$ is in $\tau$.
In particular, if $\{A \cup \{v\}, B\}$ is a separation of $G$, then one of its orientations is contained in $\tau$, and thus at least one of $(A,B)$ or $(B,A)$ is in $\tau'$; but $\tau'$ might contain both if also $\{A, B \cup \{v\}\}$ is a separation of $G$.
\cref{lem:VertexOfDegreeTwo} below ensures that if $k \geq 3$, then not only the latter does not happen, but $\tau'$ is even a $k$-tangle.
In this situation, we also sometimes say that the $k$-tangle $\tau$ in $G$ \defn{survives} as the $k$-tangle $\tau'$ in the graph $G'$ obtained from $G$ by suppressing a vertex.
We remark that if a $k$-tangle $\tau$ in~$G$ extends to a $k$-tangle~$\tau'$ in the graph~$G'$ obtained from~$G$ by suppressing a vertex, then~$\tau$ is obviously the lift of~$\tau'$ to $G$.

\begin{lemma}\label{lem:VertexOfDegreeTwo}
    Let $G$ be a graph, and let $G' = G - v + uw$ be the graph obtained by suppressing a vertex $v$ of degree $2$, where its two neighbours in $G$ are $u,w$.
    Then for a given $k$-tangle $\tau$ with $k \geq 3$ in $G$ the set $\tau' \subseteq \vS_k(G')$ induced by $\tau$ is a $k$-tangle in $G'$.
\end{lemma}

\begin{proof}
    We claim that $(A \cup \{v\},B) \in \tau$ if and only if $(A, B \cup \{v\}) \in \tau$, if both $\{A \cup \{v\},B\}$ and $\{A, B \cup \{v\}\}$ are separations of $G$. 
    Suppose for a contradiction that this is not the case.
    The consistency of $\tau$ ensures that $(A,B \cup \{v\}) \in \tau$ and $(B, A\cup \{v\}) \in \tau$.
    Since $k \geq 3$ and every tangle is regular, $(\{u,v,w\}, V(G) \setminus \{v\}) \in \tau$.
    Now $\{(A,B \cup \{v\}), (B, A\cup \{v\}), (\{u,v,w\}, V(G) \setminus \{v\})\}$ forms a forbidden triple in the tangle $\tau$, which is a contradiction.

    The claim above shows that $\tau'$ is an orientation of $S_k(G')$.
    It remains to show that it is a $k$-tangle in~$G'$.
    Suppose that $\{(A'_i,B'_i) : i \in [3] \}$ is a forbidden triple in $\tau'$.
    Then there is $j \in [3]$ with $u,w \in A'_j$, since $uw$ is an edge in $G'$.
    Hence, $\{A'_j \cup \{v\}, B'_j\}$ is a separation of $G$, and thus $(A_j, B_j) := (A'_j \cup \{v\}, B'_j) \in \tau$ by definition of $\tau'$. 
    For every $i \in [3] \setminus \{i\}$, we let $(A_i, B_i) \in \tau$ be the separation which witnesses that $(A'_i,B'_i) \in \tau'$.
    Then $\{(A_i,B_i) : i \in [3] \}$ is a forbidden triple in $\tau$, as $\{(A'_i, B'_i) : i \in [3]\}$ is a forbidden triple in~$G'$ and $u,v,w \in A_j$.
    This is a contradiction to the $\tau$ being a tangle in~$G$.
\end{proof}

\subsection{Suppressing vertices of degree~\texorpdfstring{$2$}{2} in~\texorpdfstring{\cref{thm:InductiveProofMethod}}{Theorem 4}}

We conclude this section with an example that shows that there are connected graphs $G$ with arbitrarily many edges which have a $k$-tangle such that every graph obtained from $G$ by deleting an edge has no $k$-tangles at all. In particular, in \cref{thm:InductiveProofMethod} it is necessary to allow the suppression of a vertex.

\begin{example}\label{ex:noktangleafterdeletion}
    For every $k,M \in \N$  with $k \geq 3$, there exists a connected graph $G$ with at least $M$ edges and which has a $k$-tangle such that, for every edge $e$ of $G$, the graph $G-e$ does not have a $k$-tangle.
\end{example}

\begin{proof}
    Let $G'$ be some connected graph which has a $k$-tangle but which is such that, for every edge $e$ of $G'$, the graph $G'-e$ does not have any $k$-tangles. Such graphs exist: Take any graph $H$ that has a $k$-tangle, and let $H_0 := H \supset H_1 \supset \dots \supset H_n$ be a maximal sequence such that $H_{i+1}$ is obtained from $H_{i}$ by either deleting an edge or taking a proper component of $H_i$, and such that $H_n$ still has a $k$-tangle. 
    Set $G' := H_n$. 
    By the maximal choice of the sequence $(H_i)_{i \in [n]}$ and because of \cref{prop:CarryOverToComponent}, $G'$ is connected and no graph obtained from $G'$ by deleting an edge has a $k$-tangle. 
    
    Since edgeless graphs have no tangles of order $\geq 2$, the graph $G'$ contains an edge $uv$.
    Let $G$ be obtained from $G'$ replacing $uv$ by a $u$--$v$ path of length at least $M+1$. Then $G$ has at least $M$ edges, and it has a $k$-tangle as every tangle of order $k \geq 3$ in $G'$ lifts to a $k$-tangle in $G$.
    But $G-e$ has no $k$-tangle for every edge $e$ of $G$, since any such tangle would induce a $k$-tangle in $G'-e'$ by \cref{lemma:vertexofdegree1,lem:VertexOfDegreeTwo} where $e' := e$ if $e \in E(G)$ or $e' := uv$ otherwise.
\end{proof}

\section{Proof of \texorpdfstring{\cref{main:SplitterTheorem}}{Theorem 5} if \texorpdfstring{$G$}{G} has a higher-order tangle} \label{sec:ExistenceOfHighOrderTangle}

In this section we prove~\cref{main:SplitterTheorem} for graphs which contain a tangle of order~$>k$.

\begin{theorem} \label{thm:NoBigTangle}
    Let $\tau$ be a tangle in $G$ of order $k \geq 2$.
    Suppose further that there exists a $(k+1)$-tangle $\tau^*$ in $G$.
    Then there is an edge $e \in E(G)$ such that $\tau$ extends to some $k$-tangle $\tau'$ in~$G - e$.
\end{theorem}

We distinguish two cases:
First, we show that if $\tau$ itself extends to a $(k+1)$-tangle in $G$, then $\tau$ extends to a $k$-tangle in $G - e$ for every edge $e \in E(G)$ (\cref{lem:TangleExtends}).
Otherwise, there exists a $(k+1)$-tangle $\tilde{\tau}$ in $G$ which does not extend $\tau$.
Here, $\tau$ does not extend to a $k$-tangle in $G - e$ for every edge~$e \in G$, but only for some such edges $e$ and we will need some care to find them (\cref{lem:AnotherExtendingTangle}).

For both cases, we will make use of the following observation.
\begin{lemma}\label{lem:StdStuff-ShiftingAnEdge}
    Let $k \geq 3$ be an integer, and let $G$ be a graph with a $k$-tangle $\tau$. 
    Further, let~$e$ be an edge of~$G$ and let~$\{A, B\}$ be a separation of~$G - e$ with~$e \in E_G(A \setminus B, B \setminus A)$.
    If $\abs{A \cap B} < k-1$, then $(A \cup e, B) \in \tau$ if and only if $(A, B \cup e) \in \tau$.
\end{lemma}
\begin{proof}
    The assumptions immediately imply that~$\{A \cup e, B\}$ and~$\{A, B \cup e\}$ are separations of~$G$ of order~$<k$, and so~$\tau$ orients both of them.
    The consistency of~$\tau$ together with $(A, B \cup e) \leq (A \cup e, B)$ yields the forwards implication.
    The backwards implication follows from the fact that~$\{(A, B \cup e), (B, A \cup e), (e, V(G))\}$ is a forbidden triple but~$(e, V(G)) \in \tau$, since~$\tau$ is a~$k$-tangle with~$k \ge 3$ and~$V(G)$ cannot be the small side of any separation in~$\tau$.
\end{proof}

We start with the case in which $\tau$ extends to a $(k+1)$-tangle in $G$:

\begin{lemma} \label{lem:TangleExtends}
    Let $k \geq 2$ be an integer, and let $G$ be a graph with a $k$-tangle $\tau$.
    If $\tau$ extends to a $(k+1)$-tangle~$\tilde{\tau}$ in $G$, then $\tau$ extends to a $k$-tangle $\tau'$ in $G - e$ for every edge $e \in E(G)$.
\end{lemma}
\begin{proof}
    Consider $G' := G-e$.
    We define an orientation $\tau'$ of $S_k(G')$ as follows:
    If~$\{A, B\} \in S_k(G')$ is also a separation of~$G$, then we put $(A, B) \in \tau'$ if and only if $(A, B) \in \tau$.
    Otherwise, the edge~$e$ has one endvertex in~$A \setminus B$ and the other one in~$B \setminus A$.
    So both $\{A \cup e, B\}$ and $\{A, B \cup e\}$ are separations of $G$ of order~$|A, B| + 1 \le k$, and these two separations are oriented by the $(k + 1)$-tangle $\tilde{\tau}$, and we have $(A \cup e, B) \in \tau$ if and only if $(A, B \cup e) \in \tilde{\tau}$ by~\cref{lem:StdStuff-ShiftingAnEdge}.
    Then we put $(A, B) \in \tau'$ if and only if $(A \cup e, B) \in \tilde{\tau}$ (equivalently: $(A,B \cup e)$).
    The first part of the definition guarantees that $\tau$ extends to $\tau'$. 
    
    It remains to show that $\tau'$ is a $k$-tangle in~$G'$. 
    If there exists a forbidden triple $\{(A_i, B_i) : i \in [3]\}$ in~$G'$ for~$\tau'$, then we obtain a forbidden triple for $\tilde{\tau}$ in $G$ by replacing those $(A_i, B_i)$ that are not already separations of $G$ with $(A_i \cup e, B_i)$.
    The arising triple is then by the definition of $\tau'$ a forbidden triple in $\tilde{\tau}$. This contradicts that $\tilde{\tau}$ is a tangle in $G$ by assumption.
\end{proof}

\noindent We remark that one can also show a vertex-version of~\cref{lem:TangleExtends} along the same lines:
if a~$k$-tangle $\tau$ in $G$ extends to a $(k+1)$-tangle in $G$, then $\tau$ extends to a $k$-tangle in $G' := G-v$ for every vertex $v \in V(G)$. 
\COMMENT{Indeed, we obtain for every separation $\{A,B\}$ of $G'$ of order~$<k$ a separation $\{A \cup \{v\}, B \cup \{v\}\}$ of $G$ of order~$\leq k$ which is oriented by the $(k+1)$-tangle $\tilde{\tau}$.
By letting $(A, B) \in \tau'$ if and only if $(A \cup \{v\}, B \cup \{v\}) \in \tilde{\tau}$ we get an orientation $\tau'$ of $S_k(G')$ which is in fact a $k$-tangle in $G$:
If there exists a forbidden triple $((A_i, B_i))_{i\leq 3}$  for $\tau'$ in $G'$, then this defines a forbidden triple $((A_i \cup \{v\}, B_i \cup \{v\}))_{i \leq 3}$ for $\tilde{\tau}$ in $G'$.}

Now we turn to the case that~$G$ has a $(k+1)$-tangle $\tilde{\tau}$ which does not extend $\tau$.
Let us briefly describe our proof strategy:
First, we observe that there exists a separation~$(B, A) \in \tilde{\tau}$ which is $\leq$-maximal in~$\tilde{\tau} \cap \vS_k(G)$ and distinguishes~$\tau$ and~$\tilde{\tau}$.
We will then delete an arbitrary edge~$e$ on the side of~$\{A, B\}$ which is small with respect to~$\tau$, i.e.\,$e \in G[A \setminus B]$.
To define the desired $k$-tangle~$\tau'$ of~$G'$ to which~$\tau$ shall extend, we first orient all separations of~$G'$ that are `forced' by~$\tau$ in that they have an orientation which was either already in $\tau$ or which must be in $\tau'$ to achieve the desired consistency of $\tau'$.
The remaining separations are then oriented according to~$\tilde{\tau}$; for this, we draw on the fact that $\tilde{\tau}$ has order $k+1$ and hence naturally defines an orientation of all the separations in $S_k(G')$ using \cref{lem:StdStuff-ShiftingAnEdge}.
While $\tau$ extends to this orientation~$\tau'$ by construction, the main part of the proof is devoted to show that $\tau'$ is indeed a tangle.
Intuitively speaking, the construction of~$\tau'$ ensures that neither the separations forced by~$\tau$ nor those oriented according to~$\tilde{\tau}$ contain a forbidden triple.
The maximal choice of~$(B, A)$ together with submodularity arguments then ensures that there is also no forbidden triple consisting of both kinds of separations: it allows us to transfer any such forbidden triple in~$\tau'$ either into one in~$\tau$ or into one in~$\tilde{\tau}$.

\begin{lemma} \label{lem:AnotherExtendingTangle}
    Let $k \geq 2$ be an integer, and let $G$ be a graph with a $k$-tangle $\tau$.
    If there exists a $(k+1)$-tangle $\tilde{\tau}$ in $G$ with $\tau \nsubseteq \tilde{\tau}$, then there exists an edge $e \in E(G)$ such that $\tau$ extends to a $k$-tangle $\tau'$ in~$G - e$.
\end{lemma}

\begin{proof}
    We first find an edge~$e$ of~$G$ that we afterwards prove to be as desired.
    For this, let~$(B,A) \in \tilde{\tau}$ be a separation of $G$ which distinguishes $\tau$ and $\tilde{\tau}$ and is $\leq$-maximal in $\tilde{\tau}$ among all such distinguishing separations. 
    Note that~$(B,A)$ is even maximal in $\tilde{\tau} \cap \vS_{k}$: any separation $(C,D) \in \tilde{\tau} \cap \vS_{k}$ with~$(B,A) < (C,D)$ would also distinguish $\tau$ and $\tilde{\tau}$ since $(D,C) \in \tau$ by the consistency of $\tau$.
    We then choose an arbitrary edge $e$ in~$G[A \setminus B]$. Let us first show its existence.

    There exists a vertex~$v \in A \setminus B$, since otherwise~$(V(G), A) = (B, A) \in \tilde{\tau}$, contradicting that~$\tilde{\tau}$ is a tangle.
    If the vertex~$v$ has a neighbour in~$A \setminus B$, then the edge joining them is as desired.
    So suppose for a contradiction that all neighbours of~$v$ are in~$A \cap B$.
    We can then find a forbidden triple in~$\tilde{\tau}$:
    First, we can move~$v$ from~$A \setminus B$ to~$B \setminus A$ to obtain a new separation~$\{A \setminus \{v\}, B \cup \{v\}\}$ of~$G$, which has the same order as~$\{A, B\}$.
    Thus, $\tilde{\tau}$ contains an orientation of it, and we must have~$(A \setminus \{v\}, B \cup \{v\}) \in \tilde{\tau}$ due to the maximality of~$(B, A)$ in~$\tilde{\tau} \cap \vS_k(G)$.
    Secondly, since~$\abs{A \cap B} < k$ and~$\tilde{\tau}$ is a~$(k+1)$-tangle in~$G$, we have $((A \cap B) \cup \{v\}, V(G)) \in \tilde{\tau}$.
    Hence, $\{(B, A), (A \setminus \{v\}, B \cup \{v\}), ((A \cap B) \cup \{v\}, V(G))\}$ is contained in the tangle $\tilde{\tau}$, but it is also a forbidden triple, which is a contradiction.
    All in all, $v$ has a neighbour in $A \setminus B$; in particular, $G[A \setminus B]$ contains an edge. 

    From now on, we prove that the chosen edge~$e \in G[A \setminus B]$ is as desired.
    For this, we consider~$G' := G - e$ and construct an orientation~$\tau'$ of~$S_k(G')$ to which~$\tau$ extends.
    It then remains to show that~$\tau'$ is a tangle in~$G'$.
    
    For the construction of~$\tau'$, note that~$\tau'$ has to contain not only~$\tau$, but also all orientations of separations of~$S_k(G')$ that are `forced' by the request that~$\tau'$ shall again be a tangle and hence especially consistent.
    More formally, we say that~$\tau$ \defn{forces} an orientation of a separation~$\{C,D\}$ of~$G'$ if there exists a separation~$(E,F) \in \tau$ such that~$(C,D) \leq (E,F)$ or~$(D,C) \leq (E,F)$.
    In particular, $\tau$ forces an orientation of every separation in~$S_k(G')$ which is also a separation of $G$ and the separations in~$\vS_k(G')$ which are maximal among all those forced by $\tau$ are separations of $G$.
    Note that the consistency of $\tau$ ensures that at most one orientation of a separation in $S_k(G')$ is forced by $\tau$.
    
    We now define the orientation~$\tau'$ of~$S_k(G')$.
    If~$\tau$ forces an orientation of a separation~$\{C, D\} \in S_k(G')$, then we put the respective orientation in~$\tau'$. 
    Otherwise, $\{C,D\}$ is especially not a separation of~$G$, so~$e$ has one endvertex in~$C \setminus D$ and the other one in~$D \setminus C$.
    Then $\{C \cup e, D\}$ and $\{C, D \cup e\}$ are separations of~$G$ of order at most~$k$, as $\{C, D\}$ has order less than~$k$.
    Thus, both these separations are oriented by the $(k+1)$-tangle $\tilde{\tau}$ in $G$, and by~\cref{lem:StdStuff-ShiftingAnEdge}, we have $(C \cup e, D) \in \tilde{\tau}$ if and only if $(C, D \cup e) \in \tilde{\tau}$.
    Now if $(C \cup e, D) \in \tilde{\tau}$ (equivalently: $(C, D \cup e) \in \tilde \tau$), then we put $(C, D) \in \tau'$, and if $(D \cup e, C) \in \tilde{\tau}$ (equivalently: $(D, C \cup e) \in \tilde \tau$), then we put $(D, C) \in \tau'$. This definition of $\tau'$ ensures that $\tau'$ indeed is an orientation of~$S_k(G')$ and also that $\tau$ extends to $\tau'$.

    Thus, it remains to show that~$\tau'$ is indeed a tangle.
    Suppose for a contradiction that there is a forbidden triple $\{(C_i, D_i) : i \in [3]\}$ in $\tau'$.
    Without loss of generality, we may assume that all the $(C_i, D_i)$ are $\leq$-maximal in $\tau'$.
    We now aim to use $\{(C_i, D_i) : i \in [3]\}$ together with the construction of~$\tau'$ to find a forbidden triple in $G$ which is contained in either $\tau$ or $\tilde{\tau}$.
    This then yields a contradiction since both~$\tau$ and~$\tilde{\tau}$ are tangles in~$G$.
    Towards this, we first give a condition on the~$(C_i, D_i)$ which allows us to find a forbidden triple in~$\tilde{\tau}$ and prove afterwards that if this condition does not hold, then we can find a forbidden triple in~$\tau$.    

    First, suppose that each~$\{C_i, D_i\}$ either crosses~$\{A, B\}$ or satisfies~$(C_i, D_i) \le (B, A)$.
    In this case, we aim to find a forbidden triple in~$\tilde{\tau}$.
    Towards this, the following lemma shows that~$(C_i, D_i) \in \tilde{\tau}$ if~$\{C_i, D_i\}$ crosses~$\{A, B\}$ and is also a separation of~$G$.
    \begin{sublemma}\label{sublem:CrossingSepsThatAreSepsOfG}
        Let $\{C, D\}$ be a separation of $G'$ that is also a separation of $G$ and whose orientation $(C,D) \in \tau'$ is $\leq$-maximal in $\tau'$.
        If $\{C, D\}$ crosses $\{A, B\}$, then $(C, D) \in \tilde{\tau}$.
    \end{sublemma}
    \begin{claimproof}
        Assume that the infimum~$(A \cap C, B \cup D)$ of~$(A, B)$ and~$(C, D)$ has order less than ~$k$.
        By the maximality of~$(B, A)$ in~$\tilde{\tau} \cap \vS_k$, we then have~$(A \cap C, B \cup D) \in \tilde{\tau}$.
        Since~$\{(B, A), (A \cap C, B \cup D), (D, C)\}$ is a forbidden triple in~$G$, this then implies~$(C, D) \in \tilde{\tau}$, as desired.

        It remains to prove that~$\{A \cap C, B \cup D\}$ has order less than~$k$.
        By submodularity, it suffices to show that~$\{A \cup C, B \cap D\}$ has order at least~$k$.
        Suppose for a contradiction that it has order less than $k$.
        Then~$\tau$ contains an orientation of~$\{A \cup C, B \cap D\}$.
        Since~$\tau$ extends to~$\tau'$, we have $(C,D)$ in $\tau$.
        On the one hand, as $(C,D)$ is also $\leq$-maximal in $\tau$, we must have that its supremum~$(A \cup C, B \cap D)$ with $(A,B)$ is not in~$\tau$.
        On the other hand, the profile property of $\tau$ ensures that $(B \cap D, A \cup C) \notin \tau$, as $(A,B) \in \tau$.
        This is a contradiction.
    \end{claimproof}
    
    In this first case, where each~$\{C_i, D_i\}$ either crosses~$\{A, B\}$ or satisfies~$(C_i, D_i) \le (B, A)$, we can use~\cref{sublem:CrossingSepsThatAreSepsOfG} to obtain a forbidden triple~$\{(C_i', D_i') : i \in [3]\}$ in~$\tilde{\tau}$ as follows:
    For~$i \in [3]$, assume first that~$\{C_i, D_i\}$ is also a separation of~$G$.
    If~$(C_i, D_i) \le (B, A)$, then~$(C_i', D_i') := (C_i, D_i) \in \tilde{\tau}$ since $(B,A) \in \tilde{\tau}$ and~$\tilde{\tau}$ is consistent. 
    Otherwise, $\{C_i, D_i\}$ crosses~$\{A, B\}$ and then~$(C_i', D_i') := (C_i, D_i) \in \tilde{\tau}$ by~\cref{sublem:CrossingSepsThatAreSepsOfG}.
    Secondly, if $\{C_i, D_i\}$ is not a separation of $G$, then the maximality of~$(C_i, D_i)$ in~$\tau'$ implies that~$(C_i, D_i)$ cannot be forced by~$\tau$.
    Thus, we have~$(C_i', D_i') := (C_i \cup e, D_i) \in \tilde{\tau}$ by construction.
    Now $\{(C_i', D_i') : i \in [3]\}$ is a forbidden triple in the tangle~$\tilde{\tau}$ in $G$, which is a contradiction.

    So we assume that some~$\{C_i, D_i\}$, say~$\{C_1, D_1\}$, neither crosses~$\{A, B\}$ nor satisfies~$(C_i, D_i) \le (B, A)$.
    In this case, we aim to find a forbidden triple in~$\tau$.
    We claim that~$(A, B) \le (C_1, D_1)$ and that this yields~$(C_1, D_1) \in \tau$:
    By our assumptions, $\{C_1, D_1\}$ is nested with~$\{A, B\}$, but we do not have~$(C_1, D_1) \le (B, A)$ (equivalently: $(A,B) \leq (D_1, C_1)$). 
    If $(D_1,C_1) \leq (A,B)$, then $(D_1,C_1) \in \tau'$ is forced by $(A,B) \in \tau$ but $\tau'$ is an orientation which already contains $(C_1,D_1)$, which is a contradiction.
    Furthermore, we cannot have~$(C_1, D_1) < (A, B)$, since~$(C_1, D_1)$ is maximal in~$\tau'$.
    Thus, $(A, B) \le (C_1, D_1)$; in particular, $e \in G[A \setminus B] \subseteq G[C_1 \setminus D_1]$.
    Therefore, $(C_1, D_1)$ is not only a separation of~$G - e = G'$, but also one of~$G$.
    Since~$\tau$ extends to~$\tau'$ and~$(C_1, D_1) \in \tau'$, we obtain~$(C_1, D_1) \in \tau$.

    As shown, we have $e \in G[C_1]$.
    So if~$(C_2, D_2)$ and~$(C_3, D_3)$ are separations of~$G$, then they are not only in $\tau'$ but also in~$\tau$ as~$\tau$ extends to~$\tau'$ which yields our desired forbidden triple in~$\tau$.
    So suppose that~$(C_i, D_i)$ with~$i \in \{2, 3\}$ is not a separation of~$G$.
    We claim that~$\{C_i, D_i\}$ crosses~$\{A, B\}$.
    Indeed, if $\{C_i, D_i\}$ had an orientation that is greater than $(A, B)$, then $\{C_i, D_i\}$ would be a separation of $G$, as the deleted edge $e$ is contained in~$G[A \setminus B]$.
    If $(C_i, D_i) < (A, B)$, then $(C_i, D_i)$ would in contradiction to its choice not be maximal in $\tau'$, 
    since $(A,B)$ is also in $\tau'$, 
    If $(D_i,C_i) \leq (A,B)$, then $(D_i,C_i) \in \tau'$ is forced by $(A,B) \in \tau$, which is a contradiction to $(C_i,D_i) \in \tau'$.
    So~$\{C_i, D_i\}$ cannot be nested with~$\{A, B\}$, that is, they cross.
    But then the following lemma shows that the infimum of~$(C_i, D_i)$ and~$(B, A)$ is in~$\tau$.

    \begin{sublemma} \label{sublem:CrossingSepsThatAreNotSepsOfG}
        Let $\{C, D\}$ be a separation of $G'$ that is not a separation of $G$ and whose orientation $(C,D) \in \tau'$ is maximal in $\tau'$.
        If $\{C, D\}$ crosses $\{A, B\}$, then $(B \cap C, A \cup D) \in \tau$.
    \end{sublemma}
    \begin{claimproof}
        Since~$e \in G[A \setminus B]$, $\{B \cap C, A \cup D\}$ is a separation of~$G$.
        Assume that it has order less than~$k$.
        Then~$\tau$ contains an orientation of~$\{B \cap C, A \cup D\}$, and this orientation must not be~$(A \cup D, B \cap C)$ as~$\tau$ would otherwise force the orientation~$(D, C)$ of~$\{C, D\}$ to be in the orientation~$\tau'$ which already contains $(C,D)$.

        It remains to show that~$\{B \cap C, A \cup D\}$ indeed has order less than~$k$; suppose for a contradiction otherwise.
        Since both~$\{A, B\}$ and~$\{C, D\}$ have order at most~$k-1$, this implies that~$\{B \cup C, A \cap D\}$ has order less than~$k-1$ by submodularity.
        The edge~$e$ is in~$G[A \setminus B]$ by its choice.
        Additionally, it has one endvertex in~$C \setminus D$ and the other one in~$D \setminus C$ because~$\{C, D\}$ is not a separation of~$G$ by assumption.
        Therefore, the order of~$\{B \cup C \cup e, A \cap D\}$ increases compared to the order of~$\{B \cup C, A \cap D\}$ by exactly one.
        So~$\{B \cup C \cup e, A \cap D\}$ has order~$< k$.
        We now show that none of its orientations can be contained in the tangle~$\tilde{\tau}$, which then yields the desired contradiction.

        On the one hand, the maximality of~$(B, A)$ in~$\tilde{\tau} \cap \vS_k$ implies that~$(B \cup C \cup e, A \cap D) \notin \tilde{\tau}$.
        On the other hand, since~$\{C, D\}$ is not a separation of~$G$, but~$(C, D)$ is maximal in~$\tau'$, the orientation~$(C, D)$ of~$\{C, D\}$ cannot be forced by~$\tau$.
        Hence by construction of $\tau'$, we put~$(C, D) \in \tau'$ because of~$(C \cup e, D) \in \tilde{\tau}$.
        But~$\{(B, A), (C \cup e, D), (A \cap D, B \cup C \cup e)\}$ is forbidden triple in~$G$, so~$(A \cap D, B \cup C \cup e) \notin \tilde{\tau}$.
    \end{claimproof}
    
    Using~\cref{sublem:CrossingSepsThatAreNotSepsOfG}, we can now find a forbidden triple~$\{(C_1, D_1), (C_2', D_2'), (C_3', D_3')\}$ in~$\tau$ as follows:
    As shown above, $(C_1, D_1)$ is in~$\tau$ and satisfies~$(A, B) \le (C_1, D_1)$.
    If~$(C_i, D_i)$ with~$i \in \{2, 3\}$ is a separation of~$G$, then it also is in~$\tau$, as $\tau$ extends to $\tau'$, and we set~$(C_i', D_i') := (C_i, D_i)$.
    If it is not a separation of~$G$, then~$\{C_i, D_i\}$ must cross~$\{A, B\}$, as shown above~\cref{sublem:CrossingSepsThatAreNotSepsOfG}, and~\cref{sublem:CrossingSepsThatAreNotSepsOfG} yields~$(C_i', D_i') := (B \cap C_i, A \cup D_i) \in \tau$.
    To see that $\{(C_1, D_1), (C'_2, D'_2), (C'_3, D'_3)\} \subseteq \tau$ is indeed a forbidden triple in $G$, note that
	\begin{equation*}
        G[C_1] \cup G[C'_2] \cup G[C'_3] \supseteq G[C_1] \cup G[C_2 \cap B] \cup G[C_3 \cap B] \supseteq G[C_1] \cup G[C_2 \cap D_1] \cup G[C_3 \cap D_1] = G,   
	\end{equation*}
    where the last equation holds because $e \in G[A] \subseteq G[C_1]$ and $\{(C_i, D_i) : i \in \{1,2,3\}\}$ is a forbidden triple in $G' = G - e$.
    This concludes the proof.   
\end{proof}

\begin{proof}[Proof of \cref{thm:NoBigTangle}]
    This follows immediately from \cref{lem:TangleExtends} and~\cref{lem:AnotherExtendingTangle}.
\end{proof}

Our proof of \cref{thm:NoBigTangle}  heavily relies on the fact that the order of the additional tangle~$\tilde{\tau}$ is greater than the one of~$\tau$.
However, we do not know whether similar proof techniques could also work if the order of~$\tilde{\tau}$ does not exceed the one of~$\tau$:
\begin{problem}
    Let~$\tau$ be a tangle in~$G$ of order $k \geq 3$.
    Suppose further that there exists another $k$-tangle~$\tau^*$ in~$G$ with~$\tau^* \nsubseteq \tau$ and~$G$ has minimum degree at least $3$.
    Is there an edge~$e \in G$ such that~$\tau$ extends to a~$k$-tangle~$\tau'$ in~$G-e$?
\end{problem}

\section{Rainbow-Cloud-Decompositions in the absence of high-order tangles} \label{sec:ExistenceRCDecomp}

Recall that \cref{thm:NoBigTangle} in \cref{sec:ExistenceOfHighOrderTangle} immediately yields \cref{main:SplitterTheorem} if the graph has a $(k+1)$-tangle.
So from now on we work towards a proof of~\cref{main:SplitterTheorem} for graphs without tangles of high order.
In this section, we show that, in the absence of~$(k+1)$-tangles, a large graph admits a certain type of decomposition, which we will call `rainbow-cloud decomposition'; this decomposition is inspired by~\cite{StructureOf6ConnectedGraphs}.
We will later use that this decomposition exhibits a substructure of the graph, the `rainbow', which is a long linear structure that is fairly independent of the rest of the graph and internally made up of similar enough parts such that deleting an edge in one of the parts does not change the overall structure of the graph.
In particular, it will allow us to understand how the separations of the graph change after deleting such an edge and hence how to find a tangle of this smaller graph to which our given tangle extends.

We begin by building up to the definition of a `rainbow-cloud decomposition'.
Let~$G$ be a graph.
First, a \defn{linear decomposition}\footnote{Decompositions satisfying~\labelcref{itm:LinDecomp1,itm:LinDecomp2} are often known as \defn{path-decompositions} (cf.\ \cite{DiestelBook16noEE}*{\S 12.6}). In \cite{StructureOf6ConnectedGraphs} these are refereed to as \defn{linear decompositions}. In these paper, linear decompositions will always not only satisfy~\labelcref{itm:LinDecomp1,itm:LinDecomp2} but also~\labelcref{itm:LinDecomp3,itm:LinDecomp4}.} of~$G$ of \defn{length}~$M \in \N$ of~$G$ is a family $\cW = (W_0, W_1, \dots, W_M)$ of sets~$W_i$ of vertices of~$G$ such that
\begin{enumerate}[label=(L\arabic*)]
    \item\label{itm:LinDecomp1} $\bigcup_{i=0}^M G[W_i] = G$, 
    \item\label{itm:LinDecomp2} if $0 \leq i \leq j \leq k \leq M$, then $W_i \cap W_k \subseteq W_j$, and
    \item \label{itm:LinDecomp3} there is an integer $\ell$ such that $|W_{i-1} \cap W_i| = \ell$ for every $i \in [M]$, and
    \item\label{itm:LinDecomp4} $W_{i-1} \neq W_{i-1} \cap W_{i} \neq W_{i}$ for all~$i \in [M]$.
\end{enumerate}
\noindent We call the sets~$W_i$ the \defn{bags} and and the induced subgraphs $G[W_i]$ the \defn{parts} of the linear decomposition~$\cW$. Note that the bags of a linear decomposition of length at least $1$ are non-empty by \cref{itm:LinDecomp4}.
The \defn{adhesion sets} of a linear decomposition~$\cW$ are the sets~$U_i := W_{i-1} \cap W_i$ for $i \in [M]$.
The size of the adhesion sets~$U_i$ is the \defn{adhesion} of $\cW$.
We emphasise that adhesion~$0$ is allowed.
Whenever we introduce a linear decomposition as~$\cW$ without specifying its bags, then we will tacitly assume the bags to be denoted by~$W_0, \dots, W_M$ and the adhesion sets by~$U_1, \dots, U_M$.

Next we turn to the definition of `rainbow-decompositions' which are special linear decompositions whose adhesion sets are minimal $U_1$--$U_M$ separators of~$G$ as witnessed by respective families of disjoint paths.
To make this formal, a \defn{linkage} in a graph~$G$ is a set $\cP$ of disjoint paths in~$G$.
If $A$ and $B$ are sets of vertices of~$G$ such that~$\cP$ consists of $A$--$B$ paths, i.e.\,such paths that meet~$A$ precisely in one endvertex and~$B$ precisely in the other endvertex, then $\cP$ forms an \defn{$A$--$B$ linkage}.
A linear decomposition $\cW$ of adhesion $\ell$ and length~$M$ is called a \defn{rainbow-decomposition} of \defn{adhesion} $\ell$ and \defn{length} $M$ if it has the following three properties:
\begin{enumerate}[label=(R\arabic*)]
    \item\label{itm:RainbowDecompLinkage} There exists a $U_i$--$U_{i+1}$~linkage of cardinality $\ell$ in $G[W_i]$ for every $i \in [M-1]$.
    \item\label{itm:RainbowDecompConnected} Every part $G[W_i]$ of $\cW$ is connected.
    \item\label{itm:RainbowDecompDistinctBags} 
    Every two consecutive adhesion sets $U_i, U_{i+1}$ are disjoint.
\end{enumerate}
\noindent We may combine the linkages of cardinality~$\ell$ from~\cref{itm:RainbowDecompLinkage} to obtain a $U_1$--$U_M$ linkage~$\cP$ of cardinality $\ell$ in $G$.
We call such a linkage~$\cP$ a \defn{foundational linkage} of the rainbow-decomposition.

Finally, we define `rainbow-cloud-decompositions', which consist of a rainbow-decomposition of a subgraph of~$G$ that interacts with the remainder of~$G$, the `cloud', in a very controlled manner (see~\cref{fig:RC-Decomp} for an illustration).
Formally, a \defn{rainbow-cloud-decomposition} (or \defn{RC-decomposition} for short) of~$G$ of \defn{adhesion}~$\ell$ and \defn{length}~$M$ is a quadruple~$(R, \cW, Z, C)$ consisting of two induced subgraphs~$R$ and~$C$ of a graph~$G$, a vertex set~$Z \subseteq V(C)$ disjoint from $V(R)$ such that~$G[V(R) \cup Z] \cup C = G$ and a rainbow-decomposition~$\cW = (W_0, \dots, W_M)$ of~$R$ of adhesion~$\ell$ and length~$M$ with adhesion sets~$U_1, \dots, U_M$ and two additional adhesion sets~$U_0 := V(C) \cap W_0$ and~$U_{M+1} := V(C) \cap W_{M}$ such that
\begin{enumerate}[label=(RC\arabic*)]
    \item\label{itm:RCDecompRcapC} $V(R) \cap V(C) = U_0 \cup U_{M+1}$,
    \item\label{itm:RCDecompAdditionalCutsets} $\abs{U_0} = \ell = \abs{U_{M+1}}$ and~$U_0 \cap U_1 = \emptyset = U_M \cap U_{M+1}$,
    \item\label{itm:RCDecompAdditionalLinkages} there exists a~$U_0$--$U_1$~linkage in~$G[W_0]$ and a~$U_M$--$U_{M+1}$~linkage in~$G[W_M]$, both of cardinality~$\ell$, and
    \item\label{itm:RCDecompEdgesToZ} $Z \subseteq N_G(W_i)$ for every~$i \in \{0,\dots, M \}$.
\end{enumerate}
We refer to $R$ as the \defn{rainbow}, to $C$ as the \defn{cloud} of the RC-decomposition and to $Z$ as the \defn{sun} of the RC-decomposition.
Whenever we introduce an RC-decomposition $(R, \cW, Z, C)$, we tacitly assume that $\cW = (W_0, \dots, W_M)$ and $U_0, \dots, U_{M+1}$ are defined as above.%
\COMMENT{Observe that we do not require~$U_0$ and~$U_{M+1}$ to be connected, which preserves us from stating (RC1-3) via a longer rainbow-decomposition.}

\begin{figure}[ht]
    \centering
    \pdfOrNot{\includegraphics[page=2]{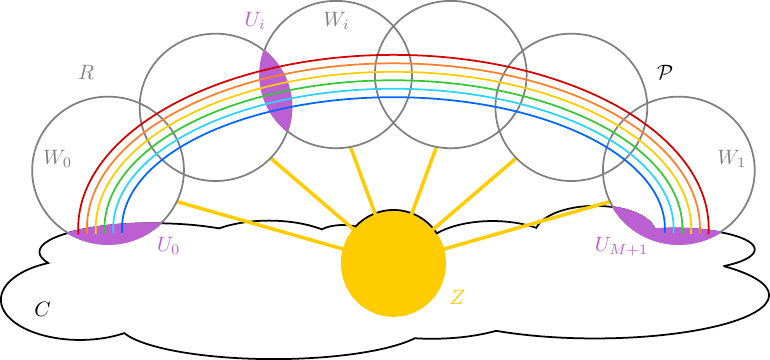}}{\includesvg[width=0.8\columnwidth]{svg/RCdecompositionschematic2.svg}}
    \caption{A rainbow-cloud decomposition $(R, \cW, Z, C)$: the rainbow $R$ together with its decomposition $\cW$ is indicated in grey, the foundational linkage $\cP$ of $\cW$ is depicted in rainbow colours, and the cloud $C$ is depicted in black. Further, indicated in yellow, is the sun $Z \subseteq V(C)$ together with the $Z$--$R$ edges required for \cref{itm:RCDecompEdgesToZ}.
    The adhesion sets $U_0, U_i, U_{M+1}$ are depicted in brown.}
    \label{fig:RC-Decomp}
\end{figure}

With the definition of rainbow-cloud-decompositions at hand, we can state the main result of this section:
\begin{theorem}\label{thm:ExistenceOfRCDecomp}
    For every two integers~$k, M \ge 1$, there exists some integer~$N = N(k, M) \ge 1$ such that every connected graph $G$ with at least $N$ vertices and no $(k+1)$-tangle admits an RC-decomposition $(R, \cW, Z, C)$ of length at least $M$ and adhesion $\ell$ such that $\abs{Z} + \ell \geq 1$.
\end{theorem}

The remainder of this section is devoted to the proof of~\cref{thm:ExistenceOfRCDecomp}, which roughly proceeds as follows:
We will start by using the tangle-tree duality theorem, one of the two central tangle theorems, to get a long nested sequence of separations which will allow us to construct a still long linear decomposition. 
This linear decomposition can subsequently be refined into an RC-decomposition, by putting the `unnecessary bits' of its bags into the cloud.

As a first step, \cref{thm:sequence_seps} asserts that every sufficiently long sequence of separations contains a long subsequence which induces a linear decomposition of some adhesion~$\ell$ satisfying~\ref{itm:RainbowDecompLinkage}.
In other words, we find a subsequence such that all its elements have the same order and such that for every two successive separations, there exists a linkage between its separators.
By Menger's theorem (e.g. \cite{DiestelBook16noEE}*{Theorem~3.3.1}), this second property is equivalent to the absence of any separation of smaller order between two successive separations.

We start with a lemma about sequences of positive integers, which we will later apply to the order of the separations in the sequence.
In what follows all sequences will be finite.

\begin{lemma}\label{thm:sequence}
    Let $n, m, p \ge 1$ be integers with $p \geq n^m$.
    Then every sequence~$(a_1, \dots, a_p)$ of integers of length~$p$ with~$a_i \in \{0, \dots, m-1\}$ for every~$i \in [p]$ has a subsequence $(a_{i_1}, \dots, a_{i_n})$ of length~$n$ such that
    \begin{enumerate}
        \item \label{item:sequence:i} $\ell := a_{i_1} = \dots = a_{i_n}$, and
        \item \label{item:sequence:ii} $a_j \geq \ell$ for all~$i_1 \leq j \leq i_n$. 
    \end{enumerate}
\end{lemma}
\begin{proof}
    We proceed by induction on~$m$.
    For~$m = 1$, we have~$a_i = 0$ for all~$1 \le i \le n \le p$ which immediately yields the statement.
    So consider~$m \ge 2$.
    If at least~$n$ of the~$a_i$ equal~$0$, then any~$n$ of them form the desired sequence.
    So suppose that at most~$n' < n$ of the~$a_i$ equal~$0$ and let~$i_1, \dots, i_{n'}$ be the respective indices.
    Now consider the~$n' + 1$ subsequences
    \begin{equation}\label{subsequences}
        (a_1, \dots, a_{i_1-1}), (a_{i_1+1}, \dots, a_{i_2-1}), \dots, (a_{i_{n'}+1}, \dots, a_{p})
    \end{equation}
    of consecutive~$a_1, \dots, a_p$; note that some of these subsequences may be empty.
    If each of these subsequences has length less than~$n^{m-1}$, then we obtain a contradiction via
    \begin{equation*}
        n^{m} \leq p \leq (n'+ 1)(n^{m-1}-1) + n' \leq n(n^{m-1}-1) + (n-1) = n^{m} - 1.
    \end{equation*}
    So one of the subsequences in \cref{subsequences}, let us call it~$(b_1, \dots, b_{p'})$, must have length at least~$p' \ge n^{m-1}$, and we thus can apply the induction hypothesis to~$(b_1 - 1, \dots, b_{p'} - 1)$ with integers $n, m-1, p'$ to obtain a subsequence of $(b_1, \dots, b_{p'})$ which is as desired.
\end{proof}

Given a graph~$G$, a sequence~$((A_i, B_i))_{i \in [p]})$ of (oriented) separations is~\defn{strictly increasing} if~$(A_i, B_i)< (A_j, B_j)$ for every two elements $i < j$ of $[p]$.%
\COMMENT{The $<$ in particular implies that all~$s_i$ are distinct.}

\begin{lemma}\label{thm:sequence_seps}
    Let $G$ be a graph and $n, m, p \ge 1$ be integers with $p \geq n^m$.
    If there is a strictly increasing sequence of length $p$ in~$\vS_m(G)$, then there is also a strictly increasing sequence $((A_i, B_i))_{i \in [n]}$ of length~$n$ in~$\vS_m(G)$ such that
    \begin{enumerate}
        \item\label{itm:sequence_seps_1} $\ell := |A_1, B_1| = \dots = |A_n, B_n|$, and
        \item\label{itm:sequence_seps_2} for every separation~$(A, B) \in \vS_\ell(G)$ there is no $i \in [n-1]$ with $(A_i, B_i) < (A, B) < (A_{i+1}, B_{i+1})$.
    \end{enumerate}
\end{lemma}
\begin{proof}
    With each strictly increasing sequence $T$ in~$\vS_m(G)$, we associate a sequence~$(n_i(T))_{i \in [m-1]}$ of integers where~$n_i(T)$ denotes the number of separations in the strictly increasing sequence that have exactly order~$i$.
    Whenever we will compare sequences of integers in this proof, we do so with respect to the lexicographic order.
    Let~$\cT$ denote the set of all strictly increasing sequence in~$\vS_m(G)$ of length at least~$p$, and let~$T = ((C_i, D_i))_{i \in [q]}$ be an element in~$\cT$ whose associated sequence of integers is maximal among all such sequences associated with elements in $\cT$. By assumption there exists some such strictly increasing sequence of length at least $p$ and, as the graph $G$ is finite, there is a maximal sequence of integers among such associated to elements in $\cT$.

    Applying~\cref{thm:sequence} to~$(|C_i, D_i|)_{i \in [q]}$, we obtain a subsequence~$((C_{i_j}, D_{i_j}))_{j \in [n]}$ which we will show to be as desired. A subsequence of a strictly increasing sequence is again strictly increasing. It also satisfies \cref{itm:sequence_seps_1}, as all the separations in the subsequence have the same order $\ell$ by~\cref{thm:sequence}.
    It remains to show that~\cref{itm:sequence_seps_2} holds as well.
    To do so, we show that if~\cref{itm:sequence_seps_2} does not hold for~$((C_{i_j}, D_{i_j}))_{j \in [n]}$, then we find an element in~$\cT$ whose associated sequence of integers is larger than the one associated to~$((C_i, D_i))_{i \in [q]}$, which contradicts our choice.

    So suppose that there exists a separation~$(A, B) \in \vS_m(G)$ of order less than~$\ell$ and an integer~\mbox{$j \in [n-1]$} such that~$(C_{i_j}, D_{i_j}) < (A, B) < (C_{i_{j+1}}, D_{i_{j+1}})$.
    Consider the subsequence~$$((C_{i_j}, D_{i_j})), (C_{i_j +1}, D_{i_j +1}), \dots, (C_{i_{j+1}}, D_{i_{j+1}}))$$ of~$((C_i, D_i))_{i \in [p']}$; for notational simplicity, we also denote this subsequence by~$R = ((A_i, B_i))_{i \in [r]}$.
    Then we can obtain a new strictly increasing sequence~$R' = ((A'_i, B'_i))_{i \in [r']}$ by removing duplicates, which will only appear consecutively, from the sequence
    \begin{equation*}
        (A_1 \cap A, B_1 \cup B), \dots, (A_q \cap A, B_q \cup B), (A, B), (A_1 \cup A, B_1 \cap B), \dots, (A_r \cup A, B_r \cap B).
    \end{equation*}

    We now consider the strictly increasing sequence~$T'$ which is obtained from~$T$ by replacing its subsequence~$R$ with~$R'$.
    As every element of the distributive lattice $\vU(G)$ is uniquely determined by its infimum and supremum with any given other element in $\vU(G)$, $T'$ has length at least $q +1 \geq p$; thus, $T'\in \cT$.
    Moreover, the sequence of integers associated to~$T'$ is larger than the one associated with~$T$, which contradicts the choice of $T$.
    Indeed by the choice of $R$, all separations~$(C_g, D_g)$ with $i_j \le g \le i_{j+1}$ have order at least~$\ell$.
    Hence, $n_i(T') \geq n_i(T)$ for every~$i < \ell$ and moreover~$n_{|A,B|}(T') \geq n_{|A,B|}(T) + 1$, since the new sequence $T'$ additionally contains~$(A, B)$.
\end{proof}

To proceed to the proof of~\cref{thm:ExistenceOfRCDecomp}, we recall the tangle-tree duality theorem and all necessary definition:
For a tree $T$ and two nodes or edges $x,y$ of $T$, we denote by $xTy$ the (unique) $\subseteq$-minimal path in $T$ which contains $x$ and $y$.
A \defn{\td} of a graph $G$ is a pair $(T,\cV)$ of a tree $T$ and a family $\cV$ of subsets $V_t$ of $V(G)$ indexed by the nodes of $T$ such that
\begin{enumerate}[label=(T\arabic*)]
    \item\label{itm:TD1} $\bigcup_{t \in V(T)} G[V_t] = G$, and
    \item\label{itm:TD2} for every three nodes $r,s,t \in T$ with $s \in rTt$ we have $V_r \cap V_t \subseteq V_s$.
\end{enumerate}
The maximum of the sizes of the $V_t$ minus $1$ is the \defn{width} of $(T,\cV)$.
The \defn{adhesion set} $V_e$ corresponding to an edge $e = t_1t_2$ of $T$ is $V_{t_1} \cap V_{t_2}$.
The maximum of the sizes of the $V_e$ is the \defn{adhesion} of $(T, \cV)$.
Every orientation $(t_1,t_2)$ of an edge $e$ of $T$ \defn{induces} a separation $(\bigcup_{t \in T_1} V_t, \bigcup_{t \in T_2} V_t)$ of $G$ with separator $V_e$ where $T_1 \ni t_1$ and $T_2 \ni t_2$ are the two components of $T-t_1t_2$ (cf.\, \cite{DiestelBook16noEE}*{Lemma 12.3.1}).
It is immediate from the definition that the set of separations induced by a \td\ is nested.
Conversely, it is well-known (e.g.\,proof of \cite{InfiniteSplinters}*{Lemma~2.7} or~\cites{TangleTreeGraphsMatroids,TreeSets}) that every nested set $N$ of separations of a finite graph \defn{induces} a \td\ whose induced separations are in bijection with $N$.
We remark that \td s $(T,\cV)$ where $T$ is a path of length $M$ correspond precisely to the families $\cW = (W_0,\dots, W_M)$ which satisfy \cref{itm:LinDecomp1} and \cref{itm:LinDecomp2}.
We now recall the tangle-tree duality theorem, rephrased here in the version which we need later:

\begin{theorem}[e.g. \cite{DiestelBook16noEE}*{Theorem 12.5.1}]\label{theorem:TTD}
    Every graph $G$ with no $(k+1)$-tangle admits a \td\ $(T, \cV)$ of adhesion at most $k$ such that every two induced separations are distinct and for every node $t \in T$ the set of (oriented) separations induced by the oriented edges $(s,t)$ of $T$ is in $\cT$.
    In particular, every node of $T$ has degree at most $3$ and the width of $(T, \cV)$ is less than $3k$.
\end{theorem}

The next lemma assets that in the absence of high-order tangles, the graph admits a certain type of linear decomposition.
To prove it, we first apply the tangle-tree duality to obtain a tree-decomposition of small width and then sort out a suitable linear decomposition by~\cref{thm:sequence_seps}.

\begin{lemma}\label{thm:lemma1}
    For every two integers $k, M \ge 1$, there exists an integer~$N_1 = N_1(k, M) \ge 1$ such that if a graph $G$ with more than~$N_1$ vertices has no~$(k+1)$-tangle, then there exists a linear decomposition $\cW$ of~$G$ of length at least~$M$ and adhesion at most $k$ such that~$\cW$ satisfies~\cref{itm:RainbowDecompLinkage}.
\end{lemma}

We remark that the proof will show that $N_1(k,M) = 3k \cdot 3^{(M+2)^{k+1}}$ suffices.

\begin{proof}[Proof of \cref{thm:lemma1}]
    We set~$M_1 := M + 2$, $M_2 := {M_1}^{k+1}$, $M_3 := 3^{M_2}$ and~$N_1:= N_1(k, M) := 3k M_3$. 
    Let~$G$ be a graph with more than $N_1$ vertices and no $(k+1)$-tangle.
    Since~$G$ has no~$(k+1)$-tangle, it admits a tree-decomposition~$(T, \cV)$ of width at most $3k$ and adhesion at most $k$ such that $T$ has maximum degree $\leq 3$.
    Moreover, we may choose $(T,\cV)$ such that all its induced separations are distinct.
    Then~$T$ contains at least~$|G|/3k \geq N_1 / 3k \geq M_3$ vertices.
    Hence, $T$ contains a path of length at least~$M_2$, as all the nodes of $T$ have degree at most $3$.
    
    Fix a path~$P = p_0 \dots p_{M_2}$ in~$T$, and let~$(A'_i, B'_i)$ be the separation of~$G$ induced by the oriented edge $(p_{i-1},p_i)$ of~$T$ for all~$i \in [M_2]$.
    As $(T, \cV)$ has adhesion at most $k$, all these separations have order at most~$k$.
    It is also immediate from the definition of inducing a separation that~$(A'_i, B'_i) < (A'_{i+1}, B'_{i+1})$ for all~$i \in [M_2-1]$.
    Thus, $((A'_i, B'_i))_{i \in [M_2]}$ is a strictly increasing sequence of length~$M_2$ in~$\vS_k(G)$.

    Hence, by~\cref{thm:sequence_seps}, we obtain a new strictly increasing sequence~$((A_i, B_i))_{i \in [M_1]}$ of length~$M_1$ in~$\vS_k(G)$ whose elements all have the same order~$\ell \le k$ and such that there is no separation~$(A, B) \in \vS_\ell(G)$ with~$(A_i, B_i) < (A, B) < (A_{i+1}, B_{i+1})$ for all~$i \in [M_1-1]$.
    From this sequence, we now construct a linear decomposition~$\cW = (W_0, \dots, W_{M_1})$ via
    \begin{equation*}
        W_0 := A_1, W_i := B_i \cap A_{i+1} \text{ for } i \in [M_1-1], \text{ and } W_{M_1} := B_{M_1}.
    \end{equation*}
    As we have noted above for tree-decompositions, $\cW$ indeed satisfies~\cref{itm:LinDecomp1} and~\cref{itm:LinDecomp2}. Note that the adhesion set~$U_i$ equals~$W_{i-1} \cap W_{i} = A_i \cap B_i$; thus, \cref{itm:LinDecomp3} holds as well. %
    \COMMENT{Of course, this could be done formally here. But there seems nothing too surprising. It is just a matter of calculation and invoking the definition of~$\le$ as well as the definition of separations.}
    Moreover, $\cW$ has length~$M_1$ and adhesion~$\ell \le k$.

    Before we prove~\cref{itm:LinDecomp4}, let us check that~$\cW$ satisfies~\cref{itm:RainbowDecompLinkage}.
    By Menger's theorem (e.g. \cite{DiestelBook16noEE}*{Theorem~3.3.1}), it suffices to show that for~$i \in [M_1-1]$, the part~$G[W_i]$ contains no separation $(C',D')$ of order less than~$\ell$ with $U_i \subseteq C'$ and $U_{i+1} \subseteq D'$.
    Suppose for a contradiction that such a separation exists for some~$i \in [M_1-1]$.
    Since~$\cW$ satisfies~\cref{itm:LinDecomp1} and~\cref{itm:LinDecomp2}, the ordered pair
    \begin{equation*}
        (C, D) := (C' \cup \bigcup_{j < i} W_j, D' \cup \bigcup_{j > i} W_j)
    \end{equation*}
    is a separation of~$G$ and has the same order as $(C',D')$; in particular, its order is less than~$\ell$.
    The construction of~$\cW$ ensures that for~$i \in [M_1]$ we have
    \begin{equation*}
        (A_i, B_i) = (\bigcup_{j < i} W_j, \bigcup_{j \ge i} W_j).
    \end{equation*}
    Thus, $(A_{i}, B_{i}) < (C, D) < (A_{i+1}, B_{i+1})$, which contradicts our choice of the~$(A_j, B_j)$ via~\cref{thm:sequence_seps}.

    We finally turn to~\cref{itm:LinDecomp4}.
    If~$W_{i-1} \subseteq W_i$ for some~$i \in \{2, \dots, M_1\}$, then~$W_{i-1} \subseteq A_i \cap B_i$, as we have noted above that $W_i \subseteq A_i \cap B_i$.
    Since~$\ell = |A_{i-1}, B_{i-1}| = |A_i, B_i|$, we thus have~$A_{i-1} \cap B_{i-1} = W_{i-1} = A_i \cap B_i$.
    Together with $W_{i-1} = B_{i-1} \cap A_i$, this implies~$(A_{i-1}, B_{i-1}) = (A_i, B_i)$, which contradicts that the $(A_i,B_i)$ are distinct by choice.
    A symmetrical argument shows that~$W_{i-1} \nsupseteq W_i$ for all~$i \in \{1, \dots, M_1-1\}$.
    Thus, we have~$W_i \neq W_i \cap W_{i+1} \neq W_{i+1}$ for all~$i \in \{1, \dots, M_1 - 2\}$.
    However, we still might have~$W_0 \subseteq W_1$ and~$W_{M_1 - 1} \supseteq W_{M_1}$.
    In any of these cases, we remove~$W_0$ or~$W_{M_1}$, respectively, from~$\cW$.%
    \COMMENT{If we have length~$2$ and both bags are equal, then we should technically only drop one of them.}
    This operation does not affect any of the properties of~$\cW$, except that its length might decrease by at most~$2$; however, we have~$M_1 = M + 2$, so the length of the obtained linear decomposition of~$G$ is still at least $M_1$, as desired.
    This completes the proof.
\end{proof}

We will now transform a linear decomposition such as the one in~\cref{thm:lemma1} into the desired rainbow-cloud-decomposition.
As an intermediate step, let us find a linear decomposition $\cW = (W_0, W_1, \ldots, W_M)$ with a foundational linkage~$\cP$ which satisfies two additional properties\footnote{\cref{itm:FoundationalLinkage1} and \cref{itm:FoundationalLinkage2} are (L7) and (L8) in \cite{StructureOf6ConnectedGraphs}, respectively.}:
\begin{enumerate}[label=(FL\arabic*)]
    \item \label{itm:FoundationalLinkage1} For every~$P \in \cP$, if there exists~$i \in [M-1]$ such that $P[W_i] = P \cap G[W_i]$ is a trivial path, then $P[W_i]$ is a trivial path for all $i \in [M-1]$.
    \item \label{itm:FoundationalLinkage2} For every two distinct~$P, P' \in \cP$, if there exists~$i \in [M-1]$ such that there is path in~$G[W_i]$ with one endvertex in~$P$ and the other in~$P'$ and whose internal vertices avoid every path in $\cP$, then this holds for every $i \in [M-1]$.
\end{enumerate}

The next lemma yields that the existence of a linear decomposition of some suitable length whose foundational linkage satisfies~\cref{itm:FoundationalLinkage1} and~\cref{itm:FoundationalLinkage2} is ensured by a long enough linear decomposition as in \cref{thm:lemma1}.

\begin{lemma}[\cite{StructureOf6ConnectedGraphs}*{Lemma~3.5}\footnote{More precisely, it follows from the proof of \cite{StructureOf6ConnectedGraphs}*{Lemma~3.5}, as (L6) is only assumed to guarantee that the resulting decomposition also satisfies (L6).}]\label{thm:lemma2}
    For every two integers~$M \ge 1$ and~$\ell \ge 0$, there exists an integer~$M_1 = M_1(\ell, M) \ge 1$ such that if a linear-decomposition $\cW = (W_0, \dots, W_{M_1})$ of a graph~$G$ has length $M_1$, adhesion~$\ell$ and pairwise distinct~$W_i$, and~$\cW$ satisfies~\cref{itm:RainbowDecompLinkage}, then~$G$ has a linear decomposition $\cW'$ of length at least $M$ which additionally has a foundational linkage satisfying~\cref{itm:FoundationalLinkage1} and~\cref{itm:FoundationalLinkage2}.
\end{lemma}

\noindent We remark that the proof of~\cref{thm:lemma2} in~\cite{StructureOf6ConnectedGraphs} shows that~$M_1(\ell, M) = (M^{\binom{\ell}{2} + 1} \cdot \binom{\ell}{2}!)^{\ell+1} \cdot l!$ suffices.

We now turn to the proof of~\cref{thm:ExistenceOfRCDecomp}.
For this, we use the previous lemmas to find a linear decomposition with strong structural properties, which we then refine into the desired rainbow-cloud-decomposition.

\begin{proof}[Proof of~\cref{thm:ExistenceOfRCDecomp}]
    Suppose that~$G$ has no~$(k+1)$-tangle.
    We set $N(k, M) := N_1(k, M_1(k, M+2))$, where~$N_1$ and~$M_1$ are as in~\cref{thm:lemma1} and~\cref{thm:lemma2}, respectively.
    Then~\cref{thm:lemma1} yields a linear-decomposition of~$G$ of length at least~$M_1(k, M+2)$ and adhesion~$\ell \le k$ which satisfies~\cref{itm:RainbowDecompLinkage}; by merging the first~$i$ bags for some suitable integer $i$, we may assume that this linear decomposition has length exactly~$M_1(\ell, M+2) \le M_1(k, M+2)$.%
    \COMMENT{Here, we use the sentence after~\cref{thm:lemma2}, which asserts that~$M_1(\ell, M+2) \le M_1(k, M+2)$ for~$\ell \le k$.}
    By~\cref{thm:lemma2} ,there is a linear-decomposition~$\cW''$ of~$G$ of length~$M' \ge M + 2$ which additionally has a foundational linkage~$\cP''$ satisfying~\cref{itm:FoundationalLinkage1} and~\cref{itm:FoundationalLinkage2}.
    We note that~$\ell > 0$ since~$G$ is a connected graph and~$M' \ge 1$.

    Consider the adhesion sets~$U_i''$ of~$\cW''$ and let~$Z' := \bigcap_{i \in \{1, \dots, M'\}} U_i''$.
    By~\cref{itm:FoundationalLinkage1}, the set~$Z'$ consists precisely of all trivial paths in~$\cP''$.
    Now $\cW''$ induces the linear decomposition~$\cW'$ of~$G - Z$ by setting~$W_i' := W_i'' \setminus Z'$ together with its foundational linkage~$\cP' \subseteq \cP''$.
    Now the adhesion $\ell'$ of~$\cW'$ satisfies~$\ell = \ell' + |Z'| > 0$; in particular, we might have~$\ell' = 0$. 
    Clearly, $\cP''$ again satisfies~\cref{itm:FoundationalLinkage1} and~\cref{itm:FoundationalLinkage2}.
    Moreover, $\cW'$ still satisfies~\cref{itm:RainbowDecompLinkage}, and it now also satisfies~\cref{itm:RainbowDecompDistinctBags} by~\cref{itm:FoundationalLinkage1}.
    Finally, \cref{itm:FoundationalLinkage2} implies that if~$z' \in Z'$ is adjacent to some~$W_i$ in~$G$, then it is adjacent to all~$W_i$.

    We now define an auxiliary graph~$H_{\cP'}$ with vertex set~$\cP'$ and an edge joining two distinct paths~$P, P' \in \cP'$ if there exists a path in some, and by~\cref{itm:FoundationalLinkage2} hence every, $G[W_i']$ with one endvertex in~$P$ and the other in~$P'$ and whose internal vertices avoid every path in $\cP'$.
    Let~$C_H$ be an arbitrary component of~$H_{\cP'}$.
    For~$i \in \{1, \dots, M' - 1 \} = [M'-1]$ we then let~$W_i$ be the vertex set of the  component of~$G[W_i']$ which contains all~$V(P) \cap W_i'$ for~$P \in C_H$; note that this is well-defined by the construction of~$H_{\cP'}$.
    If~$C_H$, and thus~$H_{\cP'}$, is empty, we let~$Z := Z'$; otherwise, we let~$Z$ consist of all those~$z \in Z'$ which are adjacent to some~$W_i$, and hence all~$W_i$ by~\cref{itm:FoundationalLinkage2}.
    We further set
    \begin{equation*}
        R := G[\bigcup_{i \in [M'-1]} W_i]
    \end{equation*}
    and
    \begin{equation*}
        C := G[Z' \cup \bigcup_{i \in [M'-1]} (W_i' \setminus W_i) \cup W_0' \cup W_{M'}'].
    \end{equation*}
    This ensures that~$Z \subseteq V(C)$ as well as~$G = G[V(R) \cup Z] \cup C$, as we chose the~$W_i$ as components of the~$G[W_i']$.
    We now claim that~$(R, \cW, Z, C)$ where~$\cW := (W_1, \dots, W_{M' - 1})$ is the desired rainbow-cloud-decomposition of~$G$.

    Clearly, $\cW$ is a linear decomposition of~$R$ of adhesion~$|C_H|$ and length~$M' - 2 \ge M$; note that~$|C_H| > 0$ if~$\ell' > 0$ and hence~$|Z| + |C_H| > 0$ by the choice of~$Z$.
    We now verify that~$\cW$ is even a rainbow-decomposition of~$R$.
    The foundational linkage~$\cP'$ of~$\cW'$ induces a foundational linkage of~$\cW$, as we have chosen a component~$C_H$ of~$H_{\cP'}$; thus, \cref{itm:RainbowDecompLinkage} holds.
    For~\cref{itm:RainbowDecompConnected}, note that the~$W_i$ are components of the~$G[W_i']$ and hence connected by construction.
    Finally, \cref{itm:RainbowDecompDistinctBags} transfers from~$\cW'$.

    It remains to check~\labelcref{itm:RCDecompRcapC,itm:RCDecompEdgesToZ,itm:RCDecompAdditionalCutsets,itm:RCDecompAdditionalLinkages}.
    So let~$U_1 := V(C) \cap W_1$ and~$U_{M'} := V(C) \cap W_{M' - 1}$; note that our construction implies~$U_1 = W_0' \cap W_1$ and~$U_{M'} = W_{M'}' \cap W_{M' - 1}$.
    By definition, we have~$V(R) \cap V(C) = (W_1 \cup W_{M' - 1}) \cap V(C) = U_1 \cup U_{M'}$, so~\cref{itm:RCDecompRcapC} holds, and~\cref{itm:RainbowDecompDistinctBags} for~$\cW'$ implies~\cref{itm:RCDecompAdditionalCutsets} as well as~\cref{itm:RainbowDecompLinkage} implies~\cref{itm:RCDecompAdditionalLinkages}.
    For~\cref{itm:RCDecompEdgesToZ}, we recall that this holds by the definition of~$Z$.
    This completes the proof.
\end{proof}

\section{RC-decompositions and separations} \label{sec:RCDecompAndSeps}

In this section, we investigate how a fixed RC-decomposition of a graph~$G$ interacts with the separations of~$G$.
On the one hand, we will analyse in what ways a separation of~$G$ may meet the rainbow of an RC-decomposition.
On the other hand, we look at the separations of $G$ induced by the structure of the RC-decomposition.
With this we build a set of tools that we will later apply in the proof of our structural main result, \cref{main:SplitterTheorem}. 
These tools will allow us to control the new low-order separations that arise when we delete an edge deep inside the rainbow.

\subsection{Separations and bags}

In this subsection, we prove two general lemmas that describe in which ways a separation of a graph meets the bags of a linear decomposition or an RC-decomposition.
The first lemma asserts that if a linear decomposition satisfies~\cref{itm:RainbowDecompConnected} and~\cref{itm:RainbowDecompDistinctBags}, then the strict sides of a separation contain most of its bags.

\begin{lemma}\label{prop:SomePartLiesOnAStrictSide}
    Let~$\{A,B\}$ be a separation of a graph~$G$ of order~$k$ and let~$\cW$ be a linear decomposition of a subgraph $R$ of~$G$ that satisfies \cref{itm:RainbowDecompDistinctBags}. 
    Then the separator~$A \cap B$ meets at most $2k$ bags of~$\cW$.
    Moreover, if~$\cW$ additionally satisfies \cref{itm:RainbowDecompConnected}, then at most~$2k$ bags of~$\cW$ are not contained in either $A \setminus B$ or $B \setminus A$.
\end{lemma}
\begin{proof}
    Each vertex in~$R$ is contained in at most one adhesion set of~$\cW$ by~\cref{itm:RainbowDecompDistinctBags} and thus in at most two bags of~$\cW$.
    So since $\{A, B\}$ has order~$k$, the separator $A \cap B$ meets at most $2k$ bags of $\cW$. 
    As each part~$R[W_i]$ of~$\cW$ is connected by~\cref{itm:RainbowDecompConnected}, every bag of~$\cW$ that is disjoint from $A \cap B$ is included in precisely one of $A \setminus B$ and $B \setminus A$.
    Hence, at most $2k$ bags of~$\cW$ are not contained in either $A \setminus B$ or $B \setminus A$.
\end{proof}

The second lemma states that if a strict side of a separation has a component which meets both the cloud and a bag~$W_i$ of a fixed RC-decomposition, then it also contains most bags of the rainbow along one of the two connections from~$W_i$ to the cloud.

\begin{lemma}\label{lem:SepsThatMeetRAndCContainAllBagsOnTheirSide}
    Let~$(R, \cW, Z, C)$ be an RC-decomposition of a graph~$G$ of length~$M$, and let~$\{A, B\}$ be a separation of~$G$ of order~$k$.
    Suppose that~$D \subseteq G[A \setminus B]$ is a component of~$G - (A \cap B)$ which meets both~$C$ and some bag~$W_i$.
    Then~$D$ contains all but at most~$2k$ of~$W_0, \dots, W_i$ or~$W_i, \dots, W_M$.
    Moreover, $D$ meets either~$Z$ or every bag of~$W_0, \dots, W_i$ or of~$W_i, \dots, W_M$, respectively.
\end{lemma}

\begin{proof}
    By \cref{prop:SomePartLiesOnAStrictSide} applied to the restriction $\{A \cap V(R), B \cap V(R)\}$ of $\{A,B\}$ to the rainbow $R$ at most $2k$ of the bags $W_i$ meet $A \cap B$.
    If $D$ contains a vertex $z$ of $Z$, then every bag which does not meet $A \cap B$ is contained in $D$ by \cref{itm:RCDecompEdgesToZ} and \cref{itm:RainbowDecompConnected} because $D$ is a component of $G - (A \cap B)$, as desired. 
    So let us now assume that $V(D)$ is disjoint from $Z$.
    Since $D$ is connected and meets both $C$ and $W_i$, the component $D$ contains a path~$P$ from $W_i$ to $C$. 
    This path $P$ has to first enter $C$ either through $U_0$ or~$U_{M+1}$,
    as $V(D) \cap Z = \emptyset$ and also $G[V(R) \cup Z] \cup C = G$ and $V(R) \cap V(C) = U_0 \cup U_{M+1}$  by the definition of an RC-decomposition. 
    Thus, we may assume that the only vertex of $P$ in $C$ is its endvertex which is in~$U_0$, as the other case is symmetrical.
    Hence, it meets all bags $W_j$ with $0 \le j \leq i$, because~$\cW$ is a linear decomposition of $R$.
    Thus, every bag~$W_j$ with $0 \leq j \leq i$ which avoids $A \cap B$ is contained in the component~$D$ of $G- (A \cap B)$ by~\cref{itm:RainbowDecompConnected}, as desired.
\end{proof}

\subsection{Separating along the rainbow}

We can easily turn an RC-decomposition~$(R, \cW, Z, C)$ of a graph~$G$ of length~$M$ into a new one of shorter length and of the same adhesion by restricting the linear decomposition~$\cW$ of~$R$ to some interval in~$\{0, \dots, M\}$ and adding the remaining bags to~$C$ as follows.
For $0 \leq i \leq j \leq M$, we set 
\begin{align*}
    R_{i,j} &:= G\left[ \bigcup_{h = i}^j W_h\right] = \bigcup_{h = i}^j G[W_h],\\
    \cW_{i,j} &:= (W_i, \dots, W_j),\\ 
    C_{i,j} &:= G\left[V(C) \cup \left(\bigcup_{h = 0}^{i-1} W_h\right) \cup \left(\bigcup_{h = j+1}^M W_h\right)\right]\!, \text{ and}\\
    (R, \cW, Z, C)_{i,j} &:= (R_{i,j}, \cW_{i,j}, Z, C_{i,j}).
\end{align*}
\noindent We remark that we follow the convention that an empty union, such as $\bigcup_{h = 0}^{-1} W_h$ or $\bigcup_{h=M+1}^M W_h$, is the empty set.

It is straight-forward from the definition that~$(R, \cW, Z, C)_{i,j}$ is again an~$RC$-decomposition of~$G$:

\begin{lemma}\label{prop:ShorteningTheRainbow}
    Let $G$ be a graph with an RC-decomposition $(R, \cW, Z, C)$ of length $M$ and of adhesion $\ell$.
    Then $(R, \cW, Z, C)_{i,j}$ is an RC-decomposition of $G$ of length~$j - i$ and of adhesion $\ell$ for all~$0 \leq i \leq j \leq M$. \qed
\end{lemma}

It follows immediately that we obtain separations of $G$ along the rainbow of a fixed RC-decomposition:

\begin{lemma}\label{prop:SlicesOfRAreSepsOfG}
    Let $G$ be a graph with an RC-decomposition $(R, \cW, Z, C)$ of length~$M$ and adhesion~$\ell$, and let $0 \le i \leq j \le M$.
    Then $\{V(R_{i,j}) \cup Z, V(C_{i,j})\}$ is a separation of $G$ with separator~$U_i \cup U_{j+1} \cup Z$. In particular, its order is~$2\ell + \abs{Z}$.
    \qed
\end{lemma}

\subsection{Rainbow-crossing separations}

We now investigate separations that `cross' the rainbow of a given RC-decomposition in that both strict sides of the separation contain bags of the linear decomposition of the rainbow:
Let~$(R, \cW, Z, C)$ be an~RC-decomposition of a graph~$G$ of length~$M$.
    An oriented separation~$(A,B)$ of~$G$ of order~$k$ \defn{crosses the rainbow~$R$ clockwise} if there exist integers $i \in \{0, \dots, 2k\}$ and $j \in \{M-2k, \dots,M\}$ such that $W_i \subseteq A \setminus B$ and $W_j \subseteq B \setminus A$.
    We denote the minimal such $i$ by $i_{A,B}$ and the maximal such $j$ by $j_{A,B}$.
    If an oriented separation $(A,B)$ crosses the rainbow clockwise, we say that its other orientation $(B,A)$ \defn{crosses the rainbow~$R$ counter-clockwise} and the underlying unoriented separation $\{A,B\}$ \defn{crosses the rainbow~$R$}.
We remark that a separation $\{A,B\}$ which crosses the rainbow contains $Z$ in its separator $A \cap B$ by \cref{itm:RCDecompEdgesToZ}, $W_i \subseteq A \setminus B$ and $W_j \subseteq B \setminus A$.
Note that in the above definition, we may have~$j \le i$, since we do not assume any lower bound on~$M$.
In most applications, however, we will have~$M \ge 4k$ and thus~$i < j$.

Given a separation which crosses the rainbow, we may construct several separations that separate the cloud in the same way, but split the rainbow along the adhesion sets of its linear decomposition (see \cref{fig:RainbowCrossingSeps}):
\begin{figure}[ht]
    \centering
    \pdfOrNot{\includegraphics[page=2]{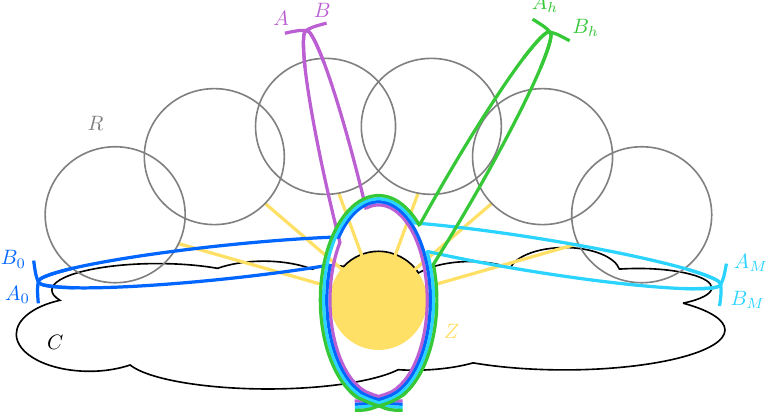}}{\includesvg[width=0.8\columnwidth]{svg/RCdecompositionschematic2seps.svg}}
    \caption{A rainbow-crossing separation $(A,B)$ and the arising separations $(A^0, B^0)$, $(A^h, B^h)$ and $(A^M, B^M)$.}
    \label{fig:RainbowCrossingSeps}
\end{figure}
More precisely, let $(R, \cW, Z, C)$ be an RC-decomposition of a graph~$G$.
For a separation~$(A, B)$ of~$G$ of order~$k$ that crosses the rainbow clockwise, we let $i := i_{A,B} \in \{0, \dots, 2k\}$ be the smallest and $j := j_{A,B} \in \{M-2k, \dots, M\}$ the largest number such that $W_{i} \subseteq A\setminus B$ and $W_{j} \subseteq B\setminus A$, which are well-defined by the definition of crossing the rainbow clockwise. 
For $h \in \{i+1, \dots, j\}$ we set 
\begin{equation*}
    (A^h, B^h) := ((A \cap V(C_{i,j})) \cup V(R_{i,h-1}),\;  V(R_{h,j}) \cup (B \cap V(C_{i,j}))).
\end{equation*}
If $(A,B)$ crosses the rainbow counter-clockwise, we use the fact that $(B,A)$ crosses the rainbow clockwise to define $(A^h, B^h)$ accordingly.

It is immediate from the definition that $(A^{i+1}, B^{i+1}) \leq (A,B) \leq (A^{j}, B^{j})$ and that the $(A^h,B^h)$ form a strictly increasing sequence.
Indeed, the $(A^h, B^h)$ also are separations of~$G$:

\begin{lemma} \label{prop:ObtainingNewSepsFromRainbowCrossingSeps}
    Let~$(R, \cW, Z, C)$ be an~RC-decomposition of a graph~$G$ of adhesion~$\ell$, and let~$(A,B)$ be a separation of~$G$ of order~$k$ which crosses the rainbow~$R$ clockwise. 
    Then $(A^h, B^h)$ is a separation of $G$ of order $k$ which crosses the rainbow $R$ clockwise for every~$h \in \{i_{A,B} + 1, \dots, j_{A,B}\}$.
\end{lemma}
\begin{proof}
    We abbreviate~$i := i_{A,B}$ and~$j := j_{A,B}$.
    As $C_{i,j}$ is a subgraph of~$G$, the separation~$\{A,B\}$ of~$G$ induces the separation $\{A \cap V(C_{i,j}), B \cap V(C_{i,j})\}$ of $C_{i,j}$.
    Note that~$Z \subseteq A \cap B$, since~$Z$ is contained in the neighbourhood of both~$W_i \subseteq A \setminus B$ and~$W_j \subseteq B \setminus A$ by~\cref{itm:RCDecompEdgesToZ}.
    Moreover, by~\cref{prop:SlicesOfRAreSepsOfG}, both $U_i \cup Z \cup U_h$ and $U_h \cup Z \cup U_{j+1}$ separate $R_{i,h-1}$ and $R_{h,j}$, respectively, from the rest of $G$.
    As $U_i \cup Z \cup U_h \subseteq A^h$ and $U_h \cup Z \cup U_{j+1} \subseteq B^h$, it follows that $\{A^h, B^h\}$ is a separation of $G$.%
    
    It remains to show that $\abs{A^h \cap B^h} \le k$.
    The foundational linkage of the RC-decomposition induces a $U_{i+1}$--$U_{j}$ linkage of cardinality $\ell$ in $R_{i+1,j-1}$.
    Since~$\{A, B\}$ is a separation of~$G$ and~$U_{i+1} \subseteq W_i \subseteq A \setminus B$ and~$U_j \subseteq W_j \subseteq B \setminus A$, the separator $A \cap B$ has to contain one vertex of every $U_{i+1}$--$U_j$~paths.
    But since $V(C_{i,j}) \cap V(R_{i+1, j-1}) = \emptyset$, these $\ell$~vertices are contained in~$(A \cap B) \setminus V(C_{i,j})$. 
    Hence, we obtain~$\abs{(A \cap B) \cap V(C_{i,j})} \leq \abs{A \cap B} - \ell = k - \ell$.
    With this, we may now  bound the order of $\{A^h,B^h\}$:
    \begin{align*}
        &\abs{A^h \cap B^h} 
            = \abs{((A \cap V(C_{i,j})) \cup V(R_{i,h-1})) \cap (V(R_{h,j}) \cup (B \cap V(C_{i,j})))} \\
        &\leq \abs{(A \cap B) \cap V(C_{i,j})} + \abs{A \cap (V(C_{i,j}) \cap V(R_{h,j}))} + \abs{B \cap (V(C_{i,j}) \cap V(R_{i,h-1}))} + \abs{V(R_{i,h-1}) \cap V(R_{h,j})} \\
        &\le (k-\ell) + \abs{A \cap W_j} + \abs{B \cap W_i} + \abs{U_h} =  (k - \ell) + 0 + 0 + \ell = k,
    \end{align*}
    where the penultimate equation holds
    because $W_i \subseteq A\setminus B$ and $W_j \subseteq B\setminus A$.
    This completes the proof.
\end{proof}

\subsection{Rainbow-slicing separations}
In this section, we study separations that `slice' the rainbow in that they cut bags in the middle of the rainbow away:

Let~$(R, \cW, Z, C)$ be an RC-decomposition of a graph~$G$ of length~$M$.
A separation~$\{A,B\}$ of~$G$ of order~$k$ \defn{slices the rainbow~$R$} if there are integers~$i < h < j$ with~$i \in \{0, \dots, 2k\}$ and $j \in \{M-2k, \dots, M\}$ such that $W_i, W_j \subseteq A\setminus B$ and $W_h \subseteq B\setminus A$ or such that $W_i, W_j \subseteq B\setminus A$ and $W_h \subseteq A \setminus B$.

We now prove a lower bound on the order of separations that slice the rainbow.
Towards this, we first show such a lower bound for a slightly more general class of separations for later use, which in particular yields the desired bound for rainbow-slicing separations. 
\begin{lemma} \label{prop:ConditionForACapBGeq2lPlust}
    Let $G$ be a graph with an RC-decomposition~$(R,\cW,Z,C)$ of adhesion~$\ell$, and let $\{A,B\}$ be a separation of~$G$.
    Suppose that there are integers $0 \leq i < h < j \leq M+1$ such that~$B\setminus A$ contains~$W_h$ and~$A$ contains the adhesion sets~$U_i$ and~$U_j$.
    Then $\abs{A \cap B \cap V(R)} \geq 2 \ell$. Moreover, if additionally $Z \subseteq A$, then $\abs{A \cap B} \geq 2\ell + \abs{Z}$.
\end{lemma}
\begin{proof}
    The fundamental linkage induces a $U_i$--$U_h$ linkage $\cP_{i}$ and a $U_h$--$U_j$ linkage $\cP_{j}$; both have cardinality $\ell$ and are in $R$.
    We remark that the paths in $\cP_i$ meet the paths in $\cP_j$ only in $U_h$.
    Hence, $W_h \subseteq B\setminus A$ and~$U_i, U_j \subseteq A$ yields that the separator~$A \cap B$ contains at least one vertex of each of these~$2 \ell$ paths.
    As every two paths in $\cP_i$ and also every two paths in $\cP_j$ are pairwise disjoint, we thus have $\abs{A \cap B \cap V(R)} \geq 2\ell$.

    If additionally~$Z \subseteq A$, then~$Z \subseteq A \cap B$ since~$Z \subseteq N_G(W_h)$ by~\ref{itm:RCDecompEdgesToZ} and~$W_h \subseteq B \setminus A$.
    As all the above~$2 \ell$ paths lie in~$R$ and hence avoid~$Z$, the claim follows.
\end{proof}

Since every rainbow-slicing separation contains~$Z$ in its separator by~\cref{itm:RCDecompEdgesToZ}, we obtain the following lower bound on their order:
\begin{corollary} \label{lem:OrderOfSlicingSeps}
    Let $G$ be a graph with an RC-decomposition~$(R,\cW,Z,C)$ of adhesion~$\ell$.
    If a separation $\{A,B\}$ of $G$ slices the rainbow~$R$, then $\abs{A \cap B \cap V(R)} \geq 2\ell$ and $\abs{A \cap B} \geq 2\ell + |Z|$. \qed
\end{corollary}

We conclude this section of preparatory lemmas by showing that every separation which separates two bags of an RC-decomposition either crosses or slices its rainbow.
\begin{lemma}\label{lem:CrossingOrSlices}
    Let $G$ be a graph with an RC-decomposition~$(R,\cW,Z,C)$ of length~$M$, and let~$\{A,B\}$ be a separation of~$G$. 
    If there are bags~$W_i$ and~$W_j$ of~$\cW$ such that $W_i \subseteq A\setminus B$ and~$W_j \subseteq B \setminus A$, then $\{A,B\}$ either crosses or slices the rainbow.
\end{lemma}
\begin{proof}
    Let~$k := \abs{A, B}$.
    If $M \leq 2k$, then $W_i$ and $W_j$ witness that $\{A,B\}$ crosses the rainbow. 
    So we may assume $M \geq 2k$.
    By~\cref{prop:SomePartLiesOnAStrictSide}, all but at most~$2k$ bags of~$\cW$ are contained in either~$A \setminus B$ or~$B \setminus A$.
    In particular, there exist bags~$W_h$ with~$h \in \{0, \dots, 2k\}$ and~$W_s$ with~$s \in \{M-2k, \dots, M\}$ that are each contained in either~$A \setminus B$ or~$B \setminus A$.
    
    If one of~$W_h$ and~$W_s$ is contained in $A\setminus B$ and the other one in $B\setminus A$, then $\{A,B\}$ crosses the rainbow.
    Otherwise, both~$W_h$ and~$W_s$ are contained in the same side of $\{A,B\}$.
    By symmetry, we may assume $W_h, W_s \subseteq A \setminus B$.
    Now if~$j \in \{0, \dots, 2k\}$, then $W_j$ and~$W_s$ witness that $(B,A)$ crosses the rainbow clockwise, and if $j \in \{M-2k, \dots, M\}$, then $W_h$ and~$W_j$ witness that $(A,B)$ crosses the rainbow clockwise.
    Otherwise, $j \in \{2k+1, \dots, M-2k-1\}$; in particular~$h < j < s$.
    Therefore~$\{A,B\}$ slices the rainbow. 
\end{proof}

\section{Proof of \texorpdfstring{\cref{main:SplitterTheorem}}{Theorem 5} if \texorpdfstring{$G$}{G} has no higher-order tangle} \label{sec:DeletingTheEdge}
In this section, we complete the proof of~\cref{main:SplitterTheorem}.
Based on the results in~\cref{sec:SmallSpecialCases,sec:ExistenceOfHighOrderTangle,sec:ExistenceRCDecomp}, it suffices to consider graphs that admit an RC-decomposition with certain properties.
More precisely, the technical main result of this section, which will allow us to finish the proof of~\cref{main:SplitterTheorem}, reads as follows:

\begin{theorem}\label{thm:AtLeastOneRainbowPath}
    Let $k \ge 1$ be an integer, let $G$ be a graph of minimum degree at least $3$, and let~$\tau$ be a $k$-tangle in~$G$.
    Suppose that~$G$ admits an RC-decomposition with sun~$Z$ which has length $ \geq 18k$ and adhesion~$\ell$ such that $\ell + \abs{Z} \geq 1$. 
    Then there exists an edge $e \in E(G)$ such that~$\tau$ extends to a $k$-tangle $\tau'$ in~$G - e$. 
\end{theorem}

\noindent Most of this section is devoted to the proof of~\cref{thm:AtLeastOneRainbowPath}. In the very end we combine the previous results to prove~\cref{main:SplitterTheorem}.

\subsection{Living in the rainbow}

In the proof of~\cref{thm:AtLeastOneRainbowPath}, we will delete an edge~$e$ deep inside the rainbow~$R$ of the given long RC-decomposition.
With this, we aim to make use of the regular structure of~$R$ to orient the newly arising separations in~$G' := G - e$ in such a way that we find a~$k$-tangle~$\tau'$ in~$G'$ to which~$\tau$ extends.
As we want to orient these new separations of~$G'$ in line with~$\tau$, the tangle $\tau$ should ideally be `regular' or `monotonic' along~$R$.
Then we could use this monotonicity of~$\tau$ along~$R$ to orient the newly arising separations in~$G'$ by consistency.

In this section, we extract from a given RC-decomposition another RC-decomposition such that $\tau$ has the desired monotonic behaviour along the rainbow.
More precisely, the tangle~$\tau$ will `point away' from the rainbow, which will turn out to be equivalent to the desired monotonicity, as we see below.
Formally, we enclose this in a definition of when a tangle~$\tau$ does not `live in the rainbow~$R$' of an RC-decomposition.
The intuition behind this definition is as follows.

Clearly, a tangle~$\tau$ should live in the rainbow~$R$, if there is a separation~$(A, B) \in \tau$ whose strict big side~$B \setminus A$ is contained in the rainbow $R$ and does not meet the cloud $C$.
For our proof of~\cref{thm:AtLeastOneRainbowPath}, this case is not the only relevant one:
we also need to regard~$\tau$ as living in~$R$ if it orients two separations in such a way that they point towards each other and to a piece of~$R$.
It turns out that for this second case, we can even restrict our attention to rainbow-crossing separations, as follows.

Let $\tau$ be a $k$-tangle in a graph $G$.
Consider a given RC-decomposition~$(R, \cW, Z, C)$ of~$G$ and a separation~$(A, B)$ of~$G$ of order~$k' < k$ that crosses the rainbow~$R$ clockwise or counter-clockwise.
Then $\tau$ \defn{orients~$\{A, B\}$ monotonically over the rainbow} if~$\tau$ orients all the $\{A^h, B^h\}$ with $h \in \{i_{A,B} +1, \dots, j_{A,B}\}$ such that they are pairwise comparable, i.e.\,either all  as $(A^h, B^h)$ or all as $(B^h, A^h)$.
We remark that if $\tau$ does not orient a rainbow-crossing separation $\{A,B\}$ of order less than $k$ monotonically over the rainbow $R$, then the consistency of $\tau$ ensures that there is a (unique) index $h^* \in \{0, \dots, M-1\}$ such that $(A^{h^*}, B^{h^*}) \in \tau$ and $(B^{h^*+1}, A^{h^*+1}) \in \tau$, as the $(A^h,B^h)$ form an increasing or decreasing sequence as $(A,B)$ crosses the rainbow clockwise or counter-clockwise, respectively.
We then call ${h^*}$ its \defn{turning point}; note that~$h^* \in \{0, \dots, M-1\}$.
Moreover, we say that~$\tau$ \defn{lives in the rainbow $R$} if at least one of the following holds:
    \begin{enumerate}[label=(LR\arabic*)]
        \item \label{itm:TauLivesInRBecauseOf1Sep} there is a separation $(A,B) \in \tau$ such that $B\setminus A \subseteq V(R) \setminus V(C)$, or
        \item \label{itm:TauLivesInRBecauseOf2Seps} $\tau$ orients at least one rainbow-crossing separation of order less than~$k$ not monotonically over the rainbow.
    \end{enumerate}

The next theorem asserts that we can shorten an RC-decomposition of sufficient length so that a given tangle does not live in the shortened rainbow:
\begin{theorem} \label{thm:ExistenceOfAShorterRainbowThatDoesNotContainTau}
    Let $k\ge 1$, $M \geq 6k$ and $\ell \ge 0$ be integers, and let $\tau$ be a $k$-tangle in a graph $G$. 
    If $G$ admits an RC-decomposition $(R, \cW, Z, C)$ of $G$ of length at least $M$ and of adhesion $\ell$, then there exist~$0 \le i < j \le M$ with~$j - i \ge M/2-k$ such that $\tau$ does not live in the rainbow of the RC-decomposition~$(R, \cW, Z, C)_{i,j}$.
\end{theorem}

\begin{proof}
    If $\tau$ does not live in $R$, then we are done by setting $i := 0$ and $j := M$; so suppose that $\tau$ lives in $R$.
    By definition, there are two possible reasons for that, \cref{itm:TauLivesInRBecauseOf1Sep} and~\cref{itm:TauLivesInRBecauseOf2Seps}, which we will treat separately.

    First, suppose that~$\tau$ lives in~$R$ because of \cref{itm:TauLivesInRBecauseOf1Sep}, i.e.\,there is a separation $(A, B) \in \tau$ with $B \setminus A \subseteq V(R\setminus C)$.
    Then there also exists such a separation in $\tau$ with an even stronger property:
    \begin{sublemma} \label{lem:StrengtheningOfCondition1ThatTauLivesInR}
        There is a separation~$(X,Y) \in \tau$ such that~$Y \setminus X \subseteq \bigcup_{t = r}^s W_t$ for integers~$r \leq s$ with~$s - r < 2k-2$. 
    \end{sublemma}
    \begin{proof}
        First, suppose that $2\ell + \abs{Z} < k$.
        Then $\{V(C), V(R) \cup Z\}$ is a separation of $G$ of order~$2\ell + |Z| < k$ and hence the $k$-tangle $\tau$ has to contain one of its orientations.
        By the consistency of the tangle $\tau$, $(V(C), V(R) \cup Z) \leq (A, B) \in \tau$ yields $(V(C), V(R) \cup Z) \in \tau$. 
        Every `slice' $\{V(R_{h,h}) \cup Z, V(C_{h,h})\}$ of the rainbow~$R$ with~$h \in \{0, \dots, M\}$ is also a separation of $G$ of order~$2\ell + |Z| < k$ by~\cref{prop:SlicesOfRAreSepsOfG}.
        
        Suppose for a contradiction that they are all oriented as~$(V(R_{h,h}) \cup Z, V(C_{h,h}))$ by~$\tau$.
        By iteratively using the fact that the tangle $\tau$ avoids the triples in $\cT$ and for every $0\leq h \leq i \leq j \leq M$ the separations $(V(R_{h,i} \cup Z, V(C_{h,i}))$, $(V(R_{i,j}) \cup Z, V(C_{i,j}))$ and $(V(C_{h,j}), V(R_{h,j} \cup Z)\}$ of order $2 \ell + |Z|$ form a triple in $\cT$, we obtain that~$(V(R_{0,M} \cup Z, V(C_{0,M})) = (V(R) \cup Z, V(C))$ is in~$\tau$, which contradicts that its other orientation is also in $\tau$.
        Therefore, there exists $h \in \{0,\dots,M\}$ with $(V(C) \cup (V(R) \setminus W_h) \cup U_{h} \cup U_{h+1}, W_h \cup Z) \in \tau$, which is the desired separation~$(X, Y)$ with $r := h =: s$.
        
        Secondly, suppose that $2\ell + \abs{Z} \geq k$, and let $(A,B) \in \tau$.
        We may assume that the separation $(A,B) \in \tau$ with $B \setminus A \subseteq V(R) \setminus V(C)$ is $\leq$-maximal in $\tau$ with that property; in particular, $G[B \setminus A]$ is connected.
        Then $B \setminus A$ cannot contain a bag of $\cW$ by~\cref{prop:ConditionForACapBGeq2lPlust}, since $B \setminus A \subseteq V(R) \setminus V(C)$ yields that $U_0, U_{M+1}, Z \subseteq V(C) \subseteq A$, and~$\{A, B\}$ has order~$< k \leq 2\ell + \abs{Z}$.
        Thus, every bag which meets $B \setminus A$ also meets the separator $A \cap B$.
        Thus, \cref{prop:SomePartLiesOnAStrictSide} ensures that at most~$2k-2$ bags of~$\cW$ meet~$B \setminus A$.
        Further, since~$G[B \setminus A] \subseteq R$ is connected and each bag $W_i$ separates $R$, the bags which are met by $B \setminus A$ are consecutive bags of~$\cW$.
        Thus, $(X, Y) := (A, B)$ is the desired separation, as witnessed by the respective indices~$r$ and~$s$ of the first and last bag of~$\cW$ which are met by~$B\setminus A$.
    \end{proof}

    Let~$(X, Y) \in \tau$ and~$r \leq s$ with $s - r < 2k-2$ be given by~\cref{lem:StrengtheningOfCondition1ThatTauLivesInR}.
    Now if $r > \frac{M}{2}-k$, then we set~$i := 0$ and~$j := r-1$, and otherwise, if $r \leq \frac{M}{2}-k$, then we set~$i := s+1$ and~$j := M$.
    We remark that in the latter case $s < \frac{M}{2} + k$.
    By~\cref{prop:ShorteningTheRainbow}, $(R, \cW, Z, C)_{i,j}$ is an RC-decomposition of~$G$ of adhesion~$\ell$ and length~$j - i \geq \frac{M}{2}-k$. 
    We claim that~$\tau$ does not live in~$R_{i,j}$, as desired.
    
    Since $Y \setminus X \subseteq \bigcup_{t = r}^s W_t$ and therefore $Y\setminus X$ is disjoint from $R_{i,j}$ by definition of $i,j$, we have~$V(R_{i,j}) \subseteq X$. 
    Thus, every two separations~$(X_1, Y_1), (X_2, Y_2)$ of~$G$ with $(Y_1 \cap Y_2) \setminus (X_1 \cup X_2) \subseteq V(R_{i,j})$ form a forbidden triple together with~$(X,Y)$.
    But every separation as in \cref{itm:TauLivesInRBecauseOf1Sep} (taken as both~$(X_i, Y_i)$) as well as every pair of separations witnessing the non-monotonicity in~\cref{itm:TauLivesInRBecauseOf2Seps} are such two separations.
    Therefore, $\tau$ cannot live in $R_{i,j}$, as desired. 
    
    Secondly, suppose that $\tau$ does live in the rainbow~$R$ because of~\cref{itm:TauLivesInRBecauseOf2Seps}, i.e.\,there is a separation of $G$ of order~$<k$ that crosses the rainbow $R$ and $\tau$ does not orient it monotonically over the rainbow $R$.
    It turns out that all such separations $\{A,B\}$ have the same turning point: 
    \begin{sublemma}\label{lem:AboutTheTurningPointsOfRainbowCrossingSeps}
        If a rainbow-crossing separation $\{A,B\}$ which $\tau$ does not orient monotonically over the rainbow~$R$ has turning point~$h^*$, then the turning point of every such separation is ${h^*}$.
    \end{sublemma}
    \begin{proof}
        Suppose for a contradiction that there is another rainbow-crossing separation~$\{X,Y\}$ which $\tau$ does not orient monotonically over the rainbow~$R$ and has turning point~$h' \neq h^*$. 
        Let $(A,B)$ and $(X,Y)$ be the orientations which cross the rainbow clockwise.
        By possibly interchanging $\{A,B\}$ and $\{X,Y\}$, we may assume~$h^* < h'$.
        Now the definition of the $(X^h,Y^h)$ immediately yields that~$W_{h^*} \subseteq X^{h'}$.
        But it also guarantees that $G = G[A^{h^*}] \cup G[B^{h^*+1}] \cup G[W_{h^*}]$.
        Thus $(X^{h'},Y^{h'})$ together with $(A^{h^*}, B^{h^*})$ and $(B^{h^*+1}, A^{h^*+1})$ forms a forbidden triple in $\tau$, which is a contradiction.
    \end{proof}

    If ${h^*} \ge M/2$, then we set $i = 0$ and~$j = {h^*}-1$, and if ${h^*} < M/2$, then we set $i = {h^*}+1$ and~$j = M$.
    By~\cref{prop:ShorteningTheRainbow}, $(R, \cW, Z, C)_{i,j}$ then is an RC-decomposition of~$G$ of length~$j - i \geq M/2 - 1$ and adhesion~$\ell$.
    We claim that~$\tau$ does not live in~$R_{i,j}$.
    Since~$\tau$ does not live in~$R$ because of~\cref{itm:TauLivesInRBecauseOf1Sep} by assumption, $\tau$ does neither live in~$R_{i,j} \subseteq R$ because of~\cref{itm:TauLivesInRBecauseOf1Sep}, either.
    Moreover, $\tau$ can also not live in $R_{i,j}$ because of~\cref{itm:TauLivesInRBecauseOf2Seps}: 
    Indeed, every separation that crosses~$R_{i,j}$ also crosses~$R$, and if $\tau$ does not orient it monotonically over~$R_{i,j}$, it does not do so over~$R$, as well.
    But since all such separations have turning point~${h^*}$ with respect to~$R$ by~\cref{lem:AboutTheTurningPointsOfRainbowCrossingSeps}, the choice of~$i$ and~$j$ ensure that $\tau$ orients them all monotonically over~$R_{i,j}$ by the choice of~$i$ and~$j$, as desired.
\end{proof}

\subsection{Deleting an edge and orienting the new separations} \label{subsec:ProofOfAtLeastOneRainbowPathThm}
This subsection is dedicated to the proof of~\cref{thm:AtLeastOneRainbowPath}.
We remark that this section is the only part of the proof of \cref{main:SplitterTheorem} for graphs~$G$ which do not have a $(k+1)$-tangle in which we make use of the assumption that $G$ has no vertex of degree $\leq 2$.
In the remainder of this section, we will always assume that the given graph satisfies the premise of \cref{thm:AtLeastOneRainbowPath}:

\begin{setting}\label{setting}
    Let $G$ be a graph of minimum degree at least $3$ that admits an RC-decomposition with sun $Z$ which has length $M_0$ and adhesion $\ell$ such that $\ell + \abs{Z} \geq 1$.
    Let $\tau$ be a $k$-tangle in $G$ with $k \geq 1$.
\end{setting}

First we find our desired edge deep inside some rainbow such that the RC-decomposition is unaffected by its deletion:

\begin{lemma}\label{lemma:findedgeandRC}
    If we assume \cref{setting} with $M_0 \geq 6k$, then there is an edge $e$ of $G$ and an RC-decomposition $(R,\cW,Z,C)$ of $G$ which has even length $M \geq M_0/2 - k -2$ and adhesion $\ell$ such that $\tau$ does not live in the rainbow and $e$ has one endvertex in $W_{M/2}$ and one in $W_{M/2}$ or $Z$.
    Moreover, $(R,\cW,Z,C)$ is also an RC-decomposition of $G-e$ after deleting $e$ from $R$, if $e$ has both its endvertices in $W_{M/2}$.
\end{lemma}

\begin{proof}
    \cref{thm:ExistenceOfAShorterRainbowThatDoesNotContainTau} ensures that $G$ admits an RC-decomposition~$(R', \cW', Z, C')$ of length $M' \geq M_0/2 -k$ and adhesion~$\ell$ such that~$\tau$ does not live in the rainbow $R'$; by~\cref{prop:ShorteningTheRainbow}, we may assume without loss of generality that~$M'$ is even.%
    \COMMENT{The WLOG part is justified because $8k-6$ is even, so if $M'$ is odd, then $M' \geq 8k-5$.}
    To find the desired edge $e$, we merge the middle bags $W'_{M'/2 - 1}$, $W'_{M'/2}$, $W'_{M'/2+1}$  of $\cW'$ into one bag~$W_{M/2}$ in $\cW$ to make it robust against the deletion of an edge, and keep all other bags of~$\cW'$ in~$\cW^*$.
    It is immediate that this yields an RC-decomposition $(R,\cW, Z, C)$ in whose rainbow $\tau$ does not live, where $R := R', C := C'$; in particular, it has length $M = M'-2$ and adhesion~$\ell$.
    
    First, assume that~$|Z| \ge 1$.
    We fix an edge~$E$ of $G$ between~$Z$ and~$W_{M/2}$, which exists by~\cref{itm:RCDecompEdgesToZ}.
    Note that we still have~$Z \subseteq N_{G'}(W_{M/2})$:
    Since~$(R', \cW', Z, C')$ is an RC-decomposition, each of the~$Z \subseteq N_G(W'_{M'/2+i})$ for every $i \in \{-1,0,+1\}$ by \cref{itm:RCDecompEdgesToZ} and~$W'_{M'/2-1} \cap W'_{M'/2+1} = \emptyset$ by~\cref{itm:RainbowDecompDistinctBags}.
    Thus, each vertex of $Z$, in particular the endvertex of $e$ in $Z$, is joined to~$W_{M/2}$ by two distinct edges of~$G$ of which one persists in~$G-e$.
    Thus, it is now easy to see that $(R, \cW, Z, C)$ is also an RC-decomposition of $G-e$.
    
    Secondly, assume that~$Z = \emptyset$.
    Then $|Z| + \ell \geq 1$ ensures that~$\ell \ge 1$.
    By~\cref{itm:RainbowDecompLinkage}, there exists a~$U_{M/2}$--$U_{M/2+1}$ linkage~$\cP$ of cardinality~$\ell$ in~$G[W_{M/2}]$; note that $U'_{M'/2+i} \subseteq V(\bigcup \cP)$ for every $i \in \{-1,0,+1,+2\}$ by the construction of~$W_{M/2}$ and since~$(R', \cW', Z', C')$ also has adhesion~$\ell$.
    We aim to fix an edge~$e$ of $G[W_{M/2}] - E(\bigcup \cP)$ such that~$G[W_{M/2}] - e$ is still connected. 
    If this is possible, it is easy to check that $(R - e, \cW, Z, C)$ is an RC-decomposition of $G-e$, i.e.\,~$e$ is as desired.
    We claim that such an edge~$e$ exists: 
    
    Consider any vertex~$v \in U'_{M'/2}$, which exists as $\ell \geq 1$. 
    All its neighbours and itself are contained in~$W'_{M'/2 - 1} \cup W'_{M'/2} \subseteq W_{M/2}$, as $Z = \emptyset$ and by the definition of RC-decompositions, in particular by~\cref{itm:RainbowDecompDistinctBags}.
    Since $v$ has degree at least $3$ by assumption on $G$ and every vertex has at most~$2$ incident edges in the linkage~$\cP$, there is a neighbour~$w$ of~$v$ with $vw \notin E(\cP)$.
    If~$G[W^*_{M^*/2}] - vw$ is connected, $e := vw$ is as desired.
    So assume that $G[W_{M/2}] - vw$ is disconnected.
    We note that the linkage~$\cP$ is contained in one component of~$G[W_{M/2}] - vw$, since~$G[W'_{M'/2+1}] \subseteq G[W_{M/2}]$ is connected by~\cref{itm:RainbowDecompConnected}, meets every path in~$\cP$ as it contains $U'_{M'/2 +1}$ and also avoids $vw$ by the choice of $vw$.
    Thus, the component~$C$ of~$G[W_{M/2}] - vw$ containing~$w$ does not meet~$V(\cP)$, as~$v \in U'_{M'/2} \subseteq V(\cP)$.
    Now~$U_{M/2} \cup U_{M/2 - 1} = U'_{M'/2 + 2} \cup U'_{M'/2 -1} \subseteq V(\cP)$ is the separator of a separation of~$G$ with precisely~$W_{M/2}$ on one of its sides by~\cref{prop:SlicesOfRAreSepsOfG}.
    Thus, $C$ is not only a component of~$G[W_{M/2} - vw]$, but also a component of~$G - vw$, since $C$ avoids $\cP$.
    In particular, all vertices in~$C$ have degree at least~$2$ in $C$ as~$G$ has minimum degree~$3$.
    Hence, $C$ contains a cycle, and we fix~$e$ as an edge on this cycle.
    Then~$C-e$ is still connected which by the previous description implies that also~$G[W_{M/2}] - e$ is connected, as desired.
\end{proof}

We emphasise that in \cref{lemma:findedgeandRC} the possibly removed edge~$e$ from~$R$ is the only difference between the two RC-decompositions of~$G$ and~$G' := G-e$; in particular, all corresponding vertex sets such as~$V(R)$ and the~$V(R_{i,j})$ are the same, and a separation crosses or slices the rainbow in~$G'$ if and only if it does so in~$G$.

For the proof of~\cref{thm:AtLeastOneRainbowPath}, we have to construct an orientation~$\tau'$ of~$S_k(G')$ that is a~$k$-tangle in~$G'$ and extends~$\tau$.
Clearly, any such extension~$\tau'$ has to contain not only~$\tau$, but also all orientations of separations of~$S_k(G')$ that are forced by the request that~$\tau'$ shall again be a tangle and hence especially consistent.
More formally, recall that~$\tau$ \defn{forces} an orientation of a separation~$\{A,B\}$ of~$G'$ if there exists a separation~$(E,F) \in \tau$ such that~$(A,B) \leq (E,F)$ or~$(B,A) \leq (E,F)$ (cf.\ the proof of~\cref{lem:AnotherExtendingTangle}).
Let us also recall that $\tau$ forces an orientation of every separation in $S_k(G')$ that is also a separation of $G$.

Towards a suitable construction of~$\tau'$, we now prove that rainbow-crossing and rainbow-slicing separations are forced by~$\tau$ (\cref{lem:OrientationOfCrossingSeps} and~\cref{lem:OrientationOfSlicingSeps}, respectively).
This then allows us to show that an even broader class of separations are forced by~$\tau$ (\cref{lem:SepsWithForcedOrientations}).
For all these subsequent three lemmas, we assume the following setting:

\begin{setting}\label{nextsetting}
    Assume \cref{setting} with $M_0 \geq 18k$.
    Fix some edge $e$ of $G$ and some RC-decomposition $(R, \cW, Z, C)$ of $G$ of length $M$ which \cref{lemma:findedgeandRC} yields.
    In particular, $M \geq 8k$.
\end{setting}

\begin{lemma} \label{lem:OrientationOfCrossingSeps}
    Assume \cref{nextsetting}.
    If~$\{A,B\} \in S_k(G')$ crosses the rainbow, then there is~$(X,Y) \in \tau$ with~$(A,B) \leq (X,Y)$ or~$(B,A) \leq (X,Y)$ such that~$V(R_{2k-1, M-2k+1}) \subseteq X$.
	In particular, $\tau$ forces an orientation of~$\{A,B\}$.
\end{lemma}
\begin{proof}
    By possibly reversing the rainbow, we may assume that~$(A,B)$ crosses the rainbow clockwise.
    Let~$k' < k$ be the order of~$\{A,B\}$, and set~$i := i_{A,B} \in \{0, \dots, 2k'\}$ and~$j := j_{A,B} \in \{M-2k', \dots, M\}$. 
    By~\cref{prop:ObtainingNewSepsFromRainbowCrossingSeps}, $(A^{i+1}, B^{i+1}), (A^j, B^j)$ are again separations of~$G'$ of order at most~$k'$ which cross the rainbow clockwise.
    They are also separations of~$G$ by the choice of~$e$:
    The bag~$W_{M/2}$ is in~$B^{i+1} \cap A^j$, since~$i \leq 2k' \le M/2 \le M - 2k' \leq j$ because~$M \geq 4k$ and $k' \leq k-1$.
    As both $(A^{i+1}, B^{i+1}), (A^j, B^j)$ cross the rainbow, $Z$ is contained in each of their separators.
    Thus, both endvertices of~$e$ are in any of the cases contained in~$B^{i+1} \cap A^j$, as claimed.

    Since~$\tau$ does not live in the rainbow by the choice of~$(R,\cW, Z, C)$, the separation~$\{A, B\}$ of~$G$ is oriented by~$\tau$ monotonically over the rainbow.
    In particular, we have either~$(A^{i+1}, B^{i+1}), (A^j, B^j) \in \tau$ or~$(B^{i+1}, A^{i+1}), (B^j, A^j) \in \tau$.
    By definition of the $(A^h,B^h)$ and the choice of $i,j$, we have $(A^{i+1}, B^{i+1}) \le (A, B) \le (A^j, B^j)$ and~$V(R_{2k'+1,M-2k'-1}) \subseteq B^{i+1}, {A}^j$.
    So either~$({A}^j, {B}^j)$ or~$({B}^{i+1}, {A}^{i+1})$ can be chosen as the desired separation~$(X,Y) \in \tau$, as $k' < k$.%
    \COMMENT{Formally, we also need that~$R_{2k'+1,M-2k'+1} \supseteq R_{2k+1,M-2k+1}$.}
\end{proof}
	
\begin{lemma}\label{lem:OrientationOfSlicingSeps}
    Assume \cref{nextsetting}.
    If~$\{A,B\} \in S_k(G')$ slices the rainbow, then there is~$(X,Y) \in \tau$ with~$(A,B) \leq (X,Y)$ or~$(B,A) \leq (X,Y)$ such that~$V(R) \subseteq X$.
	In particular, $\tau$ forces an orientation of~$\{A,B\}$.
\end{lemma}
\begin{proof}
    As the separation $\{A,B\}$ of order $k' < k$ slices the rainbow, we may, by possibly renaming the sides of $\{A,B\}$, assume that there exists integers $i < h < j$ with $i \in \{0, \dots, 2k'\}$ and $j \in \{ M-2k', \dots, M\}$ with $W_i, W_j \subseteq A \setminus B$ and $W_h \subseteq B \setminus A$.
    If~$\{A, B\}$ is only a separation of~$G'$, but not a separation of~$G$, we let~$D$ be the component of~$G'[B \setminus A]$ that contains an endvertex of $e$; otherwise, we let~$D$ be the empty graph.
    Since $\{A,B\}$ slices the rainbow, we have~$Z \subseteq A \cap B$ and thus~$Z \cap V(D) = \emptyset$.
    So by the choice of the edge~$e$ that we deleted from~$G$ to get~$G'$, we observe that if~$V(D) \neq \emptyset$, then it meets~$W_{M/2}$.
    By~\cref{lem:SepsThatMeetRAndCContainAllBagsOnTheirSide} applied in~$G'$, we thus obtain that either $V(D) \subseteq  V(R) \setminus V(C)$ (which is in particular the case if~$V(D) = \emptyset$) or $D$ meets every bag of either $W_0, \dots, W _{M/2}$ or $W_{M/2}, \dots, W_{M}$. 
    But since $W_i, W_j \subseteq A \setminus B$ and $i \leq 2k' \leq M/2 \leq M -2k' \leq j$, the second case cannot occur; so we have~$V(D) \subseteq V(R)\setminus V(C)$.
    
    Further, \cref{lem:OrderOfSlicingSeps} yields~$2\ell + \abs{Z} \leq \abs{A,B} < k$ and therefore $\vert V(R) \cup Z, V(C) \vert = 2\ell + \abs{Z} < k$.
    Thus,~$\tau$ orients~$\{V(R) \cup Z, V(C)\}$, and we have~$(V(R) \cup Z, V(C)) \in \tau$ since~$\tau$ does not live in the rainbow by construction of the RC-decomposition $(R, \cW, Z, C)$.

    By the choice of~$D$, the pair $(A',B') := (A \cup V(D), B\setminus V(D))$ is a separation of $G$, which again has the same order~$< k$ as~$\{A,B\}$.
    Thus, $\tau$ contains an orientation of~$\{A', B'\}$.    
    If $(A', B') \in \tau$, we define
    \begin{align*}
	    (E, F)
            &:= (A',B') \vee (V(R) \cup Z, V(C))\\
            &= ((A \cup V(D)) \cup V(R), (B \setminus V(D)) \cap V(C)) \\
            &= (A \cup V(R), B \cap V(C)).
	\end{align*}
    If $(B',A') \in \tau$, we define
	\begin{align*}
	    (E, F)
            &:= (B',A') \vee (V(R) \cup Z, V(C))\\
            &= ((B\setminus V(D)) \cup V(R), (A \cup V(D)) \cap V(C)) \\
            &= (B \cup V(R), A \cap V(C)).
	\end{align*}
    We remark that in both definition the second equality holds due to $Z \subseteq A \cap B \setminus V(D)$ and the third holds due to $V(D) \subseteq V(R)\setminus V(C)$.
    We remark that the proof in the case of $(A', B') \in \tau$ is analogous to the one below in the case of $(B',A') \in \tau$ by swapping $A$ and $B$; thus, one obtains $(A,B) \leq (A \cup V(R), B \cap V(C) = (E,F) \in \tau$, if $(A',B') \in \tau$.
    
    So let us assume that $(B',A') \in \tau$. 
    The pair $(E,F)$ is again a separation of~$G$, since it is the supremum of two separations of~$G$. 
	We claim that $\{E, F\}$ has order~$<k$.
    Indeed, we have
    \begin{align*}
    	\vert E, F \vert
            &=  \vert (B \cup V(R)) \cap (A \cap V(C))\vert
                = \vert (B \cap A \cap V(C)) \cup (V(R) \cap A \cap V(C))\vert \\
    	    &= \vert B \cap A \cap (V(C)\setminus V( R))\vert + \vert V(R) \cap V(C) \cap A\vert 
    	        \leq \vert B \cap A \cap (V(C)\setminus V( R)) \vert + \vert V(R) \cap V(C)\vert \\
    	    &= \vert B \cap A \cap (V(C)\setminus V( R)) \vert + 2 \ell
    	        \leq \vert B \cap A \cap (V(C)\setminus V( R))\vert + \vert B \cap A \cap V(R)\vert \\
    	    &= \vert B \cap A\vert = \vert B, A \vert < k,
	\end{align*}
	where $2 \ell \le \vert B \cap A \cap V(R)\vert$ holds by~\cref{lem:OrderOfSlicingSeps}.
    Hence, $\tau$ has to orient $\{E, F\}$ and it does so as~$(E, F)$ by the profile property of the tangle~$\tau$, since~$(E, F)$ is the supremum of the two separations~$(B', A')$ and $(V(R) \cup Z, V(C))$, which are both contained in $\tau$.
    Thus,~$(B, A) \leq (B \cup V(R), A \cap V(C)) = (E, F) \in \tau$ and $V(R) \subseteq E$, so $(X,Y) := (E,F)$ is as desired.
\end{proof}

\begin{lemma}\label{lem:SepsWithForcedOrientations}
    Assume \cref{nextsetting}.
	Let~$\{A,B\} \in S_k(G') \setminus S_k(G)$, and let~$C_A \subseteq G'[A\setminus B]$ and~$C_B \subseteq G'[B\setminus A]$ be the two components of~$G' - (A\cap B)$ that contain an endvertex of~$e$. 
	If either both or none of~$C_A$ and~$C_B$ meet~$C$, then~$\tau$ forces an orientation of~$\{A,B\}$.
\end{lemma}
\begin{proof}
    Let $k' <k$ be the order of $\{A,B\}$.
    We first argue that we are in the case in which both endvertices of~$e$ are in~$W_{M/2}$.
    Suppose for a contradiciton that $e$ otherwise has one endvertex in~$W_{M/2}$ and the other one in~$Z$.
    Then one of~$C_A$ and~$C_B$, say~$C_A$, contains the endvertex~$z$ of $e$ in $Z$ and thus meets~$V(C) \supseteq Z$.
    Then $C_B$ meets~$C$ as well by assumption.
    Since $C_B$ also contains the other endvertex of $e$, i.e.\,the one in $W_{M/2}$, and $M \geq 4k'$,  \cref{lem:SepsThatMeetRAndCContainAllBagsOnTheirSide} yields that $C_B$ contains a bag $W_i$.
    By~\cref{itm:RCDecompEdgesToZ}, $G'$ contains an edge joining $z \in V(C_A) \subseteq A\setminus B$ and $W_i\subseteq V(C_B) \subseteq B\setminus A$, which is a contradiction.
    So we may assume by the choice of $e$ that both endvertices of~$e$ lie in~$W_{M/2}$; in particular, both~$C_A$ and~$C_B$ meet~$W_{M/2}$.
 
    The assumptions on~$C_A$ and~$C_B$ ensure that $\{A \cup V(C_B), B \setminus V(C_B)\}$ and $\{A \setminus V(C_A), B \cup V(C_A)\}$ are both separations of $G$.
    Also their order~$\abs{A \cap B}$ is $< k$; thus, $\tau$ orients them.
    If $\tau$ orients one of them as $(A \cup V(C_B), B \setminus V(C_B)) \geq (A,B)$ or $(B \cup V(C_A), A \setminus V(C_A)) \geq (B,A)$, then $\tau$ forces an orientation of~$\{A,B\}$, as desired. 
    So assume that $\tau$ orients them as $(B\setminus V(C_B), A \cup V(C_B))$ and$(A \setminus V(C_A), B \cup V(C_A))$. 
    Then their supremum~$(V(G) \setminus (V(C_A) \cup V(C_B)), (A \cap B) \cup V(C_A) \cup V(C_B))$ is also contained in~$\tau$, as its order is the same as $|A,B| = k' <k$ and the tangle $\tau$ has the profile property.
    Since $\tau$ does not live in $R$ by construction pf $(R,\cW,Z,C)$, it follows that $(V(C_A) \cup V(C_B)) \cap V(C) \neq \emptyset$. 
    As either both or none of~$C_A$ and~$C_B$ meets~$C$ by assumption, both~$C_A$ and~$C_B$ meet~$C$.
    
    All in all, both~$C_A$ and~$C_B$ meet~$C$ and~$W_{M/2}$.
    Hence, \cref{lem:SepsThatMeetRAndCContainAllBagsOnTheirSide} yields that each of $C_A$ and $C_B$ contains all but at most $2k'$ bags of $W_0, \dots, W_{M/2}$ or of $W_{M/2}, \dots, W_M$; in particular, they contain some $W_i$ and some $W_j$ respectively, as $M \geq 4k'$.
    Now \cref{lem:CrossingOrSlices} ensures that $\{A,B\}$ either crosses or slices the rainbow in~$G'$.
    Thus, the previous~\cref{lem:OrientationOfCrossingSeps,lem:OrientationOfSlicingSeps} guarantee taht $\tau$ forces an orientation of $\{A,B\}$.
\end{proof}

With these tools at hand, we are now ready to prove~\cref{thm:AtLeastOneRainbowPath}:

\begin{proof}[Proof of~\cref{thm:AtLeastOneRainbowPath}]
    We may assume \cref{nextsetting}, as \cref{setting} with $M_0 \geq 18k$ is precisely the premise of \cref{thm:AtLeastOneRainbowPath}.
    We claim that $\tau$ extends to a $k$-tangle in $G'$.
    For this, we begin by defining an orientation~$\tau'$ of~$S_k(G')$.
    So let~$\{A,B\}$ be an arbitrary separation in~$S_k(G')$.
    If $\{A,B\}$ is not a separation of $G$, then we let $C_A \subseteq G'[A \setminus B]$ and $C_B \subseteq G'[B \setminus A]$ be the components of $G- (A \cap B)$ which contain the respective endvertex of $e$.
    \begin{enumerate}[label=(\arabic*)]
        \item \label{itm:OrientationOfSepsThatAreForcedByTau}
            If~$\tau$ forces an orientation of~$\{A,B\}$, then we let~$(A,B) \in \tau'$ if and only if~$(A,B)$ is forced by~$\tau$.
        \item \label{itm:OrientationOfSepsThatAreNotForcedByTau}
            If~$\tau$ does not force an orientation of~$\{A,B\}$, then $\{A,B\} \notin S_k(G)$.
            If~$C_B$ meets $C$, then we let~$(A, B) \in \tau'$, and if~$C_A$ meets $C$, we let~$(B, A) \in \tau'$.
    \end{enumerate}
    Note that $\tau'$ contains at least one orientation of every separation in $S_k(G')$. Moreover, $\tau'$ is an orientation of $S_k(G')$: for a separation in case \cref{itm:OrientationOfSepsThatAreForcedByTau} $\tau$ forces at most one of its orientations due to the consistency of~$\tau$, and \cref{lem:SepsWithForcedOrientations} ensures that precisely one of $C_A$ and $C_B$ meet $C$ for a separation $\{A,B\}$ in case \cref{itm:OrientationOfSepsThatAreNotForcedByTau}.
    It is immediate form \cref{itm:OrientationOfSepsThatAreForcedByTau}that $\tau'$ contains~$\tau$; so $\tau$ extends to $\tau'$.
    It remains to show that~$\tau'$ is indeed a~$k$-tangle in~$G'$.
    
    Suppose for a contradiction that~$\tau'$ is not a tangle in~$G'$, i.e.\,the orientation $\tau'$ of $S_k(G')$ contains a forbidden triple~$\{(A'_i, B'_i) : i \in [3]\} \in \cT(G')$. 
	By the definition of~$\cT(G')$, we may assume without loss of generality that all the~$(A'_i, B'_i)$ are $\leq$-maximal in~$\tau'$.
	In the remainder of the proof, we will construct from the forbidden triple~$\{(A'_i,B'_i) : i \in [3]\} \subseteq \tau'$ a forbidden triple~$\{(A_i, B_i) : i \in [3]\} \subseteq \tau$, which then contradicts that~$\tau$ is a tangle in~$G$.
		
    To this end, we will make use of the~RC-decomposition~$(R, \cW, Z, C)$.
    We first show that one of the~$(A'_i, B'_i)$ contains~$V(R_{2k-1, M-2k+1}) \cup Z$ in its respective small side $A'_i$; in particular, the edge~$e$ is contained in~$G[A'_i]$ and thus it is a separation of $G$. 
    We then show for the other two separations~$(A_i', B_i')$ in the forbidden triple that either they are separations of~$G$, too, or the component~$C_{A'_i}$ of~$G[A'_i \setminus B'_i]$ containing an endvertex of~$e$ in fact contains also~$V(R_{2k-1,M-2k+1})$. 
    Moving the~$C_{A'_i}$ to the respective big side~$B'_i$ if necessary, these three separations of~$G$ will then yield a forbidden triple in~$\tau$.
 
    So we first show the following:
	\begin{sublemma}\label{sublem:ProofOfAtLeastOneRainbowPath-BigSep}
        If~$\{(A'_i, B'_i) : i \in [3]\} \in \cT(G')$ is contained in~$\tau'$ where each~$(A'_i, B'_i)$ is $\leq$-maximal in~$\tau'$, then some~$(A'_j, B'_j)$ is also a separation of~$G$ with~$V(R_{2k-1, M-2k+1}) \cup Z \subseteq A'_j$.
        In particular, $(A'_j, B'_j) \in \tau$.
	\end{sublemma}
    \begin{proof}
        For every~$(A_i', B'_i)$, at most~$2k-2$ bags of~$\cW$ are contained neither in~$A'_i \setminus B'_i$ nor in~$B'_i \setminus A'_i$ by~\cref{prop:SomePartLiesOnAStrictSide}.
        Thus, all but at most~$6k-6$ bags are contained in one of the strict sides of every~$(A'_i, B'_i)$. 
        As~$\{(A'_i, B'_i) : i \in [3]\}$ is a forbidden triple in~$G'$, we in particular have that~$R[A'_1]$, $R[A'_2]$ and~$R[A'_3]$ together cover the rainbow~$R$ in~$G'$.
        Thus, no such bag may be contained in every big side $B'_i$.
        Since~$M \geq 6k -6$, there thus must be some bag~$W_h$ of~$\cW$ that is contained in the strict small side~$A'_j \setminus B'_j$ of some~$(A'_j,B'_j)$; by renaming the $(A_i',B_i')$, we may assume that~$j=1$.
        Note that as~$W_h \subseteq A'_1 \setminus B'_1$ and~$Z \subseteq N_{G'}(W_h)$ by~\cref{itm:RCDecompEdgesToZ}, we have~$Z \subseteq A'_1$.
        We claim that~$(A'_1, B'_1)$ is as desired, i.e.\,it is also a separation of~$G$ and~$V(R_{2k-1,M-2k+1}) \subseteq A'_1$.

        Assume that~$B'_1 \setminus A'_1$ also contains some bag~$W_s$.
        By~\cref{lem:CrossingOrSlices}, $(A'_1, B'_1)$ then either crosses or slices the rainbow in~$G'$, and thus, the orientation~$(A'_1, B'_1)$ of~$\{A'_1, B'_1\}$ was forced by~$\tau$ by \cref{lem:OrientationOfCrossingSeps,lem:OrientationOfSlicingSeps}. 
        But if a separation is both $\leq$-maximal in~$\tau'$ and forced by~$\tau$, then it is also contained in~$\tau$, and thus a separation of~$G$.
        Hence, $(A'_1, B'_1)$ is a separation of~$G$ and also $\leq$-maximal in~$\tau$.
        Moreover, then~\cref{lem:OrientationOfCrossingSeps,lem:OrientationOfSlicingSeps} yield~$V(R_{2k-1, M-2k+1}) \subseteq A'_1$, as desired.

        So now assume that no bag of~$\cW$ is contained in~$B'_1 \setminus A'_1$.
        By~\cref{prop:SomePartLiesOnAStrictSide}, the strict side $A'_1 \setminus B'_1$ then contains all but at most~$2k-2$ many bags of~$\cW$.
        Suppose for a contradiction that~$\{A'_1, B'_1\}$ was not already a separation of~$G$.
        Since~$(A'_1, B'_1)$ is $\leq$-maximal in~$\tau'$, the tangle $\tau$ did thus not force an orientation of~$\{A'_1, B'_1\}$.
        Hence we have~$(A'_1, B'_1) \in \tau'$ due to~\cref{itm:OrientationOfSepsThatAreNotForcedByTau}.
        Hence, $C_{B'_1}$ meets $C$.
        Since we have seen above that~$Z \subseteq N_{G'}(W_m) \subseteq A'_1$, the choice of $e$ yields that the endvertex of~$e$ contained in~$C_{B'_1} \subseteq B'_1\setminus A'_1$ lies in~$W_{M/2}$. 
        Thus, \cref{lem:SepsThatMeetRAndCContainAllBagsOnTheirSide} yields that $C_{B'_1} \subseteq B'_1 \setminus A'_1$ contains some bag of $\cW$, which contradicts our assumption on $B'_1 \setminus A'_1$.
        Thus, $\{A_1', B_1'\}$ is also a separation of $G$.
        It remains to show~$V(R_{2k-1,M-2k+1}) \subseteq A'_1$.
        Since $(A'_1, B'_1)$ is a separation of $G$ and $\leq$-maximal in $\tau'$, it is also $\leq$-maximal in the tangle $\tau \subseteq \tau'$.
        Thus, $G[B'_1\setminus A'_1]$ is connected, and thus equal to $C_{B'_1}$.
        As~$C_{B'_1}$ meets~$C$ but~$G[B'_1 \setminus A'_1] = C_{B'_1}$ does not contain a bag of~$\cW$, \cref{lem:SepsThatMeetRAndCContainAllBagsOnTheirSide} yields that $C_{B'_1} = G[B'_1 \setminus A'_1]$ does meet at most the first and last $2k-2$ bags of~$\cW$.
        Thus, $V(R_{2k-1,M-2k+1}) \subseteq A'_1$, as desired.
    \end{proof}	  

    So by \cref{sublem:ProofOfAtLeastOneRainbowPath-BigSep}, there is some $j \in [3]$ such that $(A'_j, B'_j) \in \tau$ and $V(R_{2k-1, M-2k+1}) \cup Z \subseteq A'_j$; by symmetry we may assume $j = 1$.
    Next, we aim to obtain from the other two separations~$(A'_2, B'_2)$ and~$(A'_3, B'_3)$ of the forbidden triple two separations~$(A_2, B_2)$ and~$(A_3, B_3)$ of~$G$ that are contained in~$\tau$ and differ from~$(A'_2, B'_2)$ or~$(A'_3, B'_3)$, respectively, only in a subset of~$A'_1$. 
    This will then be the desired forbidden triple in~$\tau$.
    \begin{sublemma}\label{sublem:ProofOfAtLeastOneRainbowPath-MaxSeps}
        Every $\leq$-maximal separation~$(A, B)$ in~$\tau'$ is either also contained in~$\tau$ or we have $(A \setminus V(C_{A}), B \cup V(C_{A})) \in \tau$~and $V(C_{A}) \subseteq V(R_{2k-1, M-2k+1})$.
    \end{sublemma}
    \begin{proof}
        If the tangle~$\tau$ forces the orientation~$(A, B)$ of~$\{A, B\}$, then we already have~$(A, B) \in \tau$ by the $\leq$-maximality of~$(A, B)$ in~$\tau'$.
        Thus, we may assume that 
        $(A,B) \in \tau'$ due to~\cref{itm:OrientationOfSepsThatAreNotForcedByTau}, i.e.\,$C_A$ avoids $C$ and $C_B$ meets $C$.
        As each of $C_A$ and $C_B$ contains one endvertex of $e$, $\{A \setminus V(C_A), B \cup V(C_A)\}$ is a separation of~$G$.
        It also has separator~$A \cap B$, and thus is oriented by the~$k$-tangle~$\tau$ in~$G$.
        Since~$\tau$ does not force an orientation of~$\{A,B\}$, it follows from the consistency of the tangle $\tau$ that~$(B \cup V(C_A), A \setminus V(C_A)) \notin \tau$.
        Hence, $(A \setminus V(C_A), B \cup V(C_A)) \in \tau$.
        It remains to show~$V(C_A) \subseteq V(R_{2k-1, M-2k+1})$.
        
        First we claim that~$C_B$ contains a bag of~$\cW$:
        Since~$C_B$ contains an endvertex~$x$ of the edge~$e$ by definition, the choice of~$e$ ensures that either~$x \in W_{M/2}$ or~$x \in Z$.
        If $x \in W_{M/2}$, then \cref{lem:SepsThatMeetRAndCContainAllBagsOnTheirSide} yields that $C_B$ contains a bag of~$\cW$, as~$C_B$ meets~$C$ and~$M \ge 4k-4$.
        So assume that $x \in Z$.
        By~\cref{itm:RCDecompEdgesToZ}, $x$ has neighbours in every bag of~$\cW$.
        In particular, every bag of~$\cW$ is adjacent to~$V(C_B)$.
        At most $2k-2$ many bags of~$\cW$ meet $A \cap B$ by \cref{prop:SomePartLiesOnAStrictSide}.
        Thus, the component $C_B$ of $G' - (A \cap B)$ contains some bag of $\cW$, since $M \geq 2k-2$ all the $G'[W_i]$ are connected by~\cref{itm:RainbowDecompConnected}.

        Secondly, we claim that~$C_{A}$ does not contain a bag of~$\cW$.
        Suppose that $C_A$ contains a bag of~$\cW$.
        Then~$\{A, B\}$ would either cross or slice the rainbow by~\cref{lem:CrossingOrSlices}.
        Thus, $\tau$ forces an orientation of~$\{A,B\}$ by~\cref{lem:OrientationOfCrossingSeps,lem:OrientationOfSlicingSeps}, which contradicts our assumption on $\{A,B\}$.

        Hence, as all the $G'[W_i]$ is connected by~\cref{itm:RainbowDecompConnected},
        every bag which meets $C_A$ also meets $A \cap B$. 
        \cref{prop:SomePartLiesOnAStrictSide} shows that there are at most $2k-2$ such bags of $\cW$ .
        Since~$C_A$ contains an endvertex of~$e$ but avoids~$V(C) \supseteq Z$, it contains the endvertex of $e$ in~$W_{M/2}$.
        Hence, $C_A$ can only meet~$W_i$ with~$|M/2 - i| \le 2k-3$, as~$C_A$ is connected and~$\cW$ is a linear decomposition of~$R$.
        Thus, $M \geq 8k$ yields $V(C_A) \subseteq V(R_{2k-1, M-2k+1})$, as desired
    \end{proof}

    If a separation~$(A, B)$ of~$G'$ is also a separation of~$G$, let~$C_A$ be the empty graph.
    Recall that we have obtained earlier from \cref{sublem:ProofOfAtLeastOneRainbowPath-BigSep} 
    that~$(A'_1, B'_1) \in \tau$ and $V(R_{2k-1,M-2k+1}) \cup Z \subseteq A'_1$.
    By~\cref{sublem:ProofOfAtLeastOneRainbowPath-MaxSeps}, we now also have~$(A_i, B_i) := (A'_i \setminus V(C_{A'_i}), B'_i \cup V(C_{A'_i})) \in \tau$ for~$i = 2, 3$.
    
    We claim that the triple $\{(A'_1,B'_1), (A_2,B_2), (A_3, B_3)\} \subseteq \tau$ is a forbidden triple in~$G$.
    Since~$\{(A'_i, B'_i) : i \in \{1, 2, 3\}\}$ is a forbidden triple in~$G'$, we have
    \begin{equation*}
        G[A'_1] \cup G[A_2] \cup G[A_3] = G[A'_1] \cup G[A'_2 \setminus V(C_{A'_2})] \cup G[A'_3 \setminus V(C_{A'_3})] \supseteq G-e-C_{A'_2}-C_{A'_3}.
    \end{equation*}
    Since the endvertices of~$e$ are contained in~$W_{M/2} \cup Z$, and~$W_{M/2} \cup Z \subseteq V(R_{2k-1, M-2k+1}) \cup Z \subseteq A'_1$ as $M \geq 4k$, we have that 
    \begin{equation*}
        G[A'_1] \cup G[A_2] \cup G[A_3] \supseteq G-C_{A'_2}-C_{A'_3}.
    \end{equation*}
    		
    By~\cref{sublem:ProofOfAtLeastOneRainbowPath-MaxSeps}, each of~$V(C_{A'_2})$ and~$V(C_{A'_3})$ is either empty or contained in~$V(R_{2k-1,M-2k+1})$.
    So since~$V(R_{2k-1,M-2k+1}) \subseteq A'_1$, it follows that both~$C_{A'_2}$ and~$C_{A'_3}$ are subgraphs of~$G[A'_1]$ and hence~$G[A'_1] \cup G[A_2] \cup G[A_3] = G$.
    Thus, $\{(A'_1,B'_1), (A_2,B_2), (A_3,B_3)\}$ is a forbidden triple in~$\tau$, which contradicts that~$\tau$ is a~$k$-tangle in~$G$.
\end{proof}

\begin{proof}[Proof of \cref{main:SplitterTheorem}]
    Choose~$M(k)$ to be~$N(k, 18k)$ as in~\cref{thm:ExistenceOfRCDecomp}.
    Let $k \ge 1$ be an integer, let $G$ be a connected graph with at least $M(k)$~edges, and let $\tau$ be a $k$-tangle in $G$.
    We may assume $k \geq 3$; otherwise, we are done by \cref{lem:Smallk1,lem:Smallk2}.
    If there exists a $(k+1)$-tangle in $G$, then we are done by~\cref{thm:NoBigTangle}.
    Therefore, we may assume that there is no tangle in $G$ of order~$> k$.
    Then, the connected graph $G$ admits an RC-decomposition with sun $Z$ which has length $\geq 18k$ and adhesion $\ell$ such that $\abs{Z} + \ell \geq 1$ by \cref{thm:ExistenceOfRCDecomp}.
    If $G$ has a vertex of degree at most $2$, then we are done by~\cref{lemma:vertexofdegree1,lem:VertexOfDegreeTwo}.
    Thus, \cref{thm:AtLeastOneRainbowPath} concludes the proof.
\end{proof}

\noindent We remark that one may calculate that $M(k) \in O(3^{k^{k^5}})$.

\section{The inductive proof method and its applications} \label{sec:Deciders}

This section consists of three parts: we first collect all above auxiliary results to conclude the formal proof of our inductive proof method, \cref{thm:InductiveProofMethod}.
Next, we deduce from~\cref{thm:InductiveProofMethod} our reduction of~\cref{conj:Decider} to small graphs, \cref{mainresult:reduction}, and then derive \cref{main:MinimalCterex,main:BoundTotalWeight} from it.
Finally, we present a further application of~\cref{thm:InductiveProofMethod} in~\cref{thm:BoundedWitnessingSet}, which bounds the size of a subgraph `witnessing' a $k$-tangle.

Let us first prove our inductive proof method, which we restate here for the reader's convenience:
\begin{customthm}{Theorem~4} \label{thm:InductiveProofMethod:Proof}
    For every integer $k \geq 1$ there is some $M(k) \in O(3^{k^{k^5}})$ such that the following holds:
    Let $\tau$ be a $k$-tangle in a graph $G$. 
    Then there exists a sequence $G_0, \dots, G_m$ of graphs and $k$-tangles $\tau_i$ in $G_i$ for every $i \in \{0,\dots,m\}$ such that
    \begin{itemize}
        \item $G_0 = G$, $\tau_0 = \tau$;
        \item $G_i$ is obtained from $G_{i-1}$ by deleting an edge, suppressing a vertex, or taking a proper component;
        \item the $k$-tangle $\tau_{i-1}$ in $G_{i-1}$ survives as the $k$-tangle $\tau_i$ in $G_i$ for every $i \in [m]$;
        \item $G_m$ is connected and has less than $M(k)$ edges.
    \end{itemize}
\end{customthm}

\begin{proof}
    Let $M(k)$ be given by \cref{main:SplitterTheorem}.
    Suppose that $G_0, \dots, G_{i-1}$ and $\tau_0, \dots, \tau_{i-1}$ are already defined.
    If $G_{i-1}$ is disconnected, then we apply \cref{prop:CarryOverToComponent} to obtain $G_{i}$ and $\tau_i$.
    If $k \leq 2$, then we apply \cref{lem:Smallk1} or \cref{lem:Smallk2}.
    If $k \geq 3$, but $G_{i-1}$ has a vertex of degree $\leq 2$, then we apply \cref{lemma:vertexofdegree1} or \cref{lem:VertexOfDegreeTwo}.
    Thus, we may assume that $G_{i-1}$ is a connected graph with minimum degree $\geq 3$ and $k \geq 3$.
    In this case, we apply \cref{main:SplitterTheorem}, if $G_{i-1}$ has at least $M(k)$ edges.
    Otherwise, we set $m := i-1$, completing the proof.
\end{proof}

\subsection{Application I: sets and functions inducing tangles}

In this section, we address~\cref{conj:Decider}.
\begin{customthm}{Problem~1.1} \label{conj:Decider:copy}
    Is every tangle in a graph $G$ induced by some set~$X \subseteq V(G)$?
\end{customthm}
\noindent 
We will use \cref{thm:InductiveProofMethod:Proof} to prove \cref{mainresult:reduction}, our reduction of \cref{conj:Decider:copy} for $k$-tangles to graphs of size bounded in~$k$.
Let us briefly recall the relevant definitions, for which we mostly follow~\cite{Focus}.

Given a tangle~$\tau$ in a graph~$G$, a set~$X \subseteq V(G)$ \defn{induces}~$\tau$ if for every separation~$(A, B) \in \tau$, we have~$|X \cap A| < |X \cap B|$; in this case, we also say that $X$ \defn{induces} the orientation $(A,B)$ of the separation $\{A,B\}$.
As a natural relaxation, a \defn{weight function} on~$V(G)$ is a map~$w \colon V(G) \to \N$, and we say that it \defn{induces}~$\tau$ if $w(A) < w(B)$ for all~$(A, B) \in \tau$; in this case, we also say that $w$ \defn{induces} the orientation $(A,B)$ of the separation $\{A,B\}$. 
We remark that a set $X \subseteq V(G)$ induces a tangle~$\tau$ if and only if its indicator function $\mathds{1}_X$ induces $\tau$.
This allows us to focus on weight functions in what follows.

Recall that \cref{mainresult:reduction} reads as follows:

\begin{customthm}{Theorem~1} \label{mainresult:Reduction:Proof}
    For every integer $k \geq 1$, there exists $M = M(k) \in O(3^{k^{k^5}})$ such that for every $k$-tangle $\tau$ in a graph $G$, there exists a $k$-tangle $\tau'$ in a connected topological minor $G'$ of $G$ with less than $M$ edges such that
    if a weight function $w'$ on $V(G')$ induces the tangle $\tau'$, then the weight function $w$ on $V(G)$ which extends $w'$ by zero induces the tangle $\tau$.
    In particular, a set of vertices which induces $\tau'$ also induces~$\tau$.
\end{customthm}

For the proof of \cref{mainresult:Reduction:Proof} via \cref{thm:InductiveProofMethod:Proof}, we need to consider the following setting: 
Let $G'$ be a graph which arises from a graph $G$ by deleting an edge, suppressing a vertex or by passing to a component such that a tangle $\tau$ in $G$ survives as a tangle~$\tau'$ in $G'$.
We now aim to transfer a weight function inducing the tangle~$\tau'$ of~$G'$ to a weight function inducing the tangle~$\tau$ of~$G$.
The subsequent three lemmas show that the extension by zero always works.

\begin{lemma} \label{lem:DeciderTransfersDeletedEdge}
    If a $k$-tangle~$\tau$ in a graph~$G$ extends to a $k$-tangle~$\tau'$ in $G-e$ for an edge $e \in G$, then every weight function~$w$ on $V(G) = V(G')$ which induces~$\tau'$ also induces~$\tau$.
    In particular, a set of vertices which induces $\tau'$ also induces $\tau$.
\end{lemma}
\begin{proof}
    As $\tau$ extends to the tangle $\tau'$ in $G-e$, we have $\tau \subseteq \tau'$.
    Thus, $w$ induces~$\tau$ as well.
\end{proof}

\begin{lemma} \label{lem:DeciderTransfersContractedEdge}
    Let $\tau$ be a $k$-tangle in a graph $G$ with $k \geq 3$, and let $\tau'$ be the induced $k$-tangle in a graph $G' = G - v + xy$ obtained by suppressing a vertex $v$ with its two neighbours $x,y$ in $G$.
    If a weight function~$w'$ on~$V(G')$ induces~$\tau'$, then the weight function~$w$ on~$V(G)$ which extends $w'$ by zero induces~$\tau$.
    In particular, a set of vertices which induces $\tau'$ also induces $\tau$.
\end{lemma}

\begin{proof}
    Throughout this proof, we will use that, by definition of $\tau'$, a separation $(A ',B')$ of $G'$ is in $\tau'$ if and only if at least one of $(A' \cup \{v\}, B')$ or $(A', B' \cup \{v\})$ is in $\tau$.
    Let $(A,B) \in \tau$.
    Our aim is to find a separation $(A', B') \in \tau'$ such that $A \subseteq A' \cup \{v\}$ and $B' \subseteq B \cup \{v\}$, which then implies that $w(A) \leq w'(A') + w(v) = w'(A') < w'(B') = w'(B') + w(v) \leq w(B)$; so $w$ induces the orientation $(A,B)$ of~$\{A,B\}$ as desired.
    
    If $x,y \in A$ and $v \notin B$, then $(A',B') := (A \setminus \{v\}, B)$ is a separation of $G'$, and $(A,B) \in \tau$ witnesses $(A',B') \in \tau'$, as desired.
    Similarly, if $x,y \in B$ and $v \notin A$, we have $(A',B') := (A, B \setminus \{v\}) \in \tau'$, as desired. 
    
    Let us now assume that $x,y \in A$ and $v \in B$.
    Then $(A , B \setminus \{v\})$ is a separation of $G$.
    Since the only neighbours $x,y$ of $v$ are in $A$, the separations $(A,B)$ and $(B \setminus \{v\}, A)$ form a forbidden tuple in $G$.
    Hence, $(A , B \setminus \{v\})$ is in the tangle $\tau$.
    Now the above described case yields $(A',B') := (A \setminus \{v\}, B \setminus \{v\}) \in \tau'$, as desired.
    Similarly, if $u,w \in B$ and $v \in A$, then $(A',B') := (A \setminus \{v\}, B \setminus \{v\}) \in \tau'$.

    So to conclude the proof, we may assume that $x \in A \setminus B$ and $y \in B \setminus A$ by possibly renaming $x,y$; in particular $v \in A \cap B$, as $xv,vy \in E(G)$.
    Hence, $(C,D) := (A \cup \{y\}, B \setminus \{v\})$ is a separation of $G$.
    It suffices to show that $(C,D) \in \tau$ because the above described case for $x,y \in C$ and $v \notin D$ yields that $(A',B') := ((A \setminus \{v\}) \cup \{y\}, B \setminus \{v\}) = (C \setminus \{v\}, D) \in \tau'$, as desired.
    We now show that $(C,D) \in \tau$:
    The regularity of the tangle $\tau$ of order $k \geq 3$ yields that the separation $(\{v,y\}, V(G))$ is in $\tau$.
    Since the separation $(D,C)$ of $G$ together with $(A,B)$ and $(\{v,y\}, V(G))$ forms a forbidden triple in $G$, we have $(C,D) \in \tau$. 
\end{proof}

\begin{lemma} \label{lem:DeciderTransfersComponent}
    Let $\tau$ be a $k$-tangle~$\tau$ in a graph~$G$, and let $\tau'$ be its induced~$k$-tangle $\tau'$ in some component~$G'$ of~$G$.
    If a weight function~$w'$ on~$V(G')$ induces~$\tau'$, then the weight function~$w$ on~$V(G)$ which extends~$w'$ by zero induces~$\tau$.
    In particular, a set of vertices which induces~$\tau'$ also induces~$\tau$.
\end{lemma}
\begin{proof}
    Consider an arbitrary separation~$(A, B) \in \tau$.
    Then~$(A', B') := (A \cap V(G'), B \cap V(G')) \in \tau'$, as $\tau$ induces the tangle $\tau'$ in the component $G'$ of $G$.
    Since~$w'$ induces~$\tau'$ and due to the definition of~$w$, we have~$w(A) = w'(A') < w'(B') = w(B)$, as desired.
\end{proof}

\begin{proof}[Proof of \cref{mainresult:Reduction:Proof}]
    Let $\tau_m$ be the $k$-tangle in the graph $G_m$ with less than $M(k)$ edges as described in \cref{thm:InductiveProofMethod:Proof}.
    Let $w'$ be a weight function on $V(G')$ which induces the tangle $\tau_m$.
    Iteratively applying \cref{lem:DeciderTransfersDeletedEdge,lem:DeciderTransfersContractedEdge,lem:DeciderTransfersComponent} yields that the extension of the weight function $w'$ by zero induces the tangle $\tau_0 = \tau$ in the graph $G_0 = G$.
\end{proof}

As direct consequences of~\cref{mainresult:Reduction:Proof}, we now deduce~\cref{main:MinimalCterex:Proof} and~\cref{main:BoundTotalWeight:Proof}:

\begin{customthm}{Corollary~2} \label{main:MinimalCterex:Proof}
    For $k \geq 1$, there exists $M = M(k) \in O(3^{k^{k^5}})$ such that 
    \cref{conj:Decider:copy} holds for $k$ if it holds for all $k$-tangles in connected graphs $G$ with fewer than $M$ edges.
\end{customthm}

\begin{proof}
    This follows immediately from \cref{mainresult:Reduction:Proof}.
\end{proof}

\begin{customthm}{Corollary~3} \label{main:BoundTotalWeight:Proof}
    For every integer $k \geq 1$, there exists $K = K(k)$ such that for every $k$-tangle $\tau$ in a graph~$G$ there exists a weight function $V(G) \to \N$ which induces~$\tau$ and whose total weight $w(V(G))$ is bounded by $K$.
    In particular, the support of $w$ has size $\leq K$.

    Moreover, if~\cref{conj:Decider:copy} holds for~$k$, then every~$k$-tangle in a graph is induced by a set of at most~$M(k)$ vertices, where $M(k)$ is given by \cref{mainresult:reduction}.
\end{customthm}

\begin{proof}
    We enumerate all the finitely many non-isomorphic connected graphs $G_1, \dots, G_m$ with fewer than~$M$ edges. 
    As every finite graph has at most finitely many tangles, there are only finitely many $k$-tangles~$\tau$ in any such $G_i$.
    Elbracht, Kneip and Teegen showed with~\cref{thm:TanglesDecidedWeightedVertexSets} that every tangle in a graph is induced by a weight function, so we may fix for every such $k$-tangle $\tau$ a weight function $w_\tau$ which induces $\tau$.
    We then set $K(k)$ to be the maximum over all the total weights of these weight functions $w_\tau$.
    \cref{mainresult:Reduction:Proof} yields that every $k$-tangle in a graph $G$ is induced by some weight function which extends one of the weight functions~$w_\tau$ by zero, and thus has total weight $\leq K(k)$.

    The moreover-part follows immediately from \cref{mainresult:Reduction:Proof} by choosing the weight functions as indicator functions of the inducing sets given by the assumed positive answer to \cref{conj:Decider:copy}.
\end{proof} 

\noindent We remark that the proof of \cref{main:BoundTotalWeight:Proof} in fact shows that the support of the weight function inducing a $k$-tangle may actually be bounded by $M(k)$ as given in \cref{mainresult:Reduction:Proof}. 

\subsection{Application II: subgraphs witnessing a tangle}

In this section, we demonstrate another application of our inductive proof method \cref{thm:InductiveProofMethod:Proof} by bounding the size of a subgraph `witnessing' a tangle.
We say that a subgraph $H$ of a graph $G$ \defn{witnesses} that an orientation $\tau$ of $S_k(G)$ is a tangle if $H \not\subseteq \bigcup_{i=1}^3 G[A_i]$ for every three (not necessarily distinct) $(A_1, B_1), (A_2, B_2), (A_3, B_3) \in \tau$.
Indeed, $\tau$ is a tangle if and only if such a witnessing subgraph $H$ exists, since every tangle in $G$ is witnessed by $G$ itself.

Grohe and Schweitzer \cite{grohe2015isomorphism}*{Lemma 3.1}\footnote{We remark that triple covers in their paper are precisely the witnessing sets here. In fact, they proved a more general result about tangles on bipartitions in a more general setting. However, every $k$-tangle $\tau$ in $G$ induces a tangle $\tau'$ on the set of bipartitions of the edge set $E(G)$ of order $< k$: let $(C,D) \in \tau'$ if and only if there exists a separation $(A,B) \in \tau$ with $C \subseteq E(G[A])$ and $D \subseteq E(G[B])$. Then their result yields the described conclusion.} proved that every $k$-tangle is witnessed by a set of edges whose size can be bounded in $k$. However, their bound is defined recursively and yields a power tower of height~$k-1$. By \cref{thm:InductiveProofMethod:Proof}, we obtain a new bound which is significantly better for sufficiently large $k$:  

\begin{corollary} \label{thm:BoundedWitnessingSet}
    For every integer $k \geq 1$, there is an integer $M' = M'(k) \in O(3^{k^{k^5}})$ such that every $k$-tangle in a graph $G$ is witnessed by some subgraph $H$ of $G$ of size at most $M'$.
\end{corollary}

\begin{proof}[Proof of \cref{thm:BoundedWitnessingSet}]
    Let $M(k)$ be given by \cref{thm:InductiveProofMethod:Proof}, and set $M'(k) := 2 M(k)$.
    Let $\tau$ be a $k$-tangle in some graph $G$.
    \cref{thm:InductiveProofMethod} yields a $k$-tangle $\tau' := \tau_m$ in a topological minor $G' := G_m$ of $G$ which is connected and has fewer than $M(k)$ edges; in particular, $G'$ has at most~$M(k)$ vertices.
    Now the lift of $\tau'$ to $G$ is indeed $\tau$, as one checks by following the lifts along the inductive structure given by \cref{thm:InductiveProofMethod:Proof}.
    
    Let $H'$ be the subdivision of $G'$ in $G$, i.e.\ the subgraph of $G$ from which we obtain $G'$ by vertex sup\-pres\-sions.
    Consider the subgraph $H$ of $H'$ consisting of the branch vertices $V(G')$ together with precisely one edge of $H'$ per $V(G')$-path in $H'$.
    We may choose these edges such that they are incident with at least one branch vertex.
    Since $G'$ witnesses $\tau'$, one again checks along the inductive structure given by \cref{thm:InductiveProofMethod:Proof} that $H$ witnesses $\tau$.
    Note that $H$ has at most $|E(G')|$ edges and, as each edge of $H$ is incident to a branch vertex of $H'$, $H$ has at most $|V(G')| + |E(G')| \le M'(k)$ vertices.
\end{proof}

\section*{Acknowledgements}

The third named author gratefully acknowledges support by doctoral scholarships of the Studienstiftung des deutschen Volkes and the Cusanuswerk -- Bisch\"{o}fliche Studienf\"{o}rderung.
The fourth named author gratefully acknowledges support by a doctoral scholarship of the Studienstiftung des deutschen Volkes.

\bibliographystyle{amsplain}
\arXivOrNot{\bibliography{references_arxiv}}{\bibliography{references}}

\end{document}